\documentclass[12pt,leqno,oneside]{amsart}
\usepackage{amssymb}
\usepackage{mathrsfs}
\usepackage{graphicx,color}
\usepackage{newtxtext}
\usepackage[varg]{newtxmath}
\theoremstyle{plain} % Italic font
\newtheorem{theorem}{Theorem}[section]
\newtheorem{corollary}[theorem]{Corollary}
\newtheorem{proposition}[theorem]{Proposition}
\newtheorem{lemma}[theorem]{Lemma}
\theoremstyle{definition} % Change the Roman font

\newtheorem{remark}[theorem]{Remark}
\newtheorem{example}[theorem]{Example}

\DeclareMathOperator{\tr}{tr}

\newcommand{\R}{\mathbb{R}}
\newcommand{\Sp}{\mathbb{S}}

\usepackage{constants}
\newconstantfamily{eps}{symbol=\varepsilon}
\numberwithin{equation}{section}
\usepackage{bm}
\renewcommand{\vec}[1]{\bm{#1}}
\usepackage{hyperref}

\newcommand{\MisOriAngle}{\theta}

\title{Large time asymptotic behavior of grain boundaries motion with
dynamic lattice misorientations and with triple junctions drag}

\author{Yekaterina Epshteyn}
\address[Yekaterina Epshteyn]%
{Department of Mathematics,
The University of Utah,
Salt Lake City, UT 84112, USA}
\email{epshteyn@math.utah.edu}

\author{Chun Liu}
\address[Chun Liu]%
{Department of Applied Mathematics, Illinois Institute of Technology.
Chicago, IL 60616, USA}
\email{cliu124@iit.edu}

\author{Masashi Mizuno}
\address[Masashi Mizuno]%
{Department of Mathematics, College of Science
and Technology, Nihon University, Tokyo 101-8308 JAPAN}
\email{mizuno@math.cst.nihon-u.ac.jp}


\allowdisplaybreaks[1]
\begin{document}

%
% temporary setting: making line number. the follwing command should be
% deleted for final version
%

%\linenumbers

%
%
%

\begin{abstract}
Many technologically useful materials are polycrystals
composed of a myriad of small monocrystalline grains
separated by grain boundaries. Dynamics of grain boundaries play an
essential role in defining the materials properties across multiple
scales. In this work, we study the large time
asymptotic behavior of the model for the motion of grain
boundaries with the dynamic lattice misorientations and the
triple junctions drag.
%, proposed in \cite{Katya-Chun-Mzn}.
\end{abstract}

\maketitle

\section{Introduction}

\par  Many technologically useful materials are polycrystals composed of a myriad of small monocrystalline grains
separated by grain boundaries.  Dynamics of grain boundaries play an
essential role in defining the materials properties across multiple
scales. Experimental and computational studies give useful insight 
into the geometric features and the crystallography of the grain
boundary network in polycrystalline
microstructures.
\par The focus of this work is on the large time evolution of a planar grain boundary network. A
classical model for the motion of  grain
boundaries in polycrystalline materials is the growth by curvature, as the local
evolution law for the grain boundaries,  due to Mullins and Herring~\cite{doi:10.1007-978-3-642-59938-5_2,
doi:10.1063-1.1722511,doi:10.1063-1.1722742}.  In
addition, to have a well-posed model for the evolution of the grain
boundary network, one has to impose a separate condition at the triple
junctions where three grain boundaries meet \cite{MR1833000}.
%For grain boundaries, there are points called triple junctions at where thethree grain boundaries meet.
A conventional choice is the
Herring condition which is the natural boundary condition at the
triple points for the
grain boundary network at equilibrium, \cite{MR0485012,MR1240580,MR1833000,MR3612327}, and reference
therein. There are several analytical studies
about grain boundary motion by mean curvature with the Herring
condition at the triple junctions, see for instance
\cite{MR1833000,MR3495423,MR2076269,MR3565976,arXiv:1611.08254,MR2075985,
  DK:BEEEKT,DK:gbphysrev, MR2772123, MR3729587, BobKohn,
  barmak_grain_2013, MR3316603}, as well as computational work,
\cite{doi:10.1023-A:1008781611991,doi:10.1016-S1359-6454(01)00446-3, MR2772123, MR2748098, MR2573343,MR3787390,MR3729587}.
\par A standard assumption in the theory and simulations of the grain
growth is the evolution of the grain boundaries/interfaces themselves and not the
dynamics of the triple junctions. However, recent experimental work
indicates that the motion of the triple junctions together with the anisotropy of
the grain interfaces can have a significant effect on the resulting grain
growth \cite{barmak_grain_2013}, and see also a recent work on
dynamics of line
defects \cite{Zhang2017PhysRevLett, Zhang2018JMPS,Thomas2019PNAS}. 
The current work is a continuation of our previous work
\cite{Katya-Chun-Mzn}, where we proposed new model for the
evolution of planar grain boundaries, which takes into account dynamic
lattice misorientations (evolving anisotropy of grain boundaries) and
the mobility of the triple
junctions. The goal here is to analyze the large time
asymptotic behavior of the model proposed in \cite{Katya-Chun-Mzn}.
\par The paper is organized as follows. In
Sections~\ref{sec:1}-\ref{sec:3}, we discuss important details and
properties of the model for the grain boundary motion. In
Sections~\ref{sec:4}-\ref{sec:5}, we present the main
  results of this paper, the global existence
(Theorem \ref{thm:4.4}), and the large
time asymptotic behavior of the considered model under the assumption of a single triple junction (Theorem \ref{thm:5.1}). In
Section~\ref{sec:6}, we discuss the extension of the theory to the grain boundary network with multiple
junctions, under the assumption of no
  critical/disappearance events and around the energy minimizing state. Finally, in Section~\ref{sec:7}, we present several numerical
experiments to illustrate the effect of the dynamic
  orientations/misorientations (grains ``rotations'') and the effect of the
  mobility of the triple junctions on the grain growth.

%This article is focused on the analysis of large time
%asymptotics of the model of the grain boundary motion, proposed in
%\cite{Katya-Chun-Mzn}. The pioneering work on the mathematical
%analysis of the evolution of grain boundaries are due to
%Mullins~\cite{doi:10.1063-1.1722511,doi:10.1063-1.1722742} and
%Herring~\cite{doi:10.1007-978-3-642-59938-5_2}. General references
%about mathematical study of the evolution equation for curves
%and mean curvature flow are, for example,  in \cite{MR2024995,MR2815949}.

%In contrast the curvature
%effects on the evolution of the grain boundaries, these misorientation
%effects are long-range interactions for the grain boundaries and they
%are not well-studied. 
\section{Review of the Model}\label{sec:1}
In this article we consider the large time
asymptotic behavior of the model for the evolution of the planar grain
boundary network with the dynamic lattice misorientations and the
triple junctions drag. Thus, in this section for the reader's convenience,
we
first review the model which was originally
proposed in \cite{Katya-Chun-Mzn}, and then,  briefly preview main
results of the current work.

 Let us first recall the system for a single triple junction which was
 derived in
 \cite{Katya-Chun-Mzn}. The total grain boundary energy for such model is
 \begin{equation}
  \label{eq:2.3}
   \sum_{j=1}^3
   \sigma(\Delta^{(j)}\alpha)|\Gamma_t^{(j)}|.
 \end{equation}
 Here, $\sigma:\R\rightarrow\R$ is a given surface tension,
 $\alpha^{(j)}=\alpha^{(j)}(t):[0,\infty)\rightarrow\R$ is
 time-dependent orientations of the grains,
 $\theta=\Delta^{(j)}\alpha:=\alpha^{(j-1)}-\alpha^{(j)}$ is a lattice
 misorientation of the grain boundary $\Gamma^{(j)}_t$, and
 $|\Gamma_t^{(j)}|$ is the length of $\Gamma_t^{(j)}$.  As a result of
 applying the maximal dissipation principle, in \cite{Katya-Chun-Mzn},
 one can obtain the following model,
\begin{equation}
 \label{eq:2.20}
  \left\{
  \begin{aligned}
   v_n^{(j)}
   &=
   \mu
   \sigma(\Delta^{(j)}\alpha)
   \kappa^{(j)},\quad\text{on}\ \Gamma_t^{(j)},\ t>0,\quad j=1,2,3, \\
   \frac{d\alpha^{(j)}}{dt}
   &=
   -\gamma
   \Bigl(
   \sigma_\MisOriAngle(\Delta^{(j+1)}\alpha)
   |\Gamma_t^{(j+1)}|
   -
   \sigma_\MisOriAngle(\Delta^{(j)}\alpha)
   |\Gamma_t^{(j)}|
   \Bigr)
   ,\quad
   j=1,2,3,
   \\
   \frac{d\vec{a}}{dt}(t)
   &=
   \eta\sum_{k=1}^3
   \sigma(\Delta^{(k)}\alpha)
   \frac{\vec{b}^{(k)}(0,t)}{|\vec{b}^{(k)}(0,t)|},
   \quad t>0, \\
   \Gamma_t^{(j)}
   &:
   \vec{\xi}^{(j)}(s,t),\quad
   0\leq s\leq 1,\quad
   t>0,\quad
   j=1,2,3, \\
   \vec{a}(t)
   &=
   \vec{\xi}^{(1)}(0,t)
   =
   \vec{\xi}^{(2)}(0,t)
   =
   \vec{\xi}^{(3)}(0,t),
   \quad
   \text{and}
   \quad
   \vec{\xi}^{(j)}(1,t)=\vec{x}^{(j)},\quad
   j=1,2,3.
  \end{aligned}
  \right.
\end{equation}
In \eqref{eq:2.20}, $v_n^{(j)}$, $\kappa^{(j)}$, and
$\vec{b}^{(j)}=\vec{\xi}_s^{(j)}$ denote a normal velocity, a curvature and
a tangent vector of the grain boundary $\Gamma_t^{(j)}$, respectively. Note that $s$
is not an arc length parameter of $\Gamma_t^{(j)}$, namely, $\vec{b}^{(j)}$
is \emph{not} necessarily a unit tangent
vector. The vector $\vec{a}=\vec{a}(t):[0,\infty)\rightarrow\R^2$ denotes a position of
the triple junction, $\vec{x}^{(j)}$ is a position of the end point of
the grain boundary. The three independent relaxation time scales
$\mu,\gamma,\eta>0$  (length, misorientation and position of the
triple junction) are considered as
positive constants. Further, we assume in \eqref{eq:2.20}, $\alpha^{(0)}=\alpha^{(3)}$,
$\alpha^{(4)}=\alpha^{(1)}$ and $\vec{b}^{(4)}=\vec{b}^{(1)}$, for
simplicity. We also use notation $|\cdot|$ for a standard
euclidean vector norm. The complete details about model
\eqref{eq:2.20} can be found in the earlier work
\cite[Section 2]{Katya-Chun-Mzn}.
%%% Continue Below Aug 12!!!
Next, in \cite{Katya-Chun-Mzn},  we relaxed the curvature effect,  by taking the limit
$\mu\rightarrow\infty$, and obtained the reduced model, 
\begin{equation}
 \label{eq:1.1}
 \left\{
  \begin{aligned}
   \frac{d\alpha^{(j)}}{dt}
   &=
   -
   \Bigl(
   \sigma_\MisOriAngle(\Delta^{(j+1)}\alpha)|\vec{b}^{(j+1)}|
   -
   \sigma_\MisOriAngle(\Delta^{(j)}\alpha)|\vec{b}^{(j)}|
   \Bigr)
   ,
   \quad
   j=1,2,3, \\
   \frac{d\vec{a}}{dt}(t)
   &=
   \eta
   \sum_{j=1}^3
   \sigma(\Delta^{(j)}\alpha)
   \frac{\vec{b}^{(j)}}{|\vec{b}^{(j)}|},
   \quad t>0, \\
   \vec{a}(t)+\vec{b}^{(j)}(t)
   &=
   \vec{x}^{(j)},
   \quad
   j=1,2,3.
  \end{aligned}
 \right.
\end{equation}
 In (\ref{eq:1.1}), we consider $\vec{b}^{(j)}(t)$ as a grain
 boundary. Note that, the system of equations \eqref{eq:1.1} can also be
 derived from the energetic variational principle for the total grain boundary
 energy (\ref{eq:2.3}) (with $|\Gamma_t^{(j)}|$ replaced by $|\vec{b}^{(j)}|$).
%\begin{equation}
 %\label{eq:1.2}
 %E
 %=
 % \sum_{j=1}^3
 % \textcolor{blue}{\sigma(\Delta^{(j)}\alpha)}
 %|\vec{b}^{(j)}|.
%\end{equation}
%
%

%We assume that the surface tension $\sigma$ is independent
%of the normal vector $\vec{n}$. 
Hereafter, we assume the following three conditions for the
surface tension $\sigma$. First, we assume $\sigma$ is $C^3$, positive
and is minimized at $0$, namely,
% positivity, namely, there exists a positive constant
%$\Cl{const:2.Assumption1}>0$ such that,
\begin{equation}
 \label{eq:2.Assumption1}
  \sigma(\MisOriAngle)\geq \sigma(0)>0,
\end{equation}
for $\MisOriAngle\in\R$. Second, we assume convexity, for all
$\MisOriAngle\in\R,$
\begin{equation}
 \label{eq:2.Assumption2}
  \sigma_\MisOriAngle(\MisOriAngle)\MisOriAngle\geq0,\qquad
  \text{and}
  \quad 
  \sigma_{\MisOriAngle\MisOriAngle}(0)>0.  
\end{equation}
Furthermore, we assume,
\begin{equation}
 \label{eq:2.Assumption3}
  \sigma_\MisOriAngle(\MisOriAngle)=0
  \ \text{if and only if}\
  \MisOriAngle=0.
\end{equation}

%\textcolor{blue}{We later
%review the derivation of the system \eqref{eq:1.1} in Section
%\ref{sec:2}.}

\begin{remark}
 Note, it is enough to assume
 $\sigma_{\MisOriAngle\MisOriAngle}(0)>0$ instead of
 \eqref{eq:2.Assumption2} if $\Delta^{(j)}\alpha$ are sufficiently
 small. Indeed, if $C^2$ function $\sigma$ satisfies
 $\sigma_{\MisOriAngle\MisOriAngle}(0)>0$ and
 $\sigma(\theta)\geq\sigma(0)$ for $\theta\in\R$, then
 $\sigma_\MisOriAngle(\MisOriAngle)\MisOriAngle\geq0$ is satisfied for
 sufficiently small $|\MisOriAngle|$. Note that, the convexity
 assumption \eqref{eq:2.Assumption2} was not used for the local
 existence result, see \cite{Katya-Chun-Mzn} (or see Proposition
 \ref{prop:3.3} below). Further, we can also show the maximum
 principle (Proposition \ref{prop:3.5}) without global convexity
 assumption $\sigma_\MisOriAngle(\MisOriAngle)\MisOriAngle\geq0$
 in \eqref{eq:2.Assumption2}, if the initial orientations
 $\alpha_0^{(j)}$ are sufficiently small. 
%Thus if the initial
 %orientation $\alpha_0^{(j)}$ is sufficient small, then the maximal
 %principle (Proposition \ref{prop:3.5}) can be deduced.
\end{remark}

 \begin{figure}
  \centering
  \includegraphics[height=6cm]{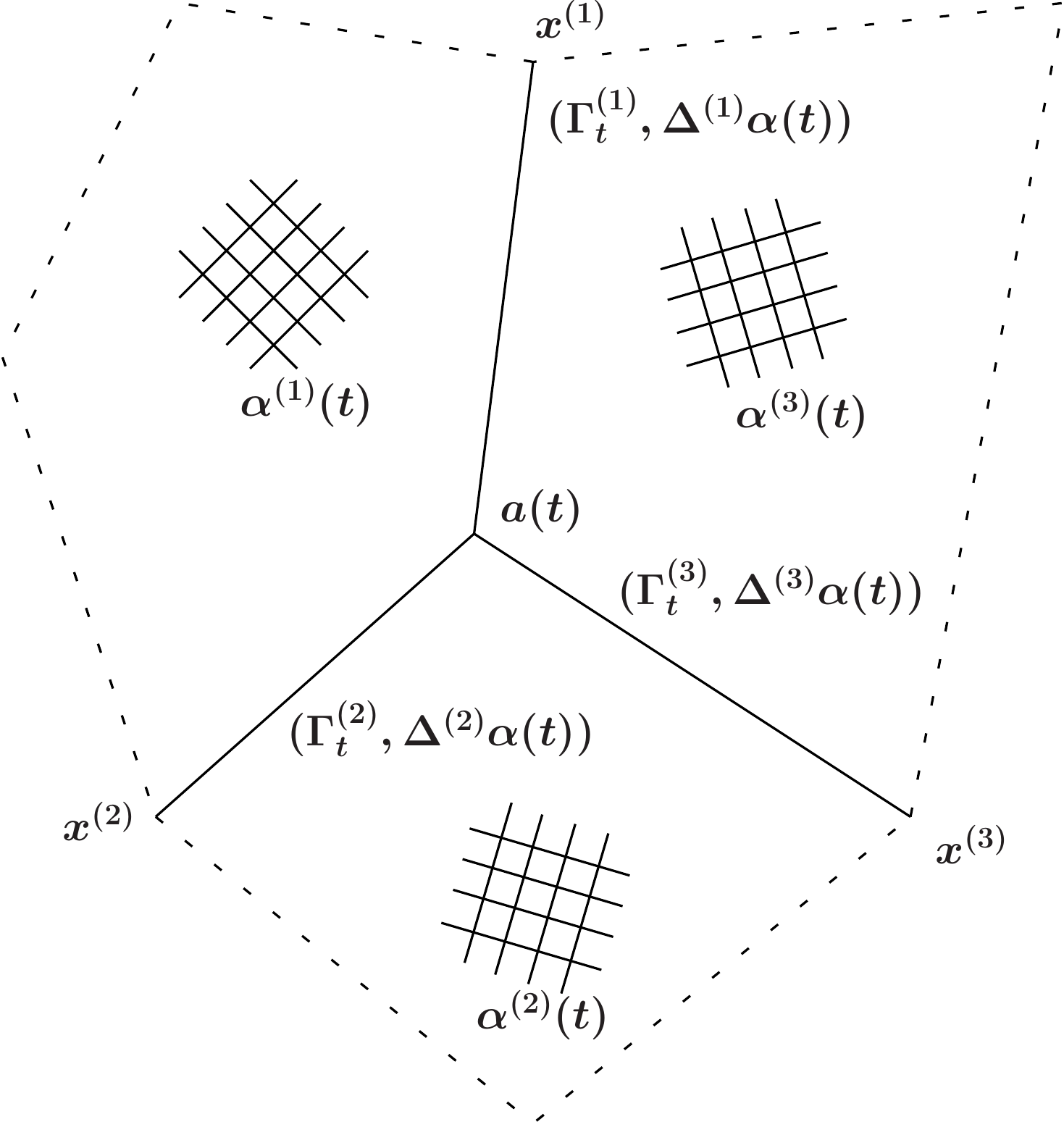}\qquad\qquad
   \includegraphics[height=6.05cm]{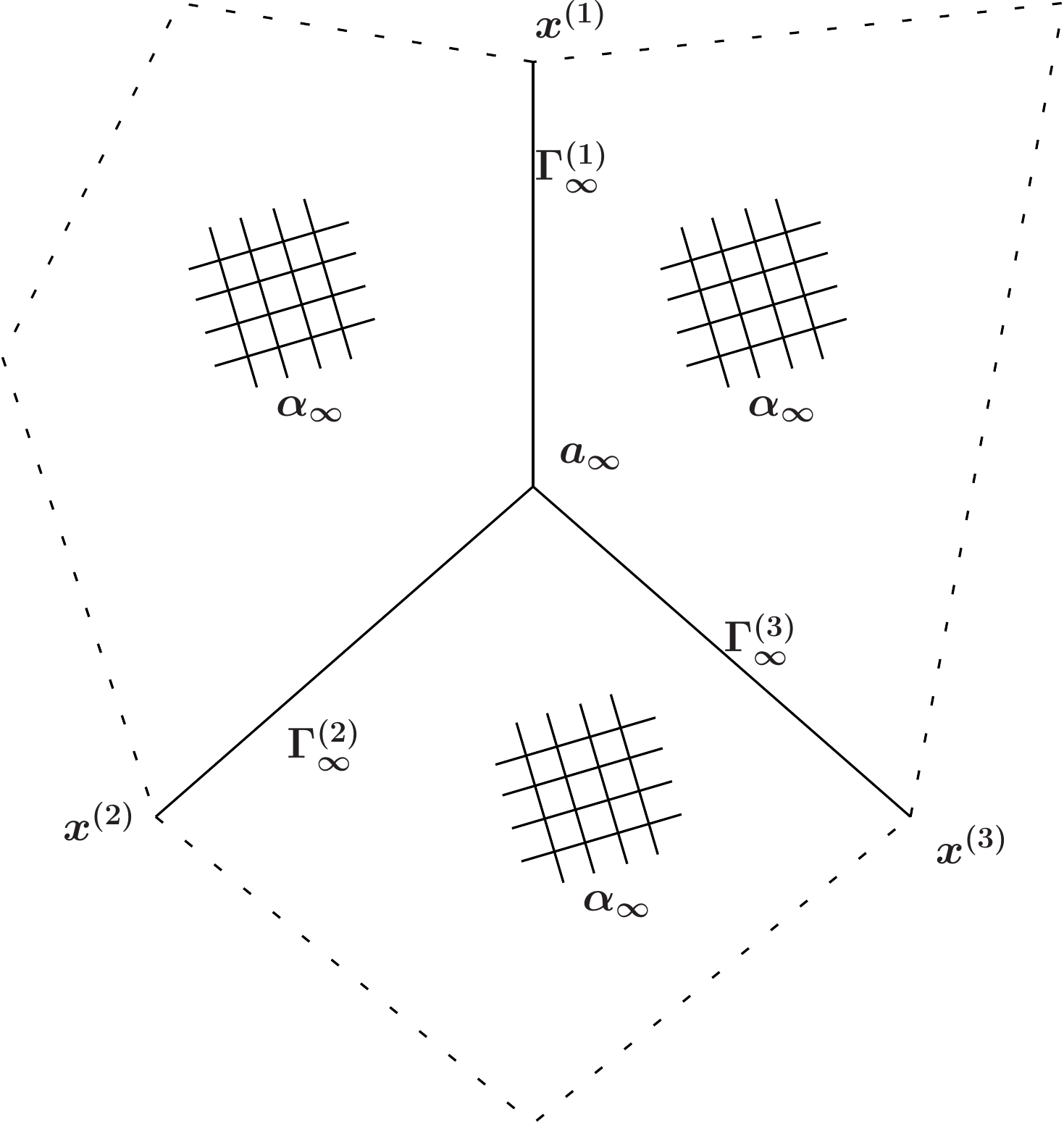}
  \caption{The left figure illustrates the model \eqref{eq:1.1}, where we
    isolate the effect of the misorientations and mobility of the
    triple junction on the motion of the grain boundaries. The right
    figure illustrates the model of the equilibrium state
  \eqref{eq:1.5} with no misorientation effect.}
  \label{fig:1.1}
 \end{figure}

Next, to preview the ideas of main results of this work, consider the equilibrium state of the grain
boundary energy \eqref{eq:2.3}, namely,
\begin{equation}
 \label{eq:1.3}
  \left\{
   \begin{aligned}
   0
   &
   =
    -
   \left(
   \sigma_\theta(\Delta^{(j+1)}\alpha_{\infty})
   |\vec{b}^{(j+1)}_{\infty}|
   -
   \sigma_\theta(\Delta^{(j)}\alpha_{\infty})
   |\vec{b}^{(j)}_{\infty}|
   \right)
   ,
   \\
   \vec{0}
   &=
   \sum_{j=1}^3
    \sigma(\Delta^{(j)}\alpha_{\infty})
   \frac{\vec{b}^{(j)}_{\infty}}{|\vec{b}^{(j)}_{\infty}|}, \\
   \vec{a}_{\infty}
   &=
   \vec{x}^{(1)}-\vec{b}^{(1)}_{\infty}
   =
   \vec{x}^{(2)}-\vec{b}^{(2)}_{\infty}
   =
   \vec{x}^{(3)}-\vec{b}^{(3)}_{\infty}.
   \end{aligned}
  \right.
\end{equation}
Assume,  for each $i=1,2,3$,
\begin{equation}
 \label{eq:1.4}
  \left|
   \sum_{j=1, j\neq i}^3
   \frac{\vec{x}^{(j)}-\vec{x}^{(i)}}{|\vec{x}^{(j)}-\vec{x}^{(i)}|}
  \right|
  >1.
\end{equation}
The assumption \eqref{eq:1.4} implies that fixed points $\vec{x}^{(1)}, \vec{x}^{(2)}$ and $\vec{x}^{(3)}$ can not belong to the single line. Furthermore,
\eqref{eq:1.4} is equivalent to the condition that in the triangle with vertices
$\vec{x}^{(1)}\vec{x}^{(2)}\vec{x}^{(3)}$, all three angles are less than
$\frac{2\pi}{3}$. Next, from the assumption \eqref{eq:1.4}, 
\eqref{eq:2.Assumption2}-\eqref{eq:2.Assumption3}, 
associated
equilibrium system \eqref{eq:1.3} becomes,
\begin{equation}
 \label{eq:1.5}
  \left\{
   \begin{aligned}
    \sum_{j=1}^3
    \frac{\vec{b}_\infty^{(j)}}{|\vec{b}_\infty^{(j)}|}&=\vec{0},
    \\
    \vec{a}_\infty+\vec{b}^{(j)}_\infty
    &=
    \vec{x}^{(j)},
    \quad
    j=1,2,3.
   \end{aligned}
  \right.
\end{equation}
In fact, assumptions \eqref{eq:2.Assumption2}-\eqref{eq:2.Assumption3} imply
$\alpha^{(1)}_\infty=\alpha^{(3)}_\infty=\alpha^{(3)}_\infty$, hence
$\Delta^{(j)}\alpha_\infty=0$ for $j=1,2,3$. We later discuss the
property of the equilibrium solution of (\ref{eq:1.3}) in Proposition \ref{prop:5.1}.

 %\begin{figure}
 % \centering
 % \includegraphics[height=7cm]{c17p02-04.eps}
  %
%  \caption{The figure illustrates the model of the equilibrium state
%  \eqref{eq:1.5}. There are no misorientation effects on the grain
%  boundaries. A solution of \eqref{eq:1.5} is unique under the
%  assumption \eqref{eq:1.4}, and it minimizes the total length of
 % $\Gamma^{(j)}_\infty$.}
  %
%  \label{fig:1.2}
% \end{figure}
The main result of this work is the local exponential stability for
the solution of the equilibrium state \eqref{eq:1.5}, Theorem
\ref{thm:5.1}. That is, if the initial misorientations are sufficiently
small and the position of the initial triple junction is sufficiently close to the
equilibrium state of the position of the triple junction, then, the solution of
\eqref{eq:1.1} exists globally in time and it exponentially converges to
the equilibrium solution of \eqref{eq:1.5}. Our strategy of the
proof here is to show a priori estimate for the position of the triple
junction, and then study the linearized problem of \eqref{eq:1.1} around
the equilibrium state. With the aid of the assumption \eqref{eq:1.4},
the equilibrium state system \eqref{eq:1.5} is uniquely
solvable. Moreover, the equilibrium state is also the energy minimizing
state. Thus, we can obtain a priori estimate for the position of the
triple junction and a full convergence result for large time asymptotics
of the solution. Again thanks to \eqref{eq:1.4}, the linearized operator
of \eqref{eq:1.1} is degenerate if and only if, there are no
misorientation effects, that is, all $\alpha^{(j)}$ are the same.
This allows us to deduce the exponential stability for the equilibrium state.

Moreover, we consider large time asymptotic behavior of the grain boundary
network. In general, the uniqueness for the equilibrium state is not
known, and there might be critical events (disappearance of the grains,
grain boundaries,
etc. \cite{DK:BEEEKT,DK:gbphysrev,MR2772123,MR3729587}). We study the
large time asymptotics of the grain
boundary network in Section \ref{sec:6} under the assumption of no critical events and around the energy minimizing state. We discuss global
existence and large time asymptotic behavior of the grain boundary
network in Section \ref{sec:6}.

\section{Properties of the Local
    Solution}\label{sec:3}
In this section, for the reader's convenience, we review some known
results, as well as established additional properties for the system,
defined in \eqref{eq:1.1}. In particular, we review local existence and
a priori estimates results for the model \eqref{eq:1.1}. More details
can be found in \cite{Katya-Chun-Mzn}.

%Let $\vec{x}^{(j)}\in\R^2$, $\vec{\alpha}_0\in\R^3$, and
%$\vec{a}_0\in\R^2$ be given initial data and we consider the initial
%value problem of \eqref{eq:1.1}, namely
%\begin{equation}
% \label{eq:3.1}
% \left\{
%  \begin{aligned}
%   \frac{d\vec{\alpha}}{dt}
%   &=
%   -
%   \mathbb{B}(t)\vec{\alpha},
%   \quad t>0, \quad
%   \vec{\alpha}(t)
%   =
%   \mathstrut^t
%   \Bigl(
%   \alpha^{(1)}(t),\
%   \alpha^{(2)}(t),\
%   \alpha^{(3)}(t)
%   \Bigr),
%   \\
%   \mathbb{B}(t)
%   &=
%   \begin{pmatrix}
%    |\vec{b}^{(1)}(t)|+|\vec{b}^{(2)}(t)|
%    &
%    -|\vec{b}^{(2)}(t)|
%    &
%    -|\vec{b}^{(1)}(t)| \\
%    -|\vec{b}^{(2)}(t)| &
%    |\vec{b}^{(2)}(t)|+|\vec{b}^{(3)}(t)| &
%    -|\vec{b}^{(3)}(t)|\\
%    -|\vec{b}^{(1)}(t)| &
%    -|\vec{b}^{(3)}(t)| &
%    |\vec{b}^{(3)}(t)|+|\vec{b}^{(1)}(t)|
%   \end{pmatrix},
%   \\
%   \frac{d\vec{a}}{dt}
%   &=
%   \sum_{j=1}^3
%   \left(
%   1+\frac12\Bigl(\alpha^{(j-1)}(t)-\alpha^{(j)}(t)\Bigr)^2
%   \right)
%   \frac{\vec{b}^{(j)}}{|\vec{b}^{(j)}|},
%   \quad t>0,
%   \\
%   \vec{a}(t)+\vec{b}^{(j)}(t)
%   &=
%   \vec{x}^{(j)},\quad
%   t>0,
%   \quad
%   j=1,2,3, \\
%   \vec{\alpha}(0)&=\vec{\alpha}_0,\quad
%   \vec{a}(0)=\vec{a}_0.
%  \end{aligned}
%  \right.
%\end{equation}
%
%Now, we are ready to state local existence and uniqueness result.

First,  using the same argument as in \cite{Katya-Chun-Mzn},
one can show the local in time existence of the triple junction and the
estimates of the maximal existence time,  under assumption of more general relaxation
time constants,
$\gamma,\ \eta>0$.

\begin{proposition}[Local existence]
 \label{prop:3.3}
 Let $\vec{x}^{(1)}$, $\vec{x}^{(2)}$, $\vec{x}^{(3)}\in \R^2$,
 $\vec{a}_0\in\R^2$, and $\vec{\alpha}_0\in\R^3$ be given initial
 data.  Assume the conditions \eqref{eq:2.Assumption1} and
 \eqref{eq:1.4} for $i=1,2,3$,  and let $\vec{a}_\infty$ be a solution of
 \eqref{eq:1.5}. Further, assume that for all $j=1,2,3$,
 \begin{equation}
  \label{eq:3.4}
   |\vec{a}_0-\vec{a}_\infty|<\frac12|\vec{b}^{(j)}_\infty|.
 \end{equation}
 Then, there exists a local in time solution $(\vec{\alpha}, \vec{a})$
 of \eqref{eq:1.1} on $[0,T_{\text{max}})$, such that
 \begin{equation}
  \label{eq:3.10}
  |\vec{a}(t)-\vec{a}_\infty|<|\vec{b}^{(j)}_\infty|
   \quad
   \text{for all}\quad
   j=1,2,3,\ \text{and}\ 
   0\leq t<T_{\text{max}}.
 \end{equation}
 Furthermore, the maximal existence time $T_{\text{max}}$ of the
 solution is estimated by
% \par {\color{blue} Masashi, below is not the analog of the final
%  version of the estimate for the maximal existence time as we have in the
%  short time paper, see estimate (4.6). Why?}
 \begin{equation}
  \label{eq:3.11}
    T_{\text{max}}
    \geq
    \min
    \left\{
     \frac{|\vec{\alpha}_0|}{4\gamma(M_1+8M_2|\vec{\alpha}_0|)\sum_{j=1}^3|\vec{b}^{(j)}_\infty|},\
     \frac{|\vec{a}_0-\vec{a}_\infty|}{3\eta M_0}
     ,\
     \frac{1}{12\gamma M_1}
     ,\
    \frac{1}{8\eta M_0\sum_{j=1}^3\frac{1}{|\vec{b}_\infty^{(j)}|-2|\vec{a}_0-\vec{a}_\infty|}}
    \right\}
 \end{equation}
 where 
\begin{equation*}
 M_0:=\sup_{|\MisOriAngle|
  \leq
  4|\vec{\alpha}_0|}|\sigma(\MisOriAngle)|,\quad
  M_1:=\sup_{|\MisOriAngle|\leq
  4|\vec{\alpha}_0|}|\sigma_\MisOriAngle(\MisOriAngle)|,\quad
  M_2
  :=
  \sup_{|\MisOriAngle_1|, |\MisOriAngle_2|\leq 4|\vec{\alpha}_0|}
  \frac{|\sigma_\MisOriAngle(\MisOriAngle_1)-\sigma_\MisOriAngle(\MisOriAngle_2)|}
  {|\MisOriAngle_1-\MisOriAngle_2|}.
\end{equation*}
\end{proposition}

To show Proposition \ref{prop:3.3}, the contraction mapping
principle is employed \cite{Katya-Chun-Mzn}.  Note that, the
assumption \eqref{eq:3.10} guarantees that the point vector $\vec{a}$ does not coincide
with the
Dirichlet point $\vec{x}^{(j)}$, and stays as the triple
junction. Note also,  if one can obtain a priori bounds for
$|\vec{\alpha}|$ and $|\vec{a}-\vec{a}_\infty|$, 
then from \eqref{eq:3.11}
the solution given by
Proposition \ref{prop:3.3} can be extended globally in time.

Next we review some a priori estimates for \eqref{eq:1.1}. Since our
problem \eqref{eq:1.1} ensures the energy dissipation principle,
%\textcolor{blue}{(Masashi; it will be deleted) \eqref{eq:2.13}}, 
one can obtain,
\begin{proposition}
 [Energy dissipation{\cite[Proposition 5.1]{Katya-Chun-Mzn}}]
 \label{prop:3.4}
 Let $(\vec{\alpha},\vec{a})$ be a solution of \eqref{eq:1.1} on $0\leq
 t\leq T$, and let $E(t)$, given by \eqref{eq:2.3},  be
 the total grain boundary energy of the system. Then, for all $0<t\leq T$,
 \begin{equation}
  \label{eq:3.6}
   E(t)
   +
   \frac1\gamma
   \int_0^t
   \left|
   \frac{d\vec{\alpha}}{dt}(\tau)
   \right|^2
   \,d\tau
   +
   \frac1\eta
   \int_0^t
   \left|
    \frac{d\vec{a}}{dt}(\tau)
	  \right|^2
   \,d\tau \\
  =
   E(0).
 \end{equation}
\end{proposition}
The estimate \eqref{eq:3.6} is obtained by taking the
derivative of the energy $E(t)$ and using the system \eqref{eq:1.1}.

We also have a maximum principle for the orientation $\alpha^{(j)}$,
\begin{proposition}
 [Maximum principle{\cite[Proposition 5.1]{Katya-Chun-Mzn}}]
 \label{prop:3.5}
 Let $(\vec{\alpha},\vec{a})$ be a solution of \eqref{eq:1.1} on $0\leq
 t\leq T$. Then, for all $0<t\leq T$,
 \begin{equation}
  \label{eq:3.7}
   |\vec{\alpha}(t)|^2
%   +
%   2\int_0^t(\mathbb{B}(\tau)
%   \vec{\alpha}(\tau)\cdot\vec{\alpha}(\tau))\,d\tau
   \leq
   |\vec{\alpha}_0|^2.
 \end{equation}
%In particular, the maximum principle $|\vec{\alpha}(t)| \leq
% |\vec{\alpha}_0|$ holds.
\end{proposition}
The estimate \eqref{eq:3.7} can be obtained by multiplying the equation
of $\alpha^{(j)}$ in \eqref{eq:1.1} by $\alpha^{(j)}$ and integrating in
time. Since, the right hand side of the first equation
in \eqref{eq:1.1} is non-positive definite for all $t>0$, the maximum
principle for the orientations $\vec{\alpha}$ holds.

%Lastly,  we show that the sum of the orientations is also preserved.
In addition, since $\frac{d}{dt}(\alpha^{(1)}+\alpha^{(2)}+\alpha^{(3)})=0$, we
obtain that the sum of the orientations is preserved, namely,
\begin{lemma}
 [Preserving total orientations]
 \label{lem:3.6}
 Let $(\vec{\alpha},\vec{a})$ be a solution of \eqref{eq:1.1} on $0\leq
 t\leq T$. Then, for all $0<t\leq T$,
 \begin{equation}
  \alpha^{(1)}(t)+\alpha^{(2)}(t)+\alpha^{(3)}(t)
   =\alpha_0^{(1)}+\alpha_0^{(2)}+\alpha_0^{(3)}.
 \end{equation}
\end{lemma}

%\begin{proof}
% Since,  matrix $\mathbb{B}$ is symmetric and
% $\mathbb{B}(1,1,1)=\vec{0}$, we obtain that,
% \[
% \frac{d}{dt}
% (\alpha^{(1)}+\alpha^{(2)}+\alpha^{(3)})
% =
% \left(\vec{\alpha}_t\cdot(1,1,1)\right)
% =
% -\left(\mathbb{B}\vec{\alpha}\cdot(1,1,1)\right)
% =
%  -\left(\vec{\alpha}\cdot\mathbb{B}(1,1,1)\right)
% =
% 0.
% \]
%\end{proof}

Above, we discussed some known results about the system
\eqref{eq:1.1} (cf. \cite{Katya-Chun-Mzn}). In the current
work, we employ these results to show the existence of the global
solution, and obtain the large time asymptotic behavior of the
solution to
\eqref{eq:1.1}.

%%\begin{lemma}
%% \label{lem:2.3}
%% If $c_1$, $c_2$, $c_3>0$, then the kernel of the matrix $\mathbb{C}$
%% given by \eqref{eq:2.7} is spanned by one vector .
%%\end{lemma}
%
%
%
\section{Global existence} \label{sec:4}
%\par {\textcolor{blue} {This is Katya, August 13 2020: Continue Below}}\\
In this section, 
%\textcolor{blue}{the following(DELETE)}
existence theory for a
global in time solution of \eqref{eq:1.1} is presented.  
%
%As mentioned earlier, to show the global existence we need to derive a
%priori estimates for $|\vec{\alpha}|$ and $|\vec{a}-\vec{a}_\infty|$. We
%already have a priori estimates for $|\vec{\alpha}|$, see Proposition
%\ref{prop:3.5} (the maximum principle).  Therefore, the main objective
%here is to obtain a priori estimates for $|\vec{a}-\vec{a}_\infty|$.
As mentioned after the Proposition \ref{prop:3.3} in
  Section \ref{sec:3}, in order to show the global existence in time
  of \eqref{eq:1.1}, we need to derive a priori estimates for
  $|\vec{\alpha}|$ and $|\vec{a}-\vec{a}_\infty|$. We already have a
priori estimates for $|\vec{\alpha}|$, see Proposition \ref{prop:3.5}
(the maximum principle).  Therefore, the main objectives of this section are to
obtain a priori estimates for $|\vec{a}-\vec{a}_\infty|$, Lemma
\ref{lem:4.1}, and then
show the global existence result,  Theorem \ref{thm:4.4}.

Define,
\begin{equation}
 \label{eq:4.1}
  \Cl{const:4.1}
  :=
  \inf
  \left\{
   \sum_{j=1}^3|\vec{x}^{(j)}-\vec{a}|
   :
   \text{There exists}\ 
   j=1,2,3\
    \text{such that}\ 
    |\vec{a}-\vec{a}_\infty|\geq\frac12|\vec{b}_\infty^{(j)}|
  \right\}.
\end{equation}
Since $\vec{a}_\infty$ is the unique minimizer of
\begin{equation}
 \label{eq:4.2}
  f(\vec{a})
  =
  \sum_{j=1}^{3}|\vec{x}^{(j)}-\vec{a}|,
  \quad
  \vec{a}\in\R^2
\end{equation}
and $f:\R^2\rightarrow\R$ is continuous, we have,
\begin{equation}
 \label{eq:4.3}
  0<
 \sum_{j=1}^3|\vec{b}^{(j)}_\infty|
 =
 f(\vec{a}_\infty)
 <
 \Cr{const:4.1}.
\end{equation}

%Here, global in time existence theorem for the grain boundary network under the assumption of a single triple junction is presented.

Next lemma gives  a priori estimate for the triple junction $\vec{a}$.

\begin{lemma}
 [Boundedness of the triple junction]
 \label{lem:4.1}
 Assume that an initial data $(\vec{\alpha}_0,\vec{a}_0)$ satisfies,
\begin{equation}
  \label{eq:4.4}
  E(0)= \sum_{j=1}^3
  \sigma(\Delta^{(j)}\alpha_0)
   |\vec{a}_0-\vec{x}^{(j)}|
   <
   \sigma(0)
   \Cr{const:4.1}.
 \end{equation}
% \begin{equation}
%  \label{eq:4.4}
%  E(0)= \sum_{j=1}^3
%  \textcolor{blue}
%  {\sigma(\Delta^{(j)}\alpha_0)}
%   |\vec{a}_0-\vec{x}^{(j)}|
%   <
%   \textcolor{blue}{\sigma(0)}
%   \Cr{const:4.1}
% \end{equation}
 Let $(\vec{\alpha},\vec{a})$ be a solution of \eqref{eq:1.1} on $0\leq
 t\leq T$. Then, we have that,
 \begin{equation}
  \label{eq:4.5}
   |\vec{a}(t)-\vec{a}_\infty|
   <
   \frac12|\vec{b}^{(j)}_\infty|
   =\frac{1}{2}|\vec{a}_\infty-\vec{x}^{(j)}|
 \end{equation}
 for $j=1,2,3$, and for any $0\leq t\leq T$.
\end{lemma}

\begin{proof}
 Assume, there is $0\leq t_1\leq T$ and $j=1,2,3$, such that,
 $|\vec{a}(t_1)-\vec{a}_\infty| \geq \frac12|\vec{b}^{(j)}_\infty|$.
 % $j=1,2,3$ . 
 Then \eqref{eq:2.Assumption1}, \eqref{eq:4.1} and
 Proposition \ref{prop:3.4} (energy dissipation) lead,
 \begin{equation*}
  \Cr{const:4.1}
   \leq
   \sum_{j=1}^3|\vec{a}(t_1)-\vec{x}^{(j)}|
   =
   \sum_{j=1}^3|\vec{b}^{(j)}(t_1)|
   \leq
   \frac1{\sigma(0)}
   \sum_{j=1}^3
   \sigma(\Delta^{(j)}\alpha_0)
   |\vec{a}_0-\vec{x}^{(j)}|
   =\frac{1}{\sigma(0)}E(0)
   ,
 \end{equation*}
 which contradicts \eqref{eq:4.4}.
\end{proof} 
\begin{remark}
 Note, if we assume \eqref{eq:4.4}, then we have,
 \[
 \sum_{j=1}^3|\vec{a}_0-\vec{x}^{(j)}|
 \leq
 \frac1{\sigma(0)}
 \sum_{j=1}^3
 \sigma(\Delta^{(j)}\alpha_0)
 |\vec{a}_0-\vec{x}^{(j)}|
 <\Cr{const:4.1},
 \]
 thus, $|\vec{a}_0-\vec{a}_\infty|<|\vec{b}_\infty^{(j)}|/2$ for all $j=1,2,3$.
 Hence,  the assumption \eqref{eq:3.4} in Proposition
   \ref{prop:3.3} (existence of the local in time solution)  will be automatically deduced.
\end{remark}

\begin{remark}
 Assumption \eqref{eq:4.4} is related to the smallness assumption for the initial data
 $(\vec{\alpha}_0,\vec{a}_0)$. Namely, if the initial misorientations are
 sufficiently small and 
 %the initial triple junction point 
the position of the initial triple junction
 is
 sufficiently close to
 %the equilibrium triple junction point, 
the position of the equilibrium triple junction $\vec{a}_\infty$,
 then we
 obtain \eqref{eq:4.4}.
%
% \textcolor{blue}{Masashi: I will give explanation that there is an
% initial data satisfying \eqref{eq:4.4} with an extra assumption
% $\sigma(0)=\sigma(0)$}
\end{remark}

Now we are in position to show
%establish 
the global existence of the solution of
\eqref{eq:1.1}.
\begin{theorem}
 [Global existence]
 \label{thm:4.4}
 Let $\vec{x}^{(1)}$, $\vec{x}^{(2)}$, $\vec{x}^{(3)}\in \R^2$,
 $\vec{a}_0\in\R^2$, and $\vec{\alpha}_0\in\R^3$ be the initial data
 for the system \eqref{eq:1.1}. Assume
 \eqref{eq:1.4},  and let $\vec{a}_\infty$ be a unique
 solution of the equilibrium system \eqref{eq:1.5}. Further, assume
condition (\ref{eq:4.4}).
 Then there exists a unique global in time solution $(\vec{\alpha},
 \vec{a})$ of \eqref{eq:1.1}.
\end{theorem}

%\begin{theorem}
% [Global existence]
% \label{thm:4.4}
% %
% Let $\vec{x}^{(1)}$, $\vec{x}^{(2)}$, $\vec{x}^{(3)}\in \R^2$,
% $\vec{a}_0\in\R^2$, and $\vec{\alpha}_0\in\R^3$ be the initial data
% for the system \eqref{eq:1.1}. Assume
% \eqref{eq:1.4},  and let $\vec{a}_\infty$ be a unique
% solution of the equilibrium system \eqref{eq:1.5}. Further,  assume
% condition \eqref{eq:4.4}. Then there exists a unique global in time
% solution $(\vec{\alpha}, \vec{a})$ of
% \eqref{eq:1.1}.
%\end{theorem}

\begin{proof}
 [Proof of Theorem \ref{thm:4.4}]
 We need to show that the solution given by Proposition \ref{prop:3.3}
 extends globally in time. Let $(\vec{\alpha},\vec{a})$ be a solution of
 \eqref{eq:1.1} on $0\leq t\leq T$. By Lemma \ref{lem:4.1}, we obtain
 $|\vec{a}(T)-\vec{a}_\infty|<\frac12|\vec{b}_\infty^{(j)}|$. Due to Proposition
 \ref{prop:3.4} (energy dissipation), we also have,
 \begin{equation*}
  E(T)
   =
   \sum_{j=1}^3
   \sigma(\Delta^{(j)}\alpha(T))
   |\vec{a}(T)-\vec{x}^{(j)}|
   \leq
   E(0)
   <
   \sigma(0)
   \Cr{const:4.1}.
 \end{equation*}
In addition,  from Proposition \ref{prop:3.5} (maximum principle),  we have that
 $|\vec{\alpha}(T)|\leq|\vec{\alpha}_0|$,  hence we can extend the
 solution globally in time.
\end{proof}

In the above proof, a key argument is how to obtain the a priori
estimate for the position of the triple junction
$\vec{a}$, Lemma \ref{lem:4.1}. An energy smallness condition
\eqref{eq:4.4} plays an important role to obtain the a priori estimate
for the solution of \eqref{eq:1.1}.

\section{Large time asymptotic behavior}\label{sec:5}
In this section, large time behavior of the global solution given by
Theorem \ref{thm:4.4} is presented. We first discuss
in Proposition \ref{prop:5.1} below large time asymptotic
profile  of the solution of \eqref{eq:1.1}. After
that, we show in Theorem \ref{thm:5.1} that the asymptotic profile is asymptotically exponentially
stable.

\begin{proposition}
 \label{prop:5.1}
 Let $\vec{x}^{(1)}$, $\vec{x}^{(2)}$, $\vec{x}^{(3)}\in \R^2$,
 $\vec{a}_0\in\R^2$, and $\vec{\alpha}_0\in\R^3$ be the initial data
 for the system \eqref{eq:1.1}. We assume that the
 initial data satisfy \eqref{eq:4.4},  and we also impose
 the same assumptions as in Theorem \ref{thm:4.4}. Define $\alpha_\infty$ as,
 \begin{equation}
  \label{eq:5.1}
  \alpha_\infty
   :=
   \frac{\alpha_0^{(1)}+\alpha_0^{(2)}+\alpha_0^{(3)}}{3}.
 \end{equation} 
Let $\vec{a}_\infty$ be a solution of the equilibrium system
 \eqref{eq:1.5} and $(\vec{\alpha}, \vec{a})$ be a time global solution of
 \eqref{eq:1.1}. Then,
 \begin{equation}
  \label{eq:5.2}
   \vec{\alpha}(t)\rightarrow\alpha_\infty(1,1,1),\quad
   \vec{a}(t)\rightarrow\vec{a}_\infty,
 \end{equation}
 as $t\rightarrow\infty$.
\end{proposition}

\begin{proof}
 Consider arbitrary time sequence $t_k\rightarrow\infty$. From 
 Proposition \ref{prop:3.5} (maximum principle), Lemma \ref{lem:4.1}
 (boundedness of the triple junction), and Proposition \ref{prop:3.4}
 (energy dissipation), we have a convergent subsequence (denoted by
 the same $t_k\rightarrow\infty$) such that,
 \begin{equation}
  \label{eq:5.3}
  \begin{aligned}
   \vec{\alpha}(t_k)&\rightarrow \vec{\alpha}_{\infty,*},\quad &
   \vec{a}(t_k)&\rightarrow \vec{a}_{\infty,*}, \\
   \left|
   \frac{d\vec{\alpha}}{dt}(t_k)
   \right|
   &\rightarrow0,\quad &
   \left|
   \frac{d\vec{a}}{dt}(t_k)
   \right|
   &\rightarrow0,
  \end{aligned}
 \end{equation}
 and,
 \begin{equation}
  \label{eq:5.4}
   |\vec{a}_{\infty,*}-\vec{a}_\infty|
   \leq
   \frac12|\vec{a}_\infty-\vec{x}^{(j)}|,
 \end{equation}
 for some $\vec{\alpha}_{\infty,*}\in \R^3$ and
 $\vec{a}_{\infty,*}\in\R^2$. We have to show
% $\vec{\alpha}_{\infty,*}=\alpha_\infty$ 
 $\vec{\alpha}_{\infty,*}=\alpha_\infty (1,1,1)$
 and
 $\vec{a}_{\infty,*}=\vec{a}_\infty$. Taking the same limit with
 respect to $t_k$
 in the equation \eqref{eq:1.1}, from \eqref{eq:5.3} we obtain,
 \begin{equation}
  \label{eq:5.5}
   \left\{
  \begin{aligned}
%   \vec{0}
%   &=
%   -
%   \mathbb{B}_{\infty,*}\vec{\alpha}_{\infty,*},
   0
   &
   =
   -
   \left(
   \sigma_\theta(\Delta^{(j+1)}\alpha_{\infty,*})
   |\vec{b}^{(j+1)}_{\infty,*}|
   -
   \sigma_\theta(\Delta^{(j)}\alpha_{\infty,*})
   |\vec{b}^{(j)}_{\infty,*}|
   \right)
   ,
%   \quad
%   \vec{\alpha}_{\infty,*}
%   =
%   \mathstrut^t\Big(
%   \alpha^{(1)}_{\infty,*},\
%   \alpha^{(2)}_{\infty,*},\
%   \alpha^{(3)}_{\infty,*}\Big),
%   \\
%   \mathbb{B}_{\infty,*}
%   &=
%   \begin{pmatrix}
%    |\vec{b}^{(1)}_{\infty,*}|+|\vec{b}^{(2)}_{\infty,*}|
%    &
%    -|\vec{b}^{(2)}_{\infty,*}|
%    &
%    -|\vec{b}^{(1)}_{\infty,*}| \\
%    -|\vec{b}^{(2)}_{\infty,*}| &
%    |\vec{b}^{(2)}_{\infty,*}|+|\vec{b}^{(3)}_{\infty,*}| &
%    -|\vec{b}^{(3)}_{\infty,*}|\\
%    -|\vec{b}^{(1)}_{\infty,*}| &
%    -|\vec{b}^{(3)}_{\infty,*}| &
%    |\vec{b}^{(3)}_{\infty,*}|+|\vec{b}^{(1)}_{\infty,*}|
%   \end{pmatrix},
   \\
   \vec{0}
   &=
   \sum_{j=1}^3
   \sigma(\Delta^{(j)}\alpha_{\infty,*})
%   \left(
%   1+\frac12\Bigl(\alpha^{(j-1)}_{\infty,*}-\alpha^{(j)}_{\infty,*}\Bigr)^2
%   \right)
   \frac{\vec{b}^{(j)}_{\infty,*}}{|\vec{b}^{(j)}_{\infty,*}|}, \\
   \vec{a}_{\infty,*}
   &=
   \vec{x}^{(1)}-\vec{b}^{(1)}_{\infty,*}
   =
   \vec{x}^{(2)}-\vec{b}^{(2)}_{\infty,*}
   =
   \vec{x}^{(3)}-\vec{b}^{(3)}_{\infty,*}.
  \end{aligned}
  \right.
 \end{equation}
 %  The kernel of $\mathbb{B}_{\infty,*}$ is spanned by
 % $\mathstrut^t(1,1,1)$. 
 We will show that $\Delta^{(j)}\alpha_{\infty,*}=0$
 for $j=1,2,3$. First, by \eqref{eq:5.4}, we have that,
 \[
 |\vec{b}_{\infty,*}^{(j)}|
  =
  |\vec{x}^{(j)}-\vec{a}_{\infty,*}|
  \geq
  |\vec{x}^{(j)}-\vec{a}_{\infty}|
  -
  |\vec{a}_{\infty}-\vec{a}_{\infty,*}|
  \geq
  \frac12|\vec{x}^{(j)}-\vec{a}_{\infty}|>0.
 \]
 Next, from the first equation of \eqref{eq:5.5}, we obtain
 \begin{equation}
  \label{eq:5.40}
   \sigma_\theta(\Delta^{(1)}\alpha_{\infty,*})
   |\vec{b}^{(1)}_{\infty,*}|
   =
   \sigma_\theta(\Delta^{(2)}\alpha_{\infty,*})
   |\vec{b}^{(2)}_{\infty,*}|
   =
   \sigma_\theta(\Delta^{(3)}\alpha_{\infty,*})
   |\vec{b}^{(3)}_{\infty,*}|.
 \end{equation}
 Multiplying \eqref{eq:5.40} by $\Delta^{(2)}\alpha_{\infty,*}$ and
 $\Delta^{(3)}\alpha_{\infty,*}$, and using the convexity assumption
 \eqref{eq:2.Assumption2}, we have
 \begin{equation*}
  \sigma_\theta(\Delta^{(1)}\alpha_{\infty,*})
   (\Delta^{(2)}\alpha_{\infty,*})
   |\vec{b}^{(1)}_{\infty,*}|\geq 0,\qquad
   \sigma_\theta(\Delta^{(1)}\alpha_{\infty,*})
   (\Delta^{(3)}\alpha_{\infty,*})
   |\vec{b}^{(1)}_{\infty,*}|\geq 0.
 \end{equation*}
 Thus $\Delta^{(1)}\alpha_{\infty,*}=0$ follows from
 $-\sigma_\theta(\Delta^{(1)}\alpha_{\infty,*})
 (\Delta^{(1)}\alpha_{\infty,*}) |\vec{b}^{(1)}_{\infty,*}|\geq 0$
 (obtained from the above),
 the assumptions \eqref{eq:2.Assumption2} and \eqref{eq:2.Assumption3}.
 We can obtain $\Delta^{(2)}\alpha_{\infty,*}=\Delta^{(3)}\alpha_{\infty,*}=0$ in a similar way. 

 From Lemma \ref{lem:3.6}, \eqref{eq:5.1}, and
 $\Delta^{(j)}\alpha_{\infty,*}=0$ for $j=1,2,3$, we find,
 $\alpha^{(j)}_{\infty,*}=\alpha_\infty$ for all $j=1,2,3$.
% Thus, due to Lemma \ref{lem:3.1}, the kernel of $\mathbb{B}_{\infty,*}$ is spanned by
% $\mathstrut^t(1,1,1)$.
% From Lemma \ref{lem:3.1}, \eqref{eq:5.5}, and
% \eqref{eq:5.1} we find,
% $\vec{\alpha}_{\infty,*}=\alpha_\infty\mathstrut^t(1,1,1)$. 
 Again using
 \eqref{eq:5.5}, we obtain that,
 \begin{equation*}
  \vec{0}
   =
   \sum_{j=1}^3
   \frac{\vec{b}^{(j)}_{\infty,*}}{|\vec{b}^{(j)}_{\infty,*}|}.
 \end{equation*}
 Therefore, by uniqueness of the Fermat-Torricelli problem, we obtain
 $\vec{a}_{\infty,*}=\vec{a}_\infty$(cf. the Fermat-Forricelli Problem
 \cite[Theorem 18.3]{MR1677397}). 
\end{proof}

Hereafter we denote
$\vec{\alpha}_\infty:=\alpha_\infty(1,1,1)$. To prove the exponential
stability for $(\vec{\alpha}_\infty, \vec{a}_\infty)$ of the system
\eqref{eq:1.1}, we study the linearized problem of \eqref{eq:1.1} and
show the decay properties of the solution to the linearized problem,
Propositions \ref{prop:5.2} and \ref{prop:5.4}. We first derive the linearized
problem in Lemma \ref{lem:linpr}.
\begin{lemma}
 [Linearized problem]
 \label{lem:linpr}
The linearized problem of \eqref{eq:1.1} around
$(\vec{\alpha}_\infty,\vec{a}_\infty)$ is given as,
%\begin{equation}
 %\vec{\alpha}(t)
%  =
 % \vec{\alpha}_\infty
 % +
 % \varepsilon\vec{\alpha}_L(t)
 % ,\quad
 % \vec{a}(t)
 % =
%  \vec{a}_\infty
 % +
 % \varepsilon\vec{a}_L(t)
%\end{equation}

%into \eqref{eq:1.1} and taking $\varepsilon\rightarrow0$, we obtain
\begin{equation}
 \label{eq:5.7}
  \left\{
   \begin{aligned}
    \frac{d\vec{\alpha}_L}{dt}
    =&
    -
    \gamma
    \sigma_{\MisOriAngle\MisOriAngle}
    (0)
    \mathbb{B}_\infty\vec{\alpha}_L,
    \quad
    \mathbb{B}_{\infty}
    =
    \begin{pmatrix}
     |\vec{b}^{(1)}_{\infty}|+|\vec{b}^{(2)}_{\infty}|
     &
     -|\vec{b}^{(2)}_{\infty}|
     &
     -|\vec{b}^{(1)}_{\infty}| \\
     -|\vec{b}^{(2)}_{\infty}| &
     |\vec{b}^{(2)}_{\infty}|+|\vec{b}^{(3)}_{\infty}| &
     -|\vec{b}^{(3)}_{\infty}|\\
     -|\vec{b}^{(1)}_{\infty}| &
     -|\vec{b}^{(3)}_{\infty}| &
     |\vec{b}^{(3)}_{\infty}|+|\vec{b}^{(1)}_{\infty}|
    \end{pmatrix}, \\
     \frac{d\vec{a}_L}{dt}
    &=
    -
    \eta\sigma(0)
    \sum_{j=1}^3
    \Biggl(
    \frac{1}{|\vec{b}_\infty^{(j)}|}
    \left(
    I
    -
    \frac{\vec{b}_\infty^{(j)}}{|\vec{b}_\infty^{(j)}|}
    \otimes
    \frac{\vec{b}_\infty^{(j)}}{|\vec{b}_\infty^{(j)}|}
    \right)
    \Biggr)\vec{a}_L
    =:
    -
    \eta\sigma(0)
    L_{\vec{a}}\vec{a}_L,
    \\
    \sum_{j=1}^3
    \frac{\vec{b}^{(j)}_{\infty}}{|\vec{b}^{(j)}_{\infty}|}
    &=
    \vec{0},
   \end{aligned}
  \right.
\end{equation}
 where $\otimes$ denotes the outer product of two vectors.
\end{lemma}
\begin{proof}
%\par {\textcolor{blue} {This is Katya, April 21 2019: Continue Below}}\\
%\textcolor{blue}{Masashi: The detailed calculation is the following. Is
%it good to insert lemma to obtain the linearized problem?}
In order to obtain the linearized problem \eqref{eq:5.7},  we consider,
\begin{equation}
 \vec{\alpha}(t)
  =
  \vec{\alpha}_\infty
  +
  \varepsilon\vec{\alpha}_L(t)
  ,\quad
  \vec{a}(t)
  =
  \vec{a}_\infty
  +
  \varepsilon\vec{a}_L(t),
\end{equation}
in \eqref{eq:1.1}, and take a derivative with respect to
$\varepsilon$. 
Note, the third equation in \eqref{eq:5.7} came from the equilibrium
system \eqref{eq:1.5}, so we will only derive the equations for $\vec{\alpha}_L$ and $\vec{a}_L$.
%It is easy to show the equation for 
%the third equation in \eqref{eq:5.7}.
%$\hat{\vec{b}}^{(j)}_{\infty}$ ().
First, we derive the equation for $\vec{\alpha}_L$ (the first equation
in \eqref{eq:5.7}). Since
$\Delta^{(j)}\alpha=\varepsilon\Delta^{(j)}\alpha_L$ and
$\sigma_\theta(0)=0$, we compute,
\begin{equation*}
  \begin{split}
   &\quad
   \frac{d}{d\varepsilon}\bigg|_{\varepsilon=0}
   \left(
   -
   \Bigl(
   \sigma_\MisOriAngle(\Delta^{(j+1)}\alpha)|\vec{b}^{(j+1)}|
   -
   \sigma_\MisOriAngle(\Delta^{(j)}\alpha)|\vec{b}^{(j)}|
   \Bigr)
   \right) \\
   &=
   \frac{d}{d\varepsilon}\bigg|_{\varepsilon=0}
   \left(
   -
   \Bigl(
   \sigma_\MisOriAngle(\varepsilon\Delta^{(j+1)}\alpha_L)
   |\vec{x}^{(j+1)}-\vec{a}_\infty-\varepsilon\vec{a}_L|
   -
   \sigma_\MisOriAngle(\varepsilon\Delta^{(j)}\alpha_L)
   |\vec{x}^{(j)}-\vec{a}_\infty-\varepsilon\vec{a}_L|
   \Bigr)
   \right) \\
   &=
   -
   \sigma_{\MisOriAngle\MisOriAngle}
   (\varepsilon\Delta^{(j+1)}\alpha_L)
   \Delta^{(j+1)}\alpha_L
   |\vec{x}^{(j+1)}-\vec{a}_\infty-\varepsilon\vec{a}_L|
   +
   \sigma_{\MisOriAngle}
   (\varepsilon\Delta^{(j+1)}\alpha_L)
   \frac{(\vec{x}^{(j+1)}-\vec{a}_\infty-\varepsilon\vec{a}_L)\cdot\vec{a}_L}{|\vec{x}^{(j+1)}-\vec{a}_\infty-\varepsilon\vec{a}_L|} \\
   &\quad
   +
   \sigma_{\MisOriAngle\MisOriAngle}
   (\varepsilon\Delta^{(j)}\alpha_L)
   \Delta^{(j)}\alpha_L
   |\vec{x}^{(j)}-\vec{a}_\infty-\varepsilon\vec{a}_L|
   -
   \sigma_{\MisOriAngle}
   (\varepsilon\Delta^{(j)}\alpha_L)
   \frac{(\vec{x}^{(j)}-\vec{a}_\infty-\varepsilon\vec{a}_L)\cdot\vec{a}_L}{|\vec{x}^{(j)}-\vec{a}_\infty-\varepsilon\vec{a}_L|}
   \bigg|_{\varepsilon=0} \\ 
   &=
   -
   \sigma_{\MisOriAngle\MisOriAngle}
   (0)
   \Delta^{(j+1)}\alpha_L
   |\vec{b}^{(j+1)}_\infty|
   +
   \sigma_{\MisOriAngle\MisOriAngle}
   (0)
   \Delta^{(j)}\alpha_L
   |\vec{b}^{(j)}_\infty| \\
   &=
   -
   \sigma_{\MisOriAngle\MisOriAngle}
   (0)
   \left(
   (|\vec{b}^{(j+1)}_\infty|+|\vec{b}^{(j)}_\infty|)
   \alpha^{(j)}_L
   -
   |\vec{b}^{(j+1)}_\infty|
   \alpha^{(j+1)}_L
   -
   |\vec{b}^{(j)}_\infty|
   \alpha^{(j-1)}_L
   \right).
  \end{split} 
\end{equation*}
Next, we will elaborate on some details of the right-hand side of the
equation for $\vec{a}_L$ (the second equation in \eqref{eq:5.7}).
Since $\sigma_\theta(0)=0$, we compute,
\begin{equation*}
 \begin{split}
  &\frac{d}{d\varepsilon}\bigg|_{\varepsilon=0}
 \left(
%  \sum_{j=1}^3
%   \left(
%    1+\frac12\Bigl(\alpha^{(j-1)}(t)-\alpha^{(j)}(t)\Bigr)^2
%   \right)
  \sigma(\Delta^{(j)}\alpha)
  \frac{\vec{b}^{(j)}}{|\vec{b}^{(j)}|}
  \right) \\
  &=
  \frac{d}{d\varepsilon}\bigg|_{\varepsilon=0}
  \left(
  \sum_{j=1}^3
%  \left(
%  1+\frac{\varepsilon^2}2\Bigl(\alpha_L^{(j-1)}(t)-\alpha_L^{(j)}(t)\Bigr)^2
%  \right)
  \sigma(\varepsilon\Delta^{(j)}\alpha_L)
  \frac{\vec{x}^{(j)}-\vec{a}_\infty-\varepsilon\vec{a}_L}{|\vec{x}^{(j)}-\vec{a}_\infty-\varepsilon\vec{a}_L|}
  \right) \\
  &=
  \bigg(
  \sum_{j=1}^3
  \sigma_\theta(\varepsilon\Delta^{(j)}\alpha_L)\Delta^{(j)}\alpha_L
% \Bigl(\alpha_L^{(j-1)}(t)-\alpha_L^{(j)}(t)\Bigr)^2
  \frac{\vec{x}^{(j)}-\vec{a}_\infty-\varepsilon\vec{a}_L}{|\vec{x}^{(j)}-\vec{a}_\infty-\varepsilon\vec{a}_L|} \\
  &\quad
  +
  \sigma(\varepsilon\Delta^{(j)}\alpha_L)
%  \left(
%  1+\frac{\varepsilon^2}2\Bigl(\alpha_L^{(j-1)}(t)-\alpha_L^{(j)}(t)\Bigr)^2
%  \right)
  \left(
  \frac{-\vec{a}_L}{|\vec{x}^{(j)}-\vec{a}_\infty-\varepsilon\vec{a}_L|}
  +
  \frac{(\vec{x}^{(j)}-\vec{a}_\infty-\varepsilon\vec{a}_L)\cdot \vec{a}_L}{|\vec{x}^{(j)}-\vec{a}_\infty-\varepsilon\vec{a}_L|^3}
  (\vec{x}^{(j)}-\vec{a}_\infty-\varepsilon\vec{a}_L)
  \right)
  \bigg)\bigg|_{\varepsilon=0} \\
  &=
  \sigma(0)
  \sum_{j=1}^3
  \left(
  \frac{-\vec{a}_L}{|\vec{b}_\infty^{(j)}|}
  +
  \frac{(\vec{b}_\infty^{(j)}\cdot \vec{a}_L)}{|\vec{b}_\infty^{(j)}|^3}
  \vec{b}_\infty^{(j)}
  \right)
  =
  -
  \sigma(0)
  \sum_{j=1}^3
  \frac{1}{|\vec{b}_\infty^{(j)}|}
  \left(
  \vec{a}_L
  -
  \frac{(\vec{b}_\infty^{(j)}\cdot \vec{a}_L)}{|\vec{b}_\infty^{(j)}|^2}
  \vec{b}_\infty^{(j)}
  \right).
\end{split}
\end{equation*}
Note, that the term,
\begin{equation*}
 \frac{(\vec{b}_\infty^{(j)}\cdot \vec{a}_L)}{|\vec{b}_\infty^{(j)}|^2}
  \vec{b}_\infty^{(j)}
  =
  \left(
   \frac{\vec{b}_\infty^{(j)}}{|\vec{b}_\infty^{(j)}|}\cdot\vec{a}_L
  \right)
  \frac{\vec{b}_\infty^{(j)}}{|\vec{b}_\infty^{(j)}|}
  =
  \left(
   \frac{\vec{b}_\infty^{(j)}}{|\vec{b}_\infty^{(j)}|}
   \otimes
   \frac{\vec{b}_\infty^{(j)}}{|\vec{b}_\infty^{(j)}|}
  \right)
  \vec{a}_L,
\end{equation*}
where $\otimes$ denotes the outer product of two vectors.
%In general for $\vec{p}=(p_1,p_2,\ldots,p_n)$,
%$\vec{q}=(q_1,q_2,\ldots,q_n)\in\R^n$, $\vec{p}\otimes\vec{q}$ is a
%$n\times n$ matrix defined by $ \vec{p}\otimes\vec{q}:=(p_iq_j)_{i,j}$. Then we obtain
%\begin{equation*}
% ((\vec{p}\cdot\vec{q})\vec{p})_l
%  =
%  \sum_{k=1}^np_kq_kp_l
 % =
 % \sum_{k=1}^np_lp_kq_k
%  =
  %(\left(\vec{p}\otimes\vec{p}\right)\vec{q})_l.
%\end{equation*}
%\textcolor{blue}{Masashi: END}
\end{proof}

It is important to note that, in the linearized problem \eqref{eq:5.7}, we consider
a matrix of the form,  for $c_1$, $c_2$, $c_3\in\R$,
\begin{equation}
 \label{eq:3.2}
  \mathbb{C}
  :=\begin{pmatrix}
     c_1+c_2
     &
     -c_2
     &
     -c_1 \\
     -c_2 &
     c_2+c_3 &
     -c_3\\
     -c_1 &
     -c_3 &
     c_3+c_1
    \end{pmatrix}.
\end{equation}
By direct calculation, it can be shown that the eigenvalues of such matrix $\mathbb{C}$
 \eqref{eq:3.2} are
 \begin{equation*}
  0\ \text{and}\
   c_1+c_2+c_3
   \pm
   \sqrt{\frac12
   \Bigl(
   (c_1-c_2)^2+(c_2-c_3)^2+(c_3-c_1)^2
   \Bigr)}.
 \end{equation*}
 If $c_1$, $c_2$, $c_3\geq0$, then the matrix $\mathbb{C}$ is
 non-negative definite. Furthermore, if $c_1$, $c_2$, $c_3>0$, then
 the zero eigenvalue of $\mathbb{C}$ is simple, and $(1,1,1)$ is an eigenvector associated with the zero eigenvalue.

Next proposition gives the decay properties for solutions $\vec{\alpha}_L$ of \eqref{eq:5.7}.

\begin{proposition}
 \label{prop:5.2}
 Let $\vec{\alpha}_L$ be a solution of \eqref{eq:5.7}.  Assume that,
 $\vec{\alpha}_L(0)\cdot(1,1,1)=0$. Then, there exists a positive
 constant $\lambda_1>0$ which depends only on
 $|\vec{b}^{(j)}_\infty|$, such that,
 \begin{equation}
  \label{eq:5.8}
   |\vec{\alpha}_L(t)|
   \leq
   e^{-\gamma\sigma_{\MisOriAngle\MisOriAngle}(0)\lambda_1
   t}|\vec{\alpha}_L(0)|.
 \end{equation}
\end{proposition}

\begin{proof}
 Using \eqref{eq:5.7} and $\mathbb{B}_\infty(1,1,1)=\vec{0}$, we have that,
 \begin{equation*}
  \frac{d}{dt}
   (
   \vec{\alpha}_L\cdot(1,1,1)
   )
   =
   -
   \gamma\sigma_{\MisOriAngle\MisOriAngle}(0)
   \mathbb{B}_\infty\vec{\alpha}_L\cdot(1,1,1)
   =
   -
   \gamma\sigma_{\MisOriAngle\MisOriAngle}(0)
   \vec{\alpha}_L\cdot\mathbb{B}_\infty(1,1,1)
   =0,
 \end{equation*}
 hence, $\vec{\alpha}_L\cdot(1,1,1)=\vec{\alpha}_L(0)\cdot(1,1,1)=0$. Define,
 \begin{equation}
  \label{eq:5.31}
   \lambda_1
   :=
   |\vec{b}^{(1)}_\infty|+|\vec{b}^{(2)}_\infty|+|\vec{b}^{(3)}_\infty|
   -
   \sqrt{\frac12
   \Bigl(
   (|\vec{b}^{(1)}_\infty|-|\vec{b}^{(2)}_\infty|)^2
   +
   (|\vec{b}^{(2)}_\infty|-|\vec{b}^{(3)}_\infty|)^2
   +
   (|\vec{b}^{(3)}_\infty|-|\vec{b}^{(1)}_\infty|)^2
   \Bigr)},
 \end{equation}
 which is the smallest positive eigenvalue of $\mathbb{B}_\infty$.
 % due to Lemma \ref{lem:3.1}. 
 Since $|\vec{b}_\infty^{(j)}|>0$ for all $j=1,2,3$, and
 $\vec{\alpha}_L\cdot(1,1,1)=0$,
 %from Lemma \ref{lem:3.1},  
 we find that
 $\vec{\alpha}_L$ is a linear combination of the eigenvectors associated with the positive
 eigenvalues of $\mathbb{B}_\infty$. Thus, we obtain,
 \[
 \frac{d}{dt}|\vec{\alpha}_L(t)|^2
 =
 -2\gamma\sigma_{\MisOriAngle\MisOriAngle}(0)
 \mathbb{B}_\infty\vec{\alpha}_L(t)\cdot\vec{\alpha}_L(t)
 \leq
 -2\gamma\sigma_{\MisOriAngle\MisOriAngle}(0)\lambda_1
 |\vec{\alpha}_L(t)|^2.
 \]
 By the Gronwall's inequality, we obtain \eqref{eq:5.8}.
 \end{proof}

 In order to derive decay properties for the solution
 $\vec{a}_L$ of \eqref{eq:5.7} in Proposition \ref{prop:5.2}, we
 first show below that $   L_{\vec{a}}$ is positive definite.

\begin{lemma}
 \label{lem:5.3}
 There exists a positive constant $\lambda_2>0$, which depends
 only on $\vec{b}_\infty^{(j)}$, such that,
 \begin{equation}
  \label{eq:5.10}
   L_{\vec{a}}
   :=
   \sum_{j=1}^3
   \frac{1}{|\vec{b}_\infty^{(j)}|}
   \left(
    I
    -
    \frac{\vec{b}_\infty^{(j)}}{|\vec{b}_\infty^{(j)}|}
    \otimes
    \frac{\vec{b}_\infty^{(j)}}{|\vec{b}_\infty^{(j)}|}
   \right)
   \geq
   \lambda_2I.
 \end{equation}
\end{lemma}

\begin{proof}
 For $\vec{\xi}\in\R^2$, we have that,
 \[
 \left(
 \sum_{j=1}^3
 \frac{1}{|\vec{b}_\infty^{(j)}|}
 \left(
 I
 -
 \frac{\vec{b}_\infty^{(j)}}{|\vec{b}_\infty^{(j)}|}
 \otimes
 \frac{\vec{b}_\infty^{(j)}}{|\vec{b}_\infty^{(j)}|}
 \right)
 \vec{\xi}
 \cdot
 \vec{\xi}
 \right)
 =
 \sum_{j=1}^3
 \frac{1}{|\vec{b}_\infty^{(j)}|}
 \left(
 |\vec{\xi}|^2
 -
 \left(
 \frac{\vec{b}_\infty^{(j)}}{|\vec{b}_\infty^{(j)}|}
 \cdot
 \vec{\xi}
 \right)^2
 \right)
 \geq0,
 \]
 hence, we obtain \eqref{eq:5.10} with some non-negative constant $\lambda_2\geq0$.

 Assume now that $\lambda_2=0$, then there is
 $\vec{\xi}_0\in\Sp^1$ such that,
 \[
 \sum_{j=1}^3
 \frac{1}{|\vec{b}_\infty^{(j)}|}
 \left(
 1
 -
 \left(
 \frac{\vec{b}_\infty^{(j)}}{|\vec{b}_\infty^{(j)}|}
 \cdot
 \vec{\xi}_0
 \right)^2
 \right)
 =0.
 \]
 Thus,  $\vec{\xi}_0$ has to be parallel to all $\vec{b}_\infty^{(j)}$ for
 $j=1,2,3$. This is contradiction to the dimension of $\R^2$.
\end{proof}

The convergence rate of the global solution to the equilibrium state depends on the decay rate
of the linearized solution, hence it is important to give estimates for the
constant $\lambda_2$. We will give an explicit form of
$\lambda_2$ in Appendix \ref{sec:A}. As a matter of fact
  from detailed calculations  in Appendix \ref{sec:A}, the constant $\lambda_2$ depends only on $|\vec{b}_\infty^{(j)}|$, see \eqref{eq:A.1}.

Similar to the result of Proposition~\ref{prop:5.2}, we have,
\begin{proposition}
 \label{prop:5.4}
 Let $\vec{a}_L(t)$ be a solution of \eqref{eq:5.7}. Then,
 \begin{equation}
  \label{eq:5.11}
   |\vec{a}_L(t)|
   \leq 
   e^{-\eta\sigma(0)\lambda_2
   t}|\vec{a}_L(0)|,
 \end{equation}
 for all $t>0$, and the constant $\lambda_2>0$ is given by Lemma \ref{lem:5.3}.
\end{proposition}

Denote,
\begin{equation}
\label{eq:5.12k}
\lambda:=\min\{\gamma\sigma_{\MisOriAngle\MisOriAngle}(0)\lambda_1,\eta\sigma(0)\lambda_2\},
\end{equation}
 where $\lambda_1$, $\lambda_2$ are given in Propositions
  \ref{prop:5.2}, and \ref{prop:5.4}, respectively. 
\begin{remark}
Note that $\lambda_1$, $\lambda_2$ are derived from eigenvalues of the linearized problem \eqref{eq:5.7}. The constant $\lambda_1$ is the smallest positive eigenvalue of the linearized operator for the equation of the orientations $\vec{\alpha}$, see Proposition \ref{prop:5.2} and \eqref{eq:5.31}. The constant $\lambda_2$ is the smallest eigenvalue of the linearized operator for the equation of the triple junction $\vec{a}$, see Lemma \ref{lem:5.3}, Proposition \ref{prop:5.4}, and \eqref{eq:A.1}.
\end{remark}

For a solution $(\vec{\alpha}, \vec{a})$ of \eqref{eq:1.1}, define
$\vec{\alpha}_p$ and $\vec{a}_p$ as,
\begin{equation}
 \label{eq:5.12}
  \vec{\alpha}=\vec{\alpha}_\infty+\vec{\alpha}_{p},\quad
  \vec{a}=\vec{a}_\infty+\vec{a}_{p}.
\end{equation}
%\textcolor{blue}{Since $\Delta\alpha^{(j)}=\Delta\alpha^{(j)}_p$},
Then
$\vec{\alpha}_{p}$ and $\vec{a}_p$ satisfy,
\begin{equation}
 \label{eq:5.13}
  \left\{
   \begin{aligned}
    (\vec{\alpha}_p)_t
    &=
    -
    \gamma\sigma_{\MisOriAngle\MisOriAngle}(0)
    \mathbb{B}_\infty\vec{\alpha}_p
    +
    \gamma
    \left(
    \sigma_{\MisOriAngle\MisOriAngle}(0)
    \mathbb{B}_\infty\vec{\alpha}_p
    -
    \vec{\Phi}(\vec{\alpha},\vec{a})
    \right)
    ,
%    \sigma_{\MisOriAngle\MisOriAngle}(0)
%    (\mathbb{B}_\infty\vec{\alpha}_p})
%    \left(
%    (|\vec{b}^{(j+1)}_\infty|+|\vec{b}^{(j)}_\infty|)
%    \alpha^{(j)}
%    -
%    |\vec{b}^{(j+1)}_\infty|
%    \alpha^{(j+1)}
%    -
%    |\vec{b}^{(j)}_\infty|
%    \alpha^{(j-1)}
%    \right)
%    } \\
%    &\quad
%    \textcolor{blue}{
%    +
%    \left(
%    \sigma_{\MisOriAngle\MisOriAngle}
%    (0)
%    \left(
%    (|\vec{b}^{(j+1)}_\infty|+|\vec{b}^{(j)}_\infty|)
%    \alpha^{(j)}
%    -
%    |\vec{b}^{(j+1)}_\infty|
%    \alpha^{(j+1)}
%    -
%    |\vec{b}^{(j)}_\infty|
%    \alpha^{(j-1)}
%    \right)
%    -\left(
%    \sigma_{\MisOriAngle}(\Delta^{(j+1)}\alpha_p)|\vec{b}^{j+1}|
%    -
%    \sigma_{\MisOriAngle}(\Delta^{(j)}\alpha_p)|\vec{b}^{j}|
%    \right)
%    \right)
%    }
%    -\mathbb{B}_\infty\vec{\alpha}_p
%    -(\mathbb{B}-\mathbb{B}_\infty)\vec{\alpha}_p, 
    \\
    (\vec{a}_p)_t
    &=
    -
    \eta\sigma(0)
    L_{\vec{a}}\vec{a}_p
    +
    \eta
    \sum_{j=1}^3
    \left(
    \frac{\sigma(0)
    }{|\vec{b}_\infty^{(j)}|}
    \left(
    I
    -
    \frac{\vec{b}_\infty^{(j)}}{|\vec{b}_\infty^{(j)}|}
    \otimes
    \frac{\vec{b}_\infty^{(j)}}{|\vec{b}_\infty^{(j)}|}
    \right)
    \vec{a}_p
    +
%    \left(
%    1+
%    \frac12(\alpha^{(j-1)}-\alpha^{(j)})^2
%    \right)
    \sigma(\Delta^{(j)}\alpha)
    \frac{\vec{b}^{(j)}}{|\vec{b}^{(j)}|}
    \right), \\
    \vec{\alpha}_p(0)&=\vec{\alpha}_0-\vec{\alpha}_\infty,\qquad
    \vec{a}_p(0)=\vec{a}_0-\vec{a}_\infty,
   \end{aligned}
  \right.
\end{equation}
where $L_{\vec{a}}$ is defined as in (\ref{eq:5.10})
 and
 $\vec{\Phi}(\vec{\alpha},\vec{a})$ is given by
\begin{equation*}
 \vec{\Phi}(\vec{\alpha},\vec{a})
  =
  \left(
   \Phi^{(j)}(\vec{\alpha},\vec{a})
  \right)_j
 ,\quad
 \Phi^{(j)}(\vec{\alpha},\vec{a})
 =
 \sigma_\MisOriAngle(\Delta^{(j+1)}\alpha)|\vec{b}^{(j+1)}|
 -
 \sigma_\MisOriAngle(\Delta^{(j)}\alpha)|\vec{b}^{(j)}|.
\end{equation*}
Note that the first equation of the original problem
\eqref{eq:1.1} can be written as
$\vec{\alpha}_t=-\gamma\vec{\Phi}(\vec{\alpha},\vec{a})$.
%\begin{equation*}
 %L_{\vec{a}}
 %=
 %\sum_{j=1}^3
 %\frac{1}{|\vec{b}_\infty^{(j)}|}
% \left(
 % I
%  -
%  \frac{\vec{b}_\infty^{(j)}}{|\vec{b}_\infty^{(j)}|}
 % \otimes
 % \frac{\vec{b}_\infty^{(j)}}{|\vec{b}_\infty^{(j)}|}
 %\right).
%\end{equation*}

In the next two Lemmas \ref{lem:5.5} and \ref{lem:5.6},  we will show the decay estimates for the perturbation terms $\vec{\alpha}_p$, $\vec{a}_p$.

\begin{lemma}
 \label{lem:5.5}
 Let $(\vec{\alpha}, \vec{a})$ be a global solution
 of \eqref{eq:1.1}, and let $\vec{\alpha}_p$ and $\vec{a}_p$ be defined as in \eqref{eq:5.12}. Then, there are positive constants
 $\Cr{const:5.11}>0$ and $\Cr{eps:5.11}>0$, such that, if
 $|\vec{\alpha}_p(t)|+|\vec{a}_p(t)|<\Cr{eps:5.11}$ for all $t>0$, then
 \begin{equation}
  \label{eq:5.14}
   e^{
   \lambda t}|\vec{\alpha}_p(t)|
   \leq
   |\vec{\alpha}_p(0)|
   +
   \Cr{const:5.11}
   \int_0^t
   e^{\lambda s}
   (
   |\vec{\alpha}_p(s)|^2
   +
   |\vec{a}_p(s)|^2
   )
   \,ds.
 \end{equation}
\end{lemma}

\begin{proof}
 By the Duhamel principle,
 \begin{equation}
  \label{eq:5.15}
  \vec{\alpha}_p(t)
   =
   e^{-t\gamma\sigma_{\MisOriAngle\MisOriAngle}(0)\mathbb{B}_\infty}
   \vec{\alpha}_p(0)
   +
   \gamma
   \int_0^t
   e^{-(t-s)\gamma\sigma_{\MisOriAngle\MisOriAngle}(0)\mathbb{B}_\infty}
   \left(
    \sigma_{\MisOriAngle\MisOriAngle}(0)
    \mathbb{B}_\infty\vec{\alpha}_p(s)
    -
    \vec{\Phi}(\vec{\alpha}(s),\vec{a}(s))
   \right)
   \,ds.
 \end{equation}
 Since $\vec{\alpha}_0\cdot(1,1,1)=\vec{\alpha}_\infty\cdot(1,1,1)$,
 $\vec{\alpha}_p(0)$ is perpendicular to the vector $(1,1,1)$.
 Also the function
 $\sigma_{\MisOriAngle\MisOriAngle}(0)
 \mathbb{B}_\infty\vec{\alpha}_p(s)
  -
 \vec{\Phi}(\vec{\alpha}(s),\vec{a}(s))
 $
 is also perpendicular to the vector $(1,1,1)$ for all $s>0$, since $\mathbb{B}_\infty\vec{\alpha}_p\cdot(1,1,1)=\vec{\Phi}\cdot(1,1,1)=0$.
 By Proposition \ref{prop:5.2}
 we obtain that,
 \begin{equation}
  \label{eq:5.16}
    |e^{-t\gamma\sigma_{\MisOriAngle\MisOriAngle}(0)
    \mathbb{B}_\infty}\vec{\alpha}_p(0)|
    \leq
    e^{-
    \lambda t}|\vec{\alpha}_p(0)|, 
 \end{equation}
 and
 \begin{multline}
  \label{eq:5.41}
  \left|
  e^{-(t-s)\gamma\sigma_{\MisOriAngle\MisOriAngle}(0)
  \mathbb{B}_\infty}
  \left(
  \sigma_{\MisOriAngle\MisOriAngle}(0)
  \mathbb{B}_\infty\vec{\alpha}_p(s)
  -
  \vec{\Phi}(\vec{\alpha}(s),\vec{a}(s))
  \right)
  \right| \\
  \leq
  e^{-
  \lambda(t-s)}
  \left|
  \sigma_{\MisOriAngle\MisOriAngle}(0)
  \mathbb{B}_\infty\vec{\alpha}_p(s)
  -
  \vec{\Phi}(\vec{\alpha}(s),\vec{a}(s))
  \right|.
 \end{multline}
 By the Taylor expansion $\vec{\Phi}(\vec{\alpha},\vec{a})$
 around  $(\vec{\alpha}_\infty,\vec{a}_\infty)$, we obtain
 \begin{equation*}
  \vec{\Phi}(\vec{\alpha},\vec{a})
   =
   \sigma_{\MisOriAngle\MisOriAngle}(0)
   \mathbb{B}_\infty\vec{\alpha}_p
   +
   O\left(|\vec{\alpha}_p|^2+|\vec{a}_p|^2\right),
   \qquad
   \text{as}
   \   
   |\vec{\alpha}_p|+|\vec{a}_p|
   \rightarrow0.
 \end{equation*}
 Hence, there are $\Cl{const:5.11}>0$ and $\Cl[eps]{eps:5.11}>0$ such
 that, if $|\vec{\alpha}_p|+|\vec{a}_p|<\Cr{eps:5.11}$,
 \begin{equation}
  \label{eq:5.42}
  \left|
   \sigma_{\MisOriAngle\MisOriAngle}(0)
   \mathbb{B}_\infty\vec{\alpha}_p(s)
   -
   \vec{\Phi}(\vec{\alpha}(s),\vec{a}(s))
  \right|  
  \leq
  \Cr{const:5.11}
  (|\vec{\alpha}_p|^2
  +
  |\vec{a}_p|^2).
 \end{equation}
 Using the estimates \eqref{eq:5.16}, \eqref{eq:5.41} and
 \eqref{eq:5.42} in \eqref{eq:5.15}, we obtain the result
 \eqref{eq:5.14}.

%By Corollary \ref{cor:3.2}, we have,
% \begin{equation}
%  \label{eq:5.17}
%   \begin{split}
%    |(\mathbb{B}(s)-\mathbb{B}_\infty)\vec{\alpha}_p(s)|
%%    &\leq 3
%%    \left(
%%    \left|
%%    |\vec{b}^{(1)}(s)|-|\vec{b}_\infty^{(1)}|
%%    \right|
%%    +
%%    \left|
%%    |\vec{b}^{(2)}(s)|-|\vec{b}_\infty^{(2)}|
%%    \right|
%%    +
%%    \left|
%%    |\vec{b}^{(3)}(s)|-|\vec{b}_\infty^{(3)}|
%%    \right|
%%    \right)
%%    |\vec{\alpha}_p(s)| \\
%    &\leq 3
%    \left(
%    \left|
%    \vec{b}^{(1)}(s)-\vec{b}_\infty^{(1)}
%    \right|
%    +
%    \left|
%    \vec{b}^{(2)}(s)-\vec{b}_\infty^{(2)}
%    \right|
%    +
%    \left|
%    \vec{b}^{(3)}(s)-\vec{b}_\infty^{(3)}
%    \right|
%    \right)
%    |\vec{\alpha}_p(s)| \\
%    &\leq 9
%    |\vec{a}_p(s)|
%    |\vec{\alpha}_p(s)|
%   \end{split}
% \end{equation}
\end{proof}

\begin{lemma}
 \label{lem:5.6}
 Let $(\vec{\alpha}, \vec{a})$ be a global solution
 of \eqref{eq:1.1}, and let $\vec{\alpha}_p$ and $\vec{a}_p$ be defined as in \eqref{eq:5.12}. 
% Let $\vec{\alpha}_p$ and $\vec{a}_p$ be defined as in \eqref{eq:5.12}. 
 Then, there are positive constants $\Cl{const:5.4}>0$
 and $\Cl[eps]{eps:5.1}>0$, such that, if
 $|\vec{\alpha}_p(t)|+|\vec{a}_p(t)|<\Cr{eps:5.1}$ for
 all $t>0$, then,
 \begin{equation}
  \label{eq:5.18}
   e^{\lambda t}
   |\vec{a}_p(t)|
   \leq
   |\vec{a}_p(0)|
   +
   \Cr{const:5.4}
   \int_0^t
   e^{\lambda s}
   \left(
    |\vec{\alpha}_p(s)|^2
    +
    |\vec{a}_p(s)|^2
   \right)
   \,ds,
 \end{equation}
 for all $t>0$.
\end{lemma}

\begin{proof}
 By the Duhamel principle for (\ref{eq:5.13}), we have,
 \begin{equation}
  \label{eq:5.19}
   \vec{a}_p(t)
   =
   e^{-t\eta\sigma(0)L_{\vec{a}}}
   \vec{a}_p(0)
   +
   \eta
   \int_0^t
   e^{-(t-s)\eta\sigma(0)L_{\vec{a}}}
   F(s)\,ds,
 \end{equation}
 where
 \begin{equation}
  \label{eq:5.20}
   F(s):=
   \sum_{j=1}^3
   \left(
    \frac{\sigma(0)}{|\vec{b}_\infty^{(j)}|}
    \left(
     I
     -
     \frac{\vec{b}_\infty^{(j)}}{|\vec{b}_\infty^{(j)}|}
     \otimes
     \frac{\vec{b}_\infty^{(j)}}{|\vec{b}_\infty^{(j)}|}
    \right)
    \vec{a}_p
    +
%    \left(
%     1+
%     \frac12(\alpha^{(j-1)}-\alpha^{(j)})^2
%    \right)
    \sigma(\Delta^{(j)}\alpha)
    \frac{\vec{b}^{(j)}}{|\vec{b}^{(j)}|}
   \right).
 \end{equation}
 By Lemma \ref{lem:5.3} and Proposition \ref{prop:5.4}, we obtain,
 \begin{equation}
  \label{estFm}
  |\vec{a}_p(t)|
   \leq
   e^{-\lambda t}
   |\vec{\alpha}_p(0)|
   +
   \eta
   \int_0^t
   e^{-\lambda (t-s)}
   |F(s)|\,ds.
 \end{equation}
 Next, by the Taylor expansion of
 $\sigma(\Delta^{(j)}\alpha)\vec{b}^{(j)}/|\vec{b}^{(j)}|$ around 
 $(\vec{\alpha}_\infty,\vec{a}_\infty)$, we obtain,
 \begin{equation}
   \label{eq:5.21modify}
    \sum_{j=1}^3
    \sigma(\Delta^{(j)}\alpha)\frac{\vec{b}^{(j)}}{|\vec{b}^{(j)}|}
    =
    -
    \sum_{j=1}^3
    \frac{\sigma(0)}{|\vec{b}_\infty^{(j)}|}
    \left(
     I
     -
     \frac{\vec{b}_\infty^{(j)}}{|\vec{b}_\infty^{(j)}|}
     \otimes
     \frac{\vec{b}_\infty^{(j)}}{|\vec{b}_\infty^{(j)}|}
    \right)
    \vec{a}_p
    +
    O(|\vec{\alpha}_p|^2+|\vec{a}_p|^2),\quad
    \text{as}\ |\vec{\alpha}_p|+|\vec{a}_p|\rightarrow0.   
%   \frac{1}{|\vec{x}^{(j)}-\vec{a}|}
%   =
%   \frac{1}{|\vec{x}^{(j)}-\vec{a}_\infty-\vec{a}_p|}
%   =
%   \frac{1}{|\vec{b}^{(j)}_\infty-x\vec{a}_p|}
%   \bigg|_{x=1}.
 \end{equation}
 Using, \eqref{eq:5.21modify}
 % and \eqref{eq:5.7}
 in \eqref{eq:5.20}, we obtain,
 \begin{equation}
%  \begin{split}
   F(s)
    =
    O(|\vec{\alpha}_p|^2+|\vec{a}_p|^2),\quad
    \text{as}\ |\vec{\alpha}_p|+|\vec{a}_p|\rightarrow0.   
    %   &=
%   \frac12\sum_{j=1}^3
%   (\alpha^{(j-1)}-\alpha^{(j)})^2
%   \hat{\vec{b}}^{(j)}
%   +O(|\vec{a}_p|^2) \\
%   &=
%   \frac12
%   \sum_{j=1}^3
%   (\alpha_p^{(j-1)}-\alpha_p^{(j)})^2
%   \hat{\vec{b}}^{(j)}
%   +O(|\vec{a}_p|^2),
%   \quad
%   \text{as}\ |\vec{a}_p|\rightarrow0.
%  \end{split}
 \end{equation}
 Hence, there are $\Cr{const:5.4}>0$ and $\Cr{eps:5.1}>0$ such that, if
 $|\vec{\alpha}_p|+|\vec{a}_p|<\Cr{eps:5.1}$,
 \begin{equation}
  \label{estF}
%  \begin{split}
%  |F(s)|
%  &\leq
%  \frac12\sum_{j=1}^3
%  (\alpha^{(j-1)}_p-\alpha_p^{(j)})^2+
%  \Cr{const:5.4}|\vec{a}_p|^2 \\
%  &\leq
%  \sum_{j=1}^3
%  \left(
%  (\alpha^{(j-1)}_p)^2+(\alpha^{(j)}_p)^2
%  \right)
%  +
%  \Cr{const:5.4}|\vec{a}_p|^2
%  =
%  2|\vec{\alpha}_p|^2
%  +
%  \Cr{const:5.4}|\vec{a}_p|^2.
%  \end{split}
  |F(s)|
  \leq
  \Cr{const:5.4}
  \left(
  |\vec{\alpha}_p|^2
  +
  |\vec{a}_p|^2
  \right).
 \end{equation}
Using estimate (\ref{estF}) in (\ref{estFm}), we conclude with the
desired estimate \eqref{eq:5.18} on $\vec{a}_p(t)$.
\end{proof}

%\todo{More detailed estimates of $C_5$ and $\varepsilon_1$ will be
%presented in Appendix B}

Now we are in position to show exponential stability
  of the asymptotic profile of the solution
 % for $(\vec{\alpha}_\infty, \vec{a}_\infty)$ 
of the system \eqref{eq:1.1}.

\begin{theorem}
 \label{thm:5.1}
 There is a small constant $\Cl[eps]{eps:5.2}>0$ such that,  if
 $|\vec{\alpha}_0-\vec{\alpha}_\infty|+|\vec{a}_0-\vec{a}_\infty|<\Cr{eps:5.2}$, then the associated
 global solution $(\vec{\alpha},\vec{a})$ of the system \eqref{eq:1.1} satisfies,
 \begin{equation}
  \label{eq:5.24}
  |\vec{\alpha}(t)-\vec{\alpha}_\infty|
   +
   |\vec{a}(t)-\vec{a}_\infty|
   \leq \Cr{const:5.5}e^{-\lambda^{\star} t},
 \end{equation}
 for some positive constants $\Cl{const:5.5}, \lambda^{\star}>0$. The
 decay order $\lambda^{\star}$ is explicitly estimated as,
 \begin{equation}\label{eq:5.41decay}
  \lambda^{\star} \geq \lambda,
%\geq
 %  \min\{\textcolor{blue}{\gamma\sigma_{\MisOriAngle\MisOriAngle}(0)}\lambda_1,\textcolor{blue}{\eta\sigma(0)}\lambda_2\},
 \end{equation}
 where $\lambda$ is defined in (\ref{eq:5.12k}).
%$\lambda_1$, $\lambda_2>0$ are positive constants depending only on
 %$\vec{b}_\infty^{(j)}$. 

\end{theorem}
\begin{proof}
 [Proof of Theorem \ref{thm:5.1}]
 Let $\vec{\alpha}_p$ and $\vec{a}_p$ be defined as in \eqref{eq:5.12}, and
 define $V(t):=e^{\lambda t}|\vec{\alpha}_p(t)|$ and
 $W(t):=e^{\lambda t}|\vec{a}_p(t)|$.  Next, take sufficiently small
 $0<\Cr{eps:5.2}<\min\{\Cr{eps:5.11},
 \Cr{eps:5.1}\}/2$ such that \eqref{eq:4.4} is fulfilled if an initial data $(\vec{\alpha}_0,\vec{a}_0)$ satisfy $|\vec{\alpha}_0-\vec{\alpha}_\infty|+|\vec{a}_0-\vec{a}_\infty|<\Cr{eps:5.2}$. Here, the constants $\Cr{eps:5.11}$ and
 $\Cr{eps:5.1}$ are given in Lemmas \ref{lem:5.5} and \ref{lem:5.6}, and
 assume that $|\vec{\alpha}_0-\vec{\alpha}_\infty|+|\vec{a}_0-\vec{a}_\infty|<\Cr{eps:5.2}$, namely $|\vec{\alpha}_p(0)|+|\vec{a}_p(0)|<\Cr{eps:5.2}$. 
 In order to show \eqref{eq:5.24}, it is enough to show the boundedness for $V(t)+W(t)$.
 Now, assume $V(t)+W(t)<2\Cr{eps:5.2}$ for $0\leq t< t_0$ and
 $V(t_0)+W(t_0)=2\Cr{eps:5.2}$. Note that,
 $V(t)+W(t)<\Cr{eps:5.11}$, $\Cr{eps:5.1}$ for
 $0<t<t_0$, thus we can apply Lemmas \ref{lem:5.5} and
 \ref{lem:5.6}.  Therefore, from \eqref{eq:5.14}, \eqref{eq:5.18} and
 that $V(t)$, $W(t)\leq 2\Cr{eps:5.2}$ for $0<t\leq t_0$, we obtain,
 \begin{equation}
  \begin{split}
   \label{eq:5.25}
   V(t)+W(t)
   &\leq
   V(0)+W(0)
   +
   \frac{\Cr{const:5.6}}{2}
   \int_0^te^{-\lambda s}(V^2(s)+W^2(s))\,ds \\
   &\leq
   \Cr{eps:5.2}
   +
   \Cr{const:5.6}\Cr{eps:5.2}\int_0^te^{-\lambda s}(V(s)+W(s))\,ds,
  \end{split}
 \end{equation}
 where
 $\Cl{const:5.6}=2(\Cr{const:5.11}+\Cr{const:5.4})>0$. Applying
 the Gronwall's inequality to \eqref{eq:5.25}, we have that,
 \begin{equation}
  V(t)+W(t)
   \leq
   \Cr{eps:5.2}
   +
   \Cr{const:5.6}
   \Cr{eps:5.2}^2
   \int_0^t
   e^{-\lambda s}
   \exp
   \left(
   \Cr{const:5.6}
   \Cr{eps:5.2}
   \int_s^t
   e^{-\lambda u}
   \,du
   \right)
   \,ds,
   \quad
   0\leq t\leq t_0.
 \end{equation}
 Hence, we can easily obtain that,
 \begin{equation*}
  \Cr{const:5.6}
   \Cr{eps:5.2}^2
   \int_0^t
   e^{-\lambda s}
   \exp
   \left(
   \Cr{const:5.6}
   \Cr{eps:5.2}
   \int_s^t
   e^{-\lambda u}
   \,du
   \right)
   \,ds
   \leq
   \frac{\Cr{const:5.6}
   \Cr{eps:5.2}^2}
   {\lambda}
   \exp
   \left(
    \frac{\Cr{const:5.6}
    \Cr{eps:5.2}}
    {\lambda}
   \right).
 \end{equation*}
 Thus, if we take $\Cr{eps:5.2}>0$ sufficiently small as,
\begin{equation}
 \label{eq:5.27}
  \frac{\Cr{const:5.6}
  \Cr{eps:5.2}}
  {\lambda}
  \exp
  \left(
   \frac{\Cr{const:5.6}
   \Cr{eps:5.2}}
   {\lambda}
  \right)
  <
  1,
\end{equation}
 then we deduce that $V(t_0)+W(t_0)< 2\Cr{eps:5.2}$, which contradicts the
 definition of $t_0$ above. Therefore, $V(t)+W(t)$ is bounded for $0<t<\infty$,
 and we obtain \eqref{eq:5.24}.
\end{proof}

\begin{remark}
Note that our argument is based on the uniqueness of the equilibrium state
for the system \eqref{eq:1.5}. Otherwise, we can not recover full
convergence in time as in 
Proposition \ref{prop:5.1}. However, thanks to the uniqueness of the equilibrium
state \eqref{eq:1.5}, we can consider the linearized problem
\eqref{eq:5.7} around the equilibrium state, and we can obtain the exponential
uniform estimate \eqref{eq:5.24} for the solution of the system \eqref{eq:1.1}.
\end{remark}

%\begin{remark}
% \color{blue}(Masashi: Is this remark still needed? Since I delete the
% brief derivation section, the parameters $\gamma$ and $\eta$ are not
% appeared in the article) If we do not assume $\eta=\gamma=1$ (see Sect
% \ref{sec:2}, (\ref{eq:2.14})-(\ref{eq:2.17})), the order of the decay
% for $k$ should be proportional to $\gamma$ and $\eta$. In
% fact, $k$ is given by
% $\min\{\gamma\textcolor{blue}{\sigma(0)}\lambda_1,\eta\textcolor{blue}{\sigma_{\MisOriAngle\MisOriAngle}(0)}\lambda_2\}$,
% where $\lambda_1$, $\lambda_2$ is defined by \eqref{eq:5.31}
% and \eqref{eq:5.10}, respectively.
%\end{remark}
Note also, that by Theorem \ref{thm:5.1}, we obtain exponential decay
of the total grain boundary energy to the equilibrium energy, that is

\begin{corollary}
 Under the same assumption as in Theorem \ref{thm:5.1}, the associated
 grain boundary energy $E(t)$ satisfies,
 \begin{equation}
  \label{eq:5.32}
  E(t)-E_\infty
   \leq
   \Cr{const:5.7}
   e^{-\lambda^{\star}t},
 \end{equation}
 for some positive constant $\Cl{const:5.7}>0$, where
 \begin{equation*}
  E_\infty
   :=
   \sigma(0)\sum_{j=1}^3|\vec{b}_\infty^{(j)}|.
 \end{equation*}
\end{corollary}

\begin{proof}
 Since $\alpha^{(1)}_\infty=\alpha^{(2)}_\infty=\alpha^{(3)}_\infty$, we obtain
 \begin{equation}
  \label{eq:5.33}
  \begin{split}
   E(t)-E_\infty
   &=
   \sum_{j=1}^3
   \left(
%    (
%    1
%    +
%    \frac12
%    (\alpha^{(j-1)}(t)-\alpha^{(j)}(t))^2)
   \sigma(\Delta^{(j)}\alpha(t))
   |\vec{b}^{(j)}(t)|
   -
   \sigma(0)
   |\vec{b}_\infty^{(j)}|
   \right) \\
   &\leq
   \sum_{j=1}^3
   \left(
   \sigma(0)
   |
   \vec{b}^{(j)}(t)
   -
   \vec{b}_\infty^{(j)}
   |
   +
   \left(
   \sigma(\Delta^{(j)}\alpha(t))
   -
   \sigma(0)
   \right)
   |\vec{b}^{(j)}(t)|
   \right) \\
   &\leq
   \sum_{j=1}^3
   \left(
   \sigma(0)
   |
   \vec{a}^{(j)}(t)
   -
   \vec{a}_\infty  
   |
   +
   \left(
   \Cl{const:5.12}
   |\Delta^{(j)}\alpha(t)|  
   \right)
   |\vec{b}^{(j)}(t)|
   \right) \\
   &\leq
   \sum_{j=1}^3
   \left(
   \sigma(0)
   |
   \vec{a}^{(j)}(t)
   -
   \vec{a}_\infty   
   |
   +
   2\Cr{const:5.12}
   |\vec{b}^{(j)}(t)|
   |\vec{\alpha}(t)-\vec{\alpha}_\infty|
   \right),
  \end{split} 
\end{equation}
 where $\Cr{const:5.12}=\sup_{|\MisOriAngle|<2\Cr{eps:5.2}}|\sigma_\MisOriAngle(\MisOriAngle)|$. 
 Using the dissipation estimate \eqref{eq:3.6} and the exponential decay estimate \eqref{eq:5.24}, we obtain 
 \eqref{eq:5.32}.
\end{proof}

\begin{remark}\label{rm:5.33}
 %\par 1. Note that, based on the estimate \eqref{eq:5.33}, the dominate
  % part of the energy decay seems to be due to the
 %the triple junction effect,
% $|\vec{b}^{(j)}(t)-\vec{b}^{(j)}_\infty|$. In fact, in the numerical
% experiments, it was observed that the energy  dissipation due to
% misorientation effect is much smaller than the
 %energy dissipation due to the triple junctions effect.
\par  Note, that the obtained exponential decay to equilibrium,
see estimates (\ref{eq:5.24}) and (\ref{eq:5.32}) is obtained by
considering linearized problem, Lemma \ref{lem:linpr}. Consideration
of the nonlinear problem instead could lead to potential power laws estimates
for the decay rates.
See also
discussion  and numerical experiments in Section \ref{sec:7}.
\end{remark}

\section{Extension to Grain Boundary Network}
\label{sec:6}

In this section, we extend our results to a grain boundary network
$\{\Gamma_t^{(j)}\}_{j}$. As in, for example, \cite{Katya-Chun-Mzn},  we define the total grain
boundary energy of the network, like,
\begin{equation} \label{eq:6.4e}
 E(t)
  =
  \sum_{j}
  \sigma
  (
  \Delta^{(j)}\alpha
  )
  |\Gamma_t^{(j)}|,
\end{equation}
where $\Delta^{(j)}\alpha$ is a misorientation, a difference between the
lattice orientation of the two neighboring grains which form the grain
boundary $\Gamma^{(j)}$. Then, the energetic variational principle
implies
\begin{equation}
 \label{eq:6.4}
  \left\{
  \begin{aligned}
   v_n^{(j)}
   &=
   \mu
   \sigma
   (
   \Delta^{(j)}\alpha
   )
   \kappa^{(j)},\quad\text{on}\ \Gamma_t^{(j)},\ t>0, \\
   \frac{d\alpha^{(k)}}{dt}
   &=
   -\gamma
   \frac{\delta E}{\delta \alpha^{(k)}},
   \\
   \frac{d\vec{a}^{(l)}}{dt}
   &=
   \eta
   \sum_{\vec{a}^{(l)}\in\Gamma^{(j)}_t}   
   \left(\sigma(\Delta^{(j)}\alpha)\frac{\vec{b}^{(j)}}{|\vec{b}^{(j)}|}\right),
   \quad t>0.
  \end{aligned}
  \right.
\end{equation}

As in \cite{Katya-Chun-Mzn}, we consider the relaxation parameters,
$\mu\rightarrow\infty$, and we further assume that the energy density
$\sigma(\MisOriAngle)$ is an even function with respect to
the misorientation $\MisOriAngle= \Delta^{(j)}\alpha$, that is, the
misorientation effects are symmetric with respect to the difference
between the lattice orientations. Then, the problem \eqref{eq:6.4}
becomes,
\begin{equation}
 \label{eq:6.7} \left\{
  \begin{aligned}
   \Gamma_t^{(j)} &\ \text{is a line segment between some}\
   \vec{a}^{(l_{j,1})}\ \text{and}\ \vec{a}^{(l_{j,2})}, \\
   \frac{d\alpha^{(k)}}{dt}
   &=
   -
   \gamma
   \sum_{\substack{\text{ grain with } \alpha^{(k')}\ \text{is the
         neighbor of the grain with }\ \alpha^{(k)} \\
   \Gamma^{(j)}_t\ \text{is formed by the two grains with }\ \alpha^{(k)}\ \text{and}\ \alpha^{(k')}}}
   |\Gamma^{(j)}_t|
   \sigma_{\MisOriAngle}(\alpha^{(k)}-\alpha^{(k')}), \\
   \frac{d\vec{a}^{(l)}}{dt}
   &=
   \eta
   \sum_{\vec{a}^{(l)}\in\Gamma^{(j)}_t}   
   \left(\sigma(\Delta^{(j)}\alpha)\frac{\vec{b}^{(j)}}{|\vec{b}^{(j)}|}\right).
  \end{aligned}
  \right.
\end{equation}
To obtain the global solution of the system \eqref{eq:6.7}, we 
 assume that
there are no critical events in the system (for
  example, the critical events are disappearance of
the grains and/or grain boundaries during coarsening of the system),
and we consider an
associated energy minimizing state,
$(\alpha_\infty^{(k)},\vec{a}_\infty^{(l)})$ of \eqref{eq:6.7}. 
%with same
%numbers of the grain boundaries, single grains, and triple junctions as
%the initial configuration themselves. 
Then,
$(\alpha_\infty^{(k)},\vec{a}_\infty^{(l)})$ satisfies,
\begin{equation}
 \label{eq:6.8}
 \left\{
  \begin{aligned}
   \Gamma_\infty^{(j)}&\ \text{is a line segment between some}\
   \vec{a}^{(l_{j,1})}_\infty\ \text{and}\ \vec{a}_\infty^{(l_{j,2})}, \\
   0
   &=
   -
   \gamma
   \sum_{\substack{\text{ grain with } \alpha^{(k')}\ \text{is the
         neighbor of the grain with }\ \alpha^{(k)} \\
   \Gamma^{(j)}_t\ \text{is formed by the two grains with }\ \alpha^{(k)}\ \text{and}\ \alpha^{(k')}}}
   |\Gamma_\infty^{(j)}|
   \sigma_{\MisOriAngle}(\alpha_\infty^{(k)}-\alpha_\infty^{(k')}), \\
   \vec{0}
   &=
   \eta
   \sum_{\vec{a}_\infty^{(l)}\in\Gamma^{(j)}_\infty}   
   \left(\sigma(\Delta^{(j)}\alpha_\infty)\frac{\vec{b}_\infty^{(j)}}{|\vec{b}_\infty^{(j)}|}\right).
  \end{aligned}
  \right.
\end{equation}
Hence, the total energy $E_\infty$ of the grain boundary network
(\ref{eq:6.8}) is
\begin{equation}
 \label{eq:6.9}
 \begin{split}
  E_\infty
  &=
  \sum_{j}
%  \left(
%  1+
%  \frac12
%  \left(
%  \Delta^{(j)}\alpha_\infty
%  \right)^2
%  \right)
  \sigma(\Delta^{(j)}\alpha_\infty)
  |\vec{b}_\infty^{(j)}| 
  =
  \inf
  \biggl\{
  \sum_{j}
  \sigma(\Delta^{(j)}\alpha_\infty) 
%  \left(
%  1+
%  \frac12
%  \left(
%  \Delta^{(j)}\alpha
%  \right)^2
%  \right)
  |\vec{b}^{(j)}|
  \biggr\}.
 \end{split}
\end{equation}
\begin{remark}
Note, we assume in (\ref{eq:6.7})-(\ref{eq:6.8}) that the total number of grains, grain boundaries and triple
    junctions are the same as in the initial configuration (assumption
    of no critical events in the network).
\end{remark}
If there is a neighborhood $U^{(l)}\subset\R^2$ of $\vec{a}_\infty^{(l)}$
such that
\begin{equation}
 \label{eq:6.10}
  E_\infty
  <
  \sum_{j}
  |\vec{b}^{(j)}|
\end{equation}
for all $\vec{a}^{(l)}\in U^{(l)}$, one can obtain a priori estimate for
the triple junctions, and, hence,  obtain the time global solution of
\eqref{eq:6.7}. Note that, the assumption \eqref{eq:6.10} is related to
the boundary condition of the line segments
$\Gamma_t^{(j)}$. Further, if the energy minimizing state is unique,
then we can proceed with the same argument as in Lemma \ref{lem:4.1}, and
obtain the global solution \eqref{eq:6.7} near the energy minimizing
state.

\begin{example}
 Note that, the solution of \eqref{eq:6.8} may not be unique even though the grain
 orientations are constant (misorientation is zero). 
%For given Dirichlet condition $x^{(1)}$
 %to $x^{(4)}$, we give two equilibrium states in Figure \ref{fig:6.1}.
 %Among all triple junction point $\vec{a}_\infty^{(1)}$,
 %$\vec{a}_\infty^{(2)}$, the Herring condition, which is the third
% equation in \eqref{eq:6.8}, is hold. When
 %$\vec{x}^{(1)}\vec{x}^{(2)}=\vec{x}^{(3)}\vec{x}^{(4)}=\sqrt{3}$ and
 %$\vec{x}^{(1)}\vec{x}^{(3)}=\vec{x}^{(2)}\vec{x}^{(4)}=2$, 
%energy is less
The total grain
 boundary energy in the 
 %\emph{left}
 left plot is different from the one in the right
 plot, see Figure \ref{fig:6.1}. 
%In fact, in the left figure, the length of all segment is
 %$1$. In the right figure,
 %\begin{equation*}
%  \vec{x}^{(1)}\vec{a}_\infty^{(1)}
% =
 %\vec{x}^{(2)}\vec{a}_\infty^{(1)}
 %=
%\vec{x}^{(3)}\vec{a}_\infty^{(2)}
% =
% \vec{x}^{(4)}\vec{a}_\infty^{(2)}
 %=
 %\frac23\sqrt{3},\
 %\text{and}\
 %\vec{a}_\infty^{(1)}\vec{a}_\infty^{(2)}
 %=
 %\frac13\sqrt{3}.
 %\end{equation*}
 %Therefore, the total grain boundary energy in left figure is $5$ and
 %one in the right figure is $3\sqrt{3}$.

 \begin{figure}
  \centering
  \includegraphics[height=6cm]{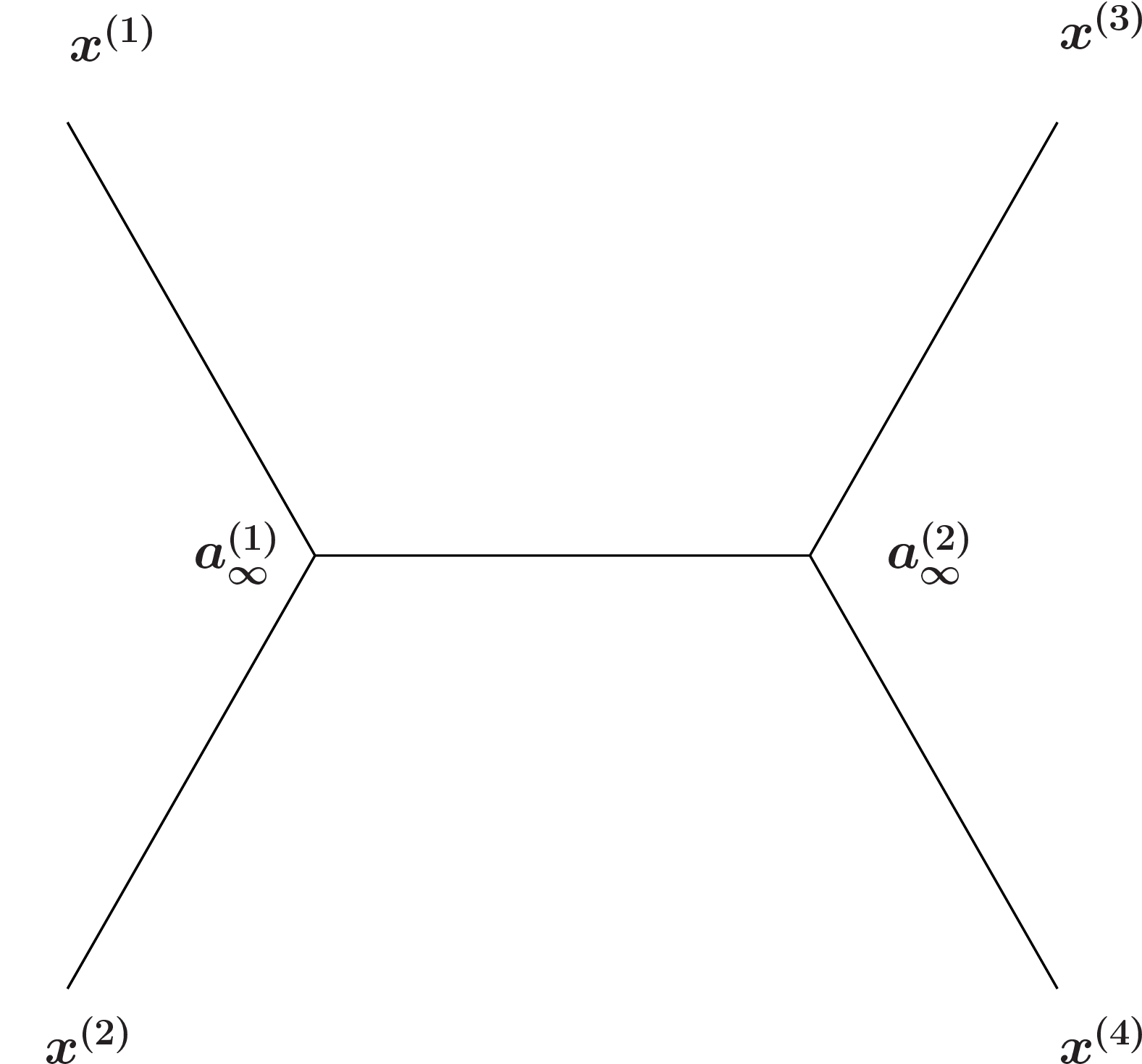}\qquad\qquad
  \includegraphics[height=6cm]{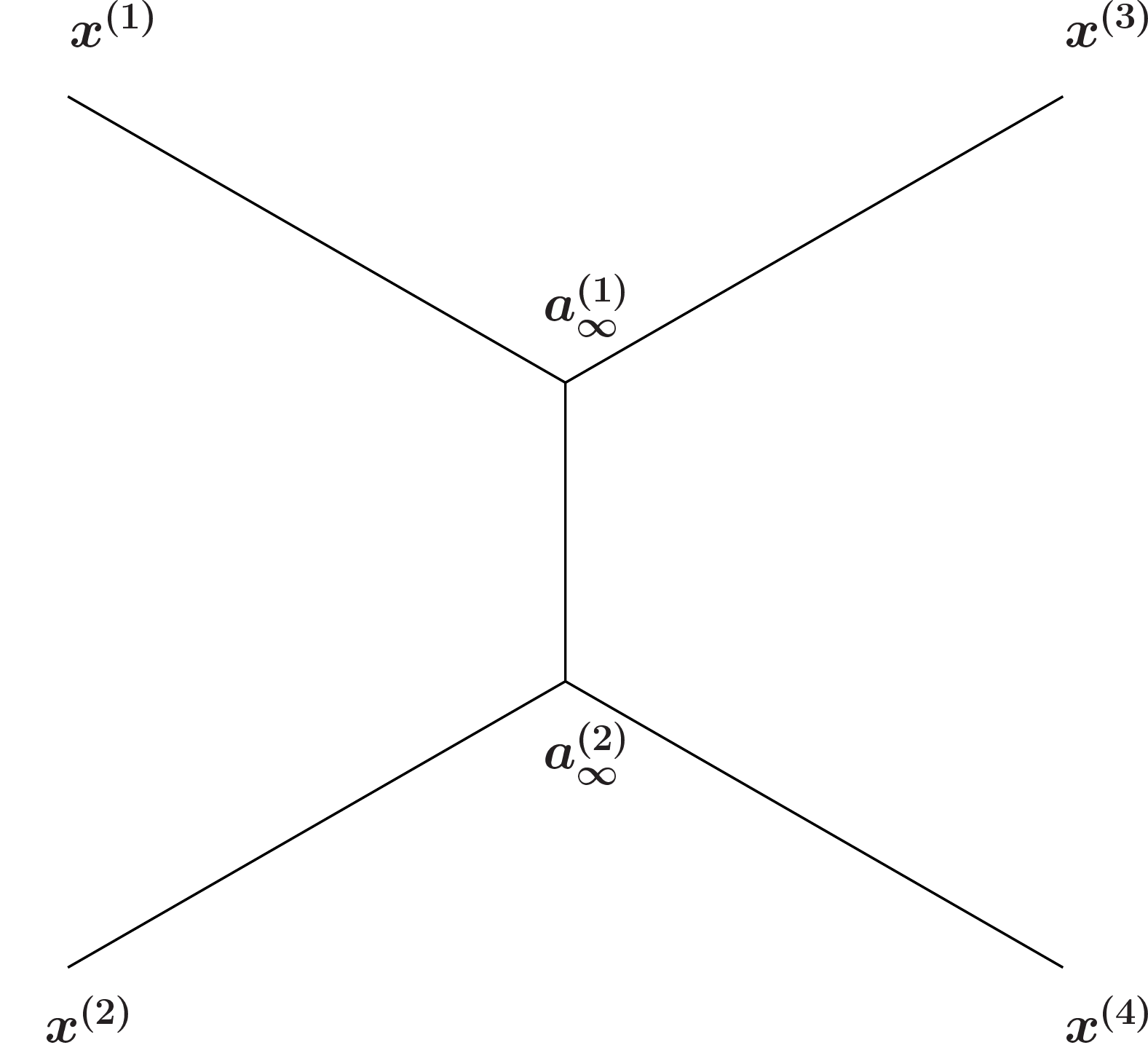}
  \caption{Even
  though, there is no misorientation effect, there are at least two equilibrium states for \eqref{eq:6.8}. 
%In general, we can not
 % guarantee whether two or more triple junctions meet each other in the
 % evolution \eqref{eq:6.7} even though the initial configuration is
 % sufficiently close to the equilibrium state.
}
  \label{fig:6.1}
 \end{figure}
 \end{example}

The asymptotics of the grain boundary networks are rather
nontrivial. Our arguments rely on the uniqueness of the equilibrium
state \eqref{eq:1.5},  but we do not know the uniqueness of solutions of
the equilibrium state for the grain boundary network
\eqref{eq:6.8}. Thus,  in general we cannot take a full limit for the
large time asymptotic behavior. Concluding the above arguments, we
have
\begin{corollary}
 In a grain boundary network \eqref{eq:6.7}, assume that the initial
 configuration is sufficiently close to an associated energy minimizing
 state \eqref{eq:6.8}.
 %, \eqref{eq:6.9}, and \eqref{eq:6.10}. 
 Then,  there
 is a global solution $(\alpha^{(k)}, \vec{a}^{(l)})$ of
 \eqref{eq:6.7}. Furthermore, there exists a time sequence
 $t_n\rightarrow\infty$ such that $(\alpha^{(k)}(t_n),
 \vec{a}^{(l)}(t_n))$ converges to an associated equilibrium
 configuration \eqref{eq:6.8}.
\end{corollary}

\section{Numerical Experiments}\label{sec:7}
\begin{figure}[hbtp]
\centering
%\begin{tabular}{cc}
%(a) & (b)\\
\vspace{-2cm}
\includegraphics[width=3.0in]{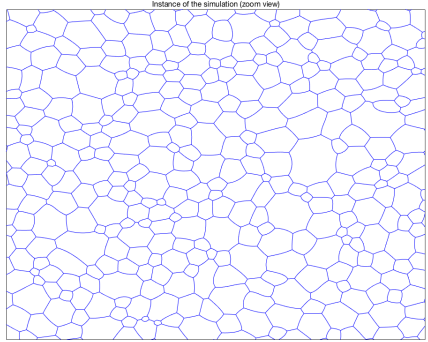}
\vspace{-2cm}
%\end{tabular}
\caption{\footnotesize Time instance from the 2D grain growth
  simulation with time-dependent orientation (zoom view). }\label{fig5}
\end{figure}
Here, we present several numerical experiments to illustrate the effects of
the dynamic orientations/misorientations (grains ``rotations'')  and of the mobility of the
triple junctions, as described in Sections \ref{sec:5}-\ref{sec:6}.  In
particular, the main goal of our numerical experiments is to
illustrate the time scales effect of the mobility of the triple junctions $\eta$
and the misorientations $\gamma$ on how the grain boundary
system decays energy and coarsens with time. For that we will study
numerically evolution of the total grain boundary energy $E(t)$
(\ref{eq:6.4e}), 
\begin{equation}\label{eq:7.1}
 E(t)
  =
  \sum_{j}
  \sigma
  (
  \Delta^{(j)}\alpha
  )
  |\Gamma_t^{(j)}|,
\end{equation}
where as before, $\Delta^{(j)}\alpha$ is a misorientation of the grain
boundary $\Gamma^{(j)}$, and  $|\Gamma^{(j)}|$ is the length of the
grain boundary. We will also consider the growth of the average area, defined as,
\begin{equation}\label{eq:7.2}
A(t)=\frac{4}{N(t)},
\end{equation}
here $4$ is the total area of the sample, and $N(t)$ is the total
number of grains at time $t$. The growth of the average area is
closely related to the coarsening rate of the grain system that
undergoes critical/disappearance events. However, it is important to
note that critical events not only include grain disappearance, but
also include  facet/grain boundary disappearance, facet interchange, splitting of unstable junctions. Further,
we will investigate the steady-state distribution of the grain boundary character
distribution (GBCD) $\rho(\Delta ^{(j)} \alpha)$.
The GBCD (in our context) is an empirical
statistical measure of the relative length (in 2D) of the grain
boundary interface with a given lattice misorientation,
\begin{eqnarray}
&\rho(\Delta ^{(j)} \alpha, t) =\mbox{ relative length of interface of
  lattice misorientation } \Delta ^{(j)} \alpha \mbox{ at time }t,\label{eq:7.3}\\
&\mbox{ normalized so that } \int_{\Omega_{\Delta ^{(j)} \alpha}} \rho
  d \Delta ^{(j)} \alpha=1, \nonumber
\end{eqnarray}
where we consider $\Omega_{\Delta ^{(j)} \alpha}=[-\frac{\pi}{4},
\frac{\pi}{4}]$ in the numerical experiments below (for planar grain
boundary network, it is reasonable to consider such range for the
misorientations). For more details, see for example \cite{DK:gbphysrev}. In all our tests below,
we compare steady-state GBCD to the stationary solution of the
Fokker-Planck equation, the  Boltzmann distribution for the grain
boundary energy density $\sigma
  (
  \Delta^{(j)}\alpha
  )$,
\begin{equation} \begin{aligned} \label{eq:7.4}
& \rho_D(\Delta\alpha^{(j)}) = \frac{1}{Z_D}e^{-\frac{\sigma(\Delta\alpha^{(j)})}{D}}, \\
& \textrm{with partition function, i.e.,normalization factor} \\
&Z_D = \int_{\Omega_{\Delta\alpha^{(j)}}}
e^{-\frac{\sigma(\Delta\alpha^{(j)})}{D}}d\Delta \alpha^{(j)},
\end{aligned}\end{equation}
\cite{DK:BEEEKT,DK:gbphysrev,MR2772123,MR3729587}. We employ
Kullback-Leibler relative entropy test to obtain a unique
``temperature-like'' parameter $D$ and to construct the corresponding
Boltzmann distribution for the steady-state GBCD as it was originally done in \cite{DK:BEEEKT,DK:gbphysrev,MR2772123,MR3729587}.
Note, that the GBCD is a primary candidate to characterize texture of the grain boundary
network, and is inversely related to the grain boundary energy density
as discovered in experiments and simulations. 
%In recent work
%\cite{DK:BEEEKT,DK:gbphysrev,MR2772123,MR3729587}, the theory was
%developed that the empirical texture statistic GBCD evolves as the
%solution to Fokker-Planck equation. 
The
reader can consult, for example, \cite{DK:BEEEKT,DK:gbphysrev,
  MR2772123, MR3729587,barmak_grain_2013} for more details about GBCD and the theory of
the GBCD. 
 In the numerical experiments in this paper, we consider the
grain boundary energy density as plotted in Figure \ref{gbend} and given below,
\[\sigma(\Delta \alpha^{(j)})=1+0.25\sin^2(2\Delta \alpha^{(j)}).\]
\par  We consider simulation of 2D grain
boundary network using further extension of the algorithm based on sharp interface
approach  \cite{MR2772123,MR3729587} (note, that in
\cite{MR2772123,MR3729587},  only Herring conditions at
triple junctions were considered, i.e., $\eta\to \infty$, and dynamic
orientations/misorientations (``rotation of grains'') was absent, i.e.,
$\gamma=0$). We recall that in the numerical scheme
we work with a variational principle. The cornerstone of the algorithm, which assures its
stability, is the discrete dissipation inequality for the total grain boundary energy that holds when
either the discrete Herring boundary condition ($\eta \to \infty$) or  discrete ``dynamic boundary
condition'' (finite mobility $\eta$ of the triple junctions,  third
equation of (\ref{eq:6.4})) is
satisfied at the triple junctions. We also recall that in the numerical algorithm we impose
Mullins theory (first equation of (\ref{eq:6.4})) as the local
evolution law for the grain boundaries (and the time scale $\mu$ is
kept finite).  For more details about computational model
based on Mullins equations (curvature driven growth), the reader can
consult, for example \cite{MR2772123,MR3729587}.

\par In all the numerical tests below we initialized our system with
$10^4$ cells/grains with normally distributed misorientation angles at
initial time $t=0$. We also assume that the final time of the
simulations $T_{\infty}$ is the time when
approximately $80\%$ of grains disappeared from the system, namely the
time when only about $2000$ cells/grains remain. The final time is
selected based on the system with no dynamic misorientations
($\gamma=0$) and with Herring condition at the triple junctions ($\eta
\to \infty$)
and, it is selected to ensure that statistically significant number of grains still remain
in the system and
the system reached its statistical stead-state. Therefore, all the
numerical results which are presented below are for the grain boundary
system that undergoes critical/disappearance events. 
\par In the first series of tests, we study the effect of the triple
junctions dynamics on the dissipation and coarsening of the system,
see Figures. \ref{fig8a}-\ref{fig11a} (finite mobility $\eta$) and
Figures \ref{fig14a}-\ref{fig15a} (Herring condition, $\eta \to
\infty$) and misorientation parameter $\gamma$ is set to $1$. We observe that for smaller
values of the mobility of the triple junctions $\eta$, the energy
decay $E(t)$ is well-approximated by an exponential decay, see Figure
\ref{fig8a} (left plot) which is consistent with the results of our
theory, see Section \ref{sec:5} and energy decay (\ref{eq:5.32}), even
though, the theoretical results are obtained under assumption of no
critical events.  In comparison, we also present fit to a power law decaying
function, see Figure \ref{fig8a} (right plot). The power law function
does not seem to give as good approximation in this case, due to the
appearance 
of the negative term in the fitted power function. However, for larger values of $\eta$, Figure \ref{fig10a} (left
plot) and for Herring condition
($\eta\to \infty$), Figure \ref{fig14a} (see also Figures
\ref{fig12a} and
\ref{fig16a} with different values of $\gamma$) (left plots),  we obtain that the total grain boundary energy does not
follow exponential decay,  and the sum of exponential functions or power law
functions give better description of the energy decay, see Figure \ref{fig14a} (see also Figures
\ref{fig12a} and
\ref{fig16a} with different values of $\gamma$). Note
also, that the numerically observed energy decay rates increase with the
mobility $\eta$  of the triple junctions which is also consistent with
the developed theory, see Sections \ref{sec:5}-\ref{sec:6}. In
addition, we observe that the average area grows as quadratic function in
time for the finite mobility $\eta$ of the triple junctions, Figures
\ref{fig9a} and \ref{fig11a} (left plots) but it exhibits a linear
growth with $\eta\to \infty$, Herring condition at the triple
junctions, Figure  \ref{fig15a} (see also Figures \ref{fig13a} and
\ref{fig17a} with different values of $\gamma$)
(left plots).  We also observe that the coarsening rate of grain
systems slows down with the
smaller $\eta$,  see Figures for the growth of the average area
\ref{fig9a}, \ref{fig11a},  \ref{fig15a} (see also Figures \ref{fig13a} and
\ref{fig17a} with different values of $\gamma$) (left plots). The
observations about the growth of the average area and the coarsening
rate can be explained by  noting that with Herring condition at the
triple junctions ($\eta \to \infty$), the grain growth is driven mainly
 by the kinetics of grain boundaries, but not the triple
 junctions. And,  the  von Neumann-Mullins $n-6$ rule for the area of
 $n$-sided grain is satisfied approximately (due to anisotropy) in that case. With the finite mobility $\eta$ of the triple junctions,
 triple junctions dynamics play an important role on the grain
 growth. Hence, finite mobility of the triple junctions can result in the
 growth of ``smaller'' grains (with sides less than $6$) and, at the
 same time, it can result in the disappearance of the ``larger''
 grains (with sides greater than $6$), namely, the von Neumann-Mullins
 $n-6$ rule will not be valid in the same way anymore. We also note that the energy
 decay in our numerical tests is consistent with the growth of the average area.
Moreover, we observe that dynamics of the triple junctions 
show some effect on the steady-state GBCD, in particular, slight
decrease in the diffusion coefficient/ in the ``temperature'' like
parameter $D$ with increase of $\eta$,  but not a significant change,  Figures
\ref{fig9a}, \ref{fig11a} and \ref{fig15a}
(right plots), (note,  the ``temperature'' like
parameter $D$   accounts for various critical events--grains
disappearance, facet/grain boundary disappearance, facet interchange, splitting of unstable junctions).
%We also note that  the ``temperature'' like
%parameter $D$  is related to various critical events in the system
%(grain disappearance, facet/grain boundary  disappearance, facet
%interchange, splitting of unstable junctions) 
\begin{figure}[hbtp]
\centering
%\begin{tabular}{cc}
%(a) & (b)\\
\vspace{-2.cm}
\includegraphics[width=3.1in]{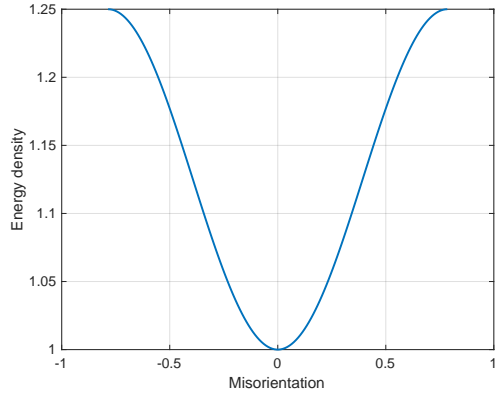}
\vspace{-2.cm}
\caption{\footnotesize Grain boundary energy density function $\sigma(\Delta \alpha)$.}\label{gbend}
\end{figure}
\begin{figure}[hbtp]
\centering
%\begin{tabular}{cc}
%(a) & (b)\\
\vspace{-2.cm}
\includegraphics[width=3.1in]{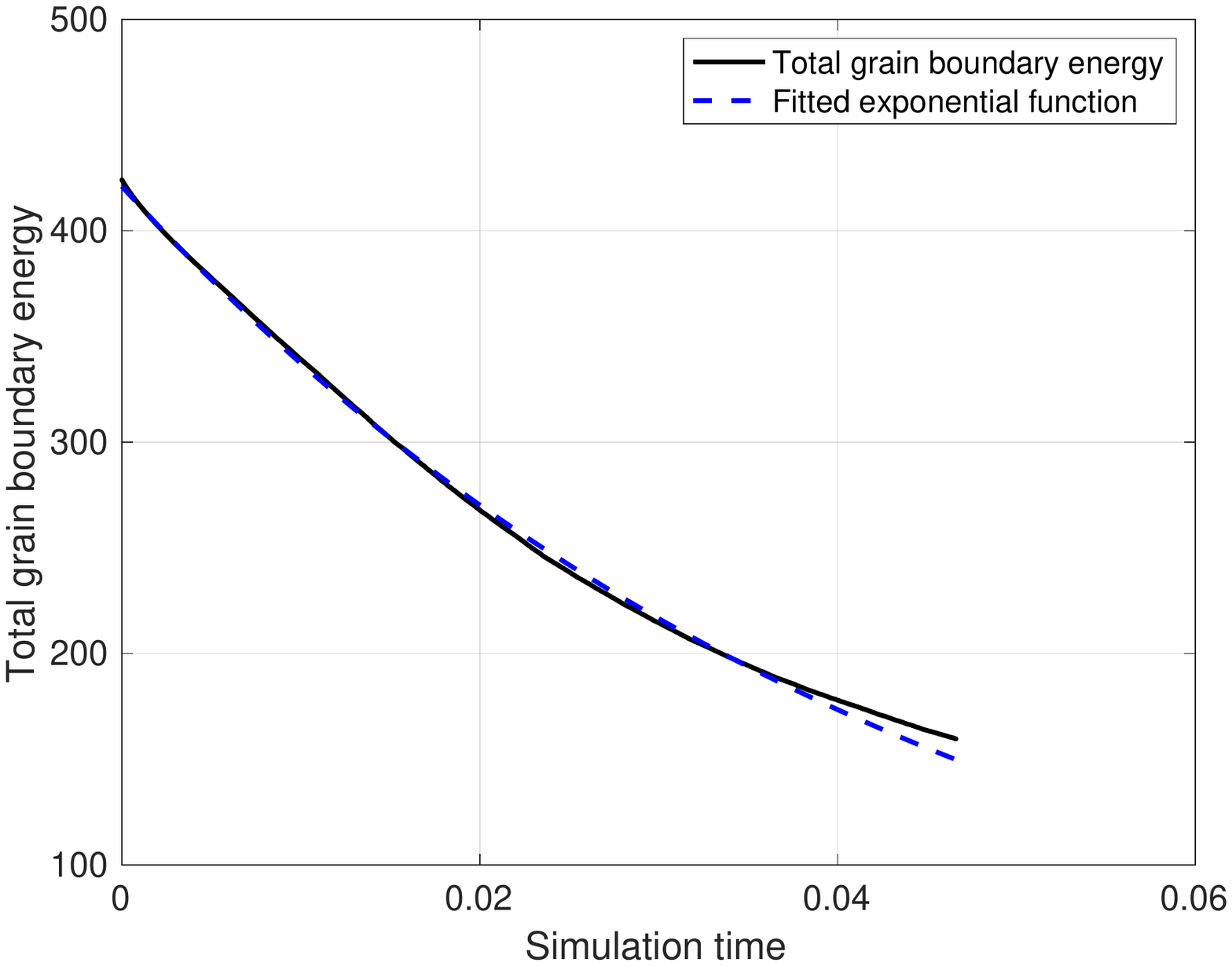}
\includegraphics[width=3.1in]{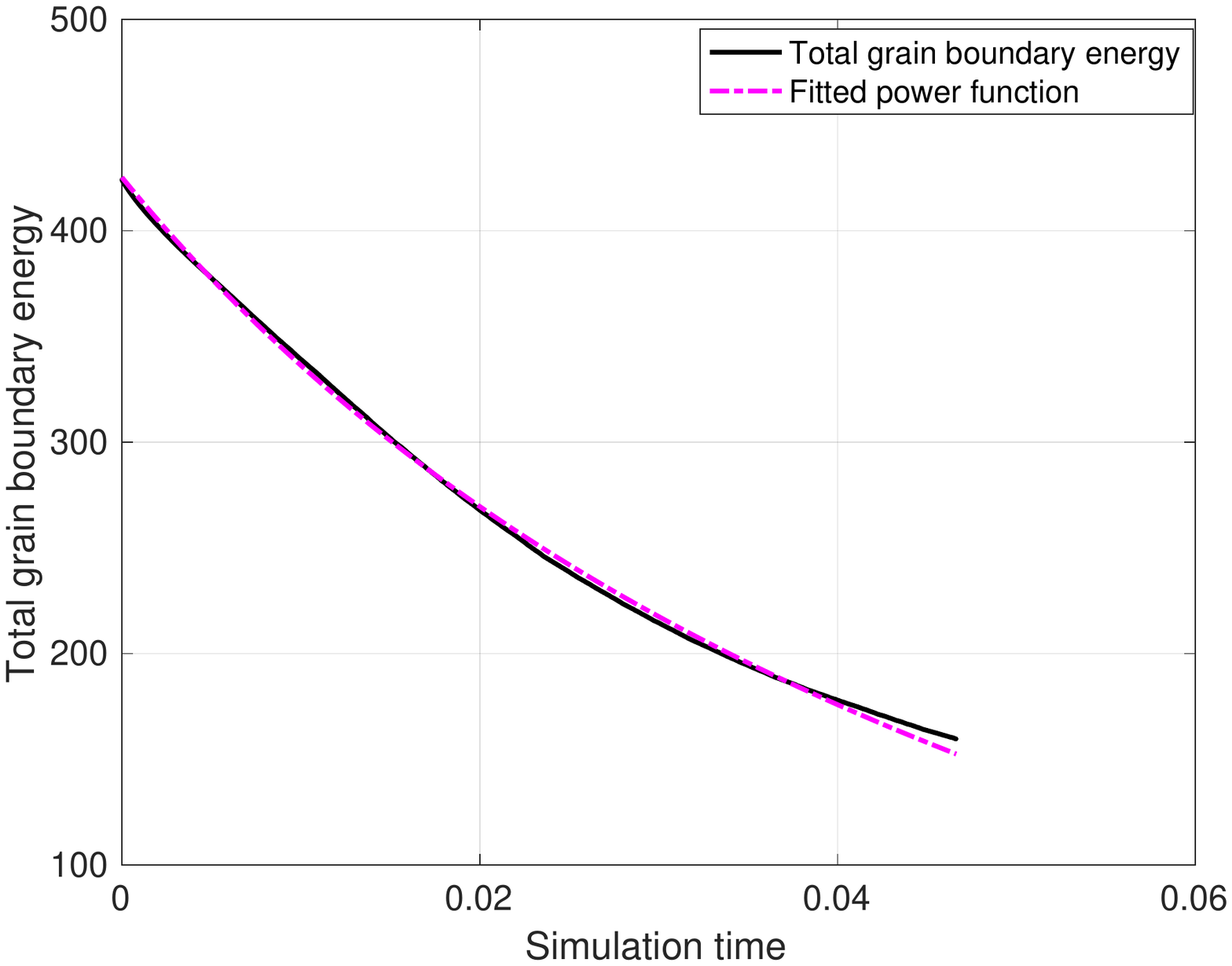}
\vspace{-2.cm}
%\end{tabular}
\caption{\footnotesize One run of $2$D trial with $10000$ initial
  grains: {\it (a) Left plot,} Total grain boundary energy plot
  (solid black) versus fitted  exponential decaying function
  $y(t)=421\exp(-22.16t)$ (dashed blue); {\it (b) Right
    plot}, Total grain boundary energy plot
  (solid black) versus fitted power law decaying function
  $y_1(t)=-201.4+626.67(1.0+16.53t)^{-1}$ (dashed magenta).
 Mobility of the triple
junctions is 
$\eta=10$ and the misorientation parameter $\gamma=1$.}\label{fig8a}
\end{figure}

\begin{figure}[hbtp]
\centering
%\begin{tabular}{cc}
%(a) & (b)\\
\vspace{-2.cm}
\includegraphics[width=3.1in]{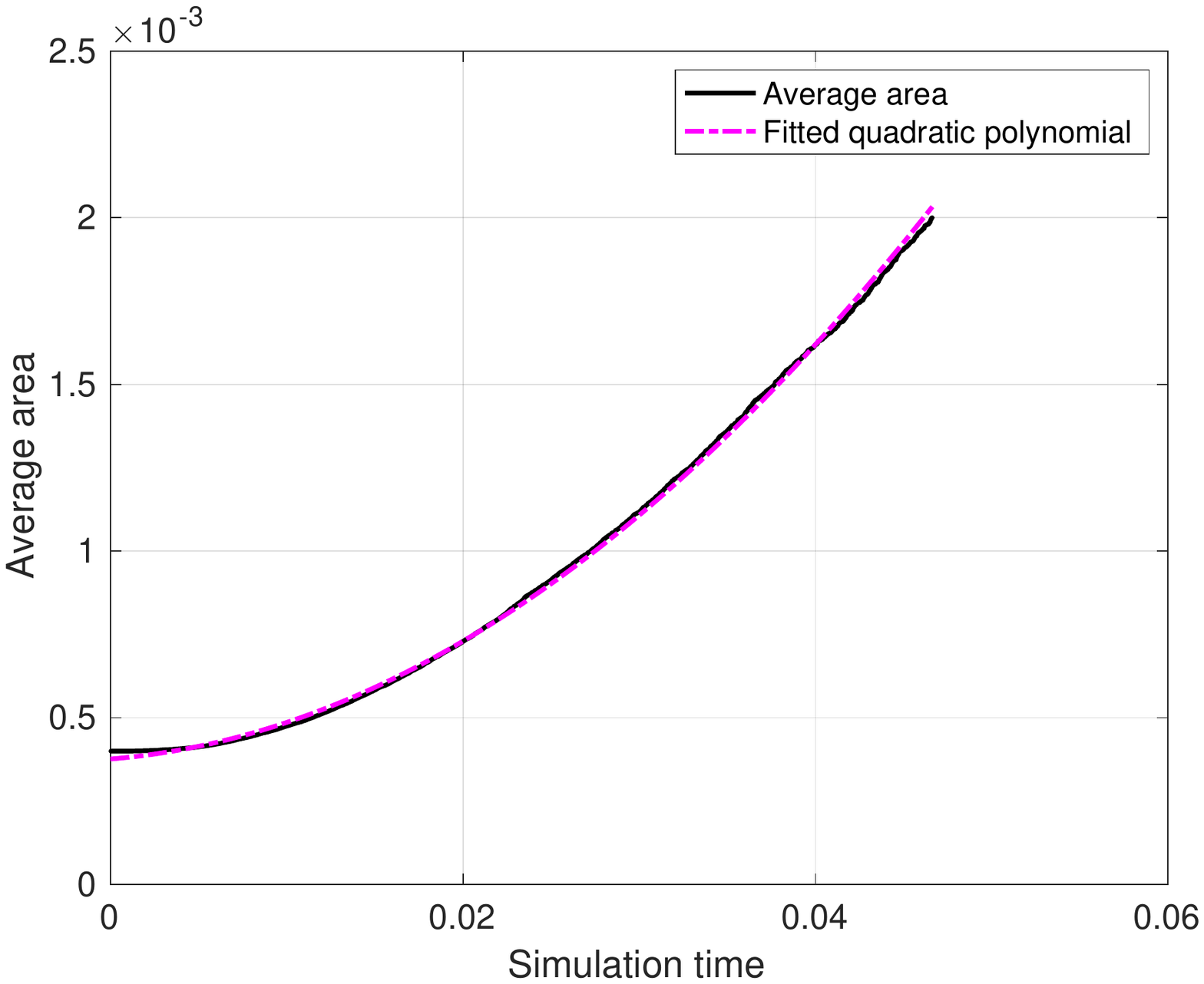}
\includegraphics[width=3.1in]{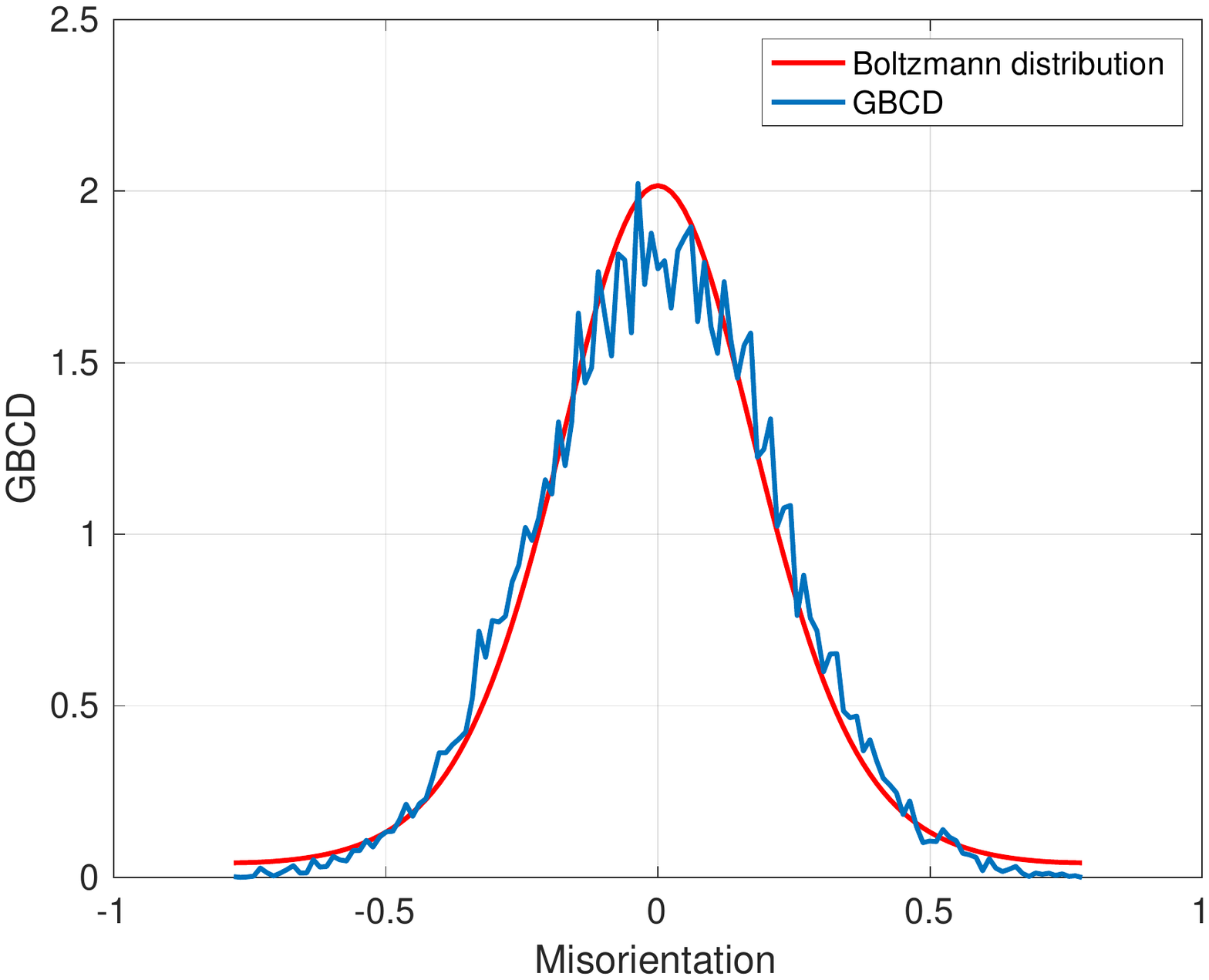}
\vspace{-2.cm}
%\end{tabular}
\caption{\footnotesize {\it (a) Left plot,} One run of $2$D trial with $10000$ initial
  grains: Growth of the average area of the
  grains (solid black) versus fitted  quadratic polynomial function
  $y(t)=0.6704t^2+0.004265t+0.0003764$ (dashed magenta).  
%The growth of the average area is
%consistent with the energy decay, Fig.~\ref{fig8a};
 {\it (b) Right plot},
  steady-state GBCD (blue curve) averaged over 3 runs of $2$D trials with $10000$ initial
  grains versus Boltzmann distribution with ``temperature''-
$D\approx 0.0650$ (red curve). Mobility of triple junctions is $\eta=10$ the misorientation parameter $\gamma=1$.}\label{fig9a}
\end{figure}

\begin{figure}[hbtp]
\centering
%\begin{tabular}{cc}
%(a) & (b)\\
\vspace{-2.cm}
\includegraphics[width=3.1in]{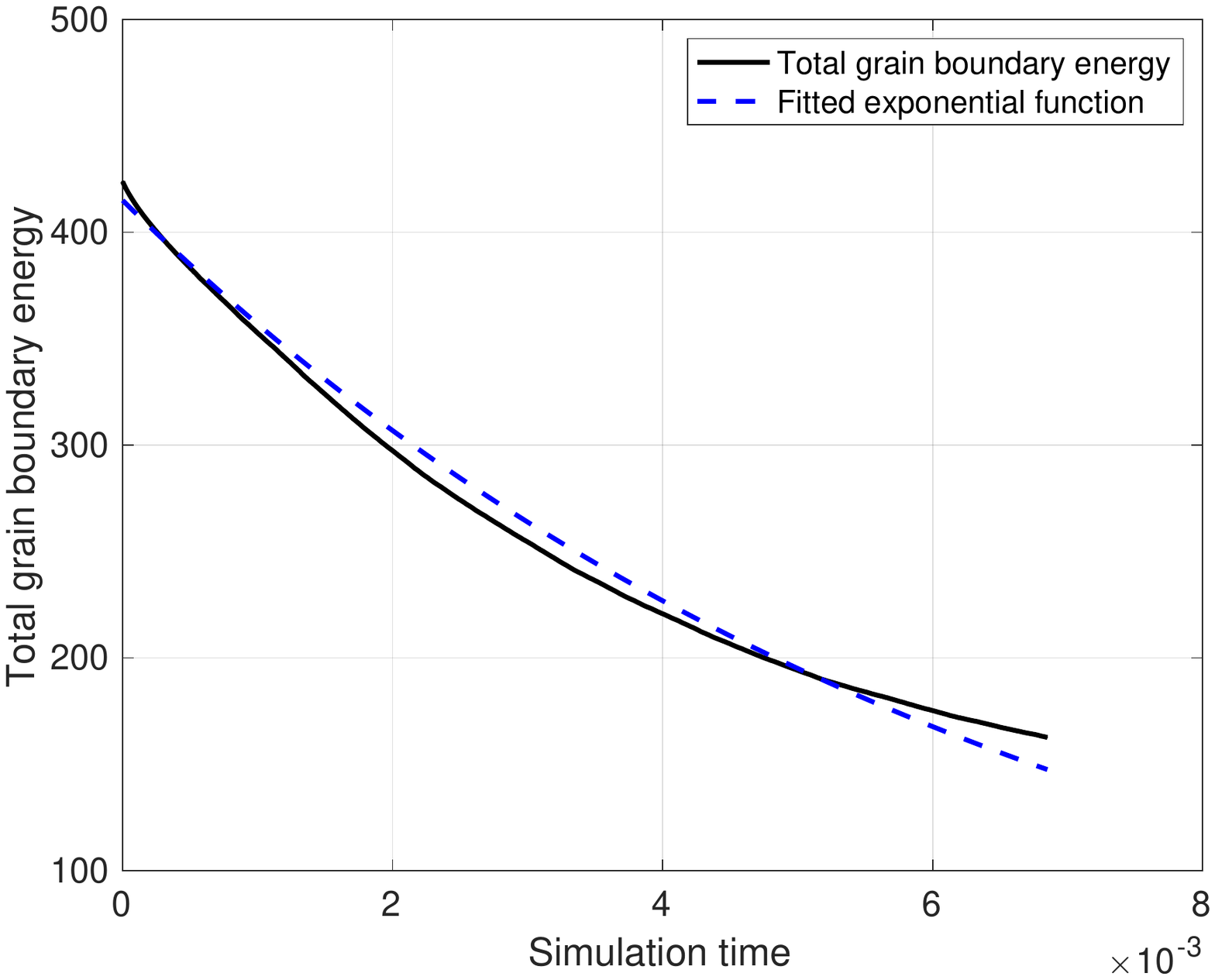}
\includegraphics[width=3.1in]{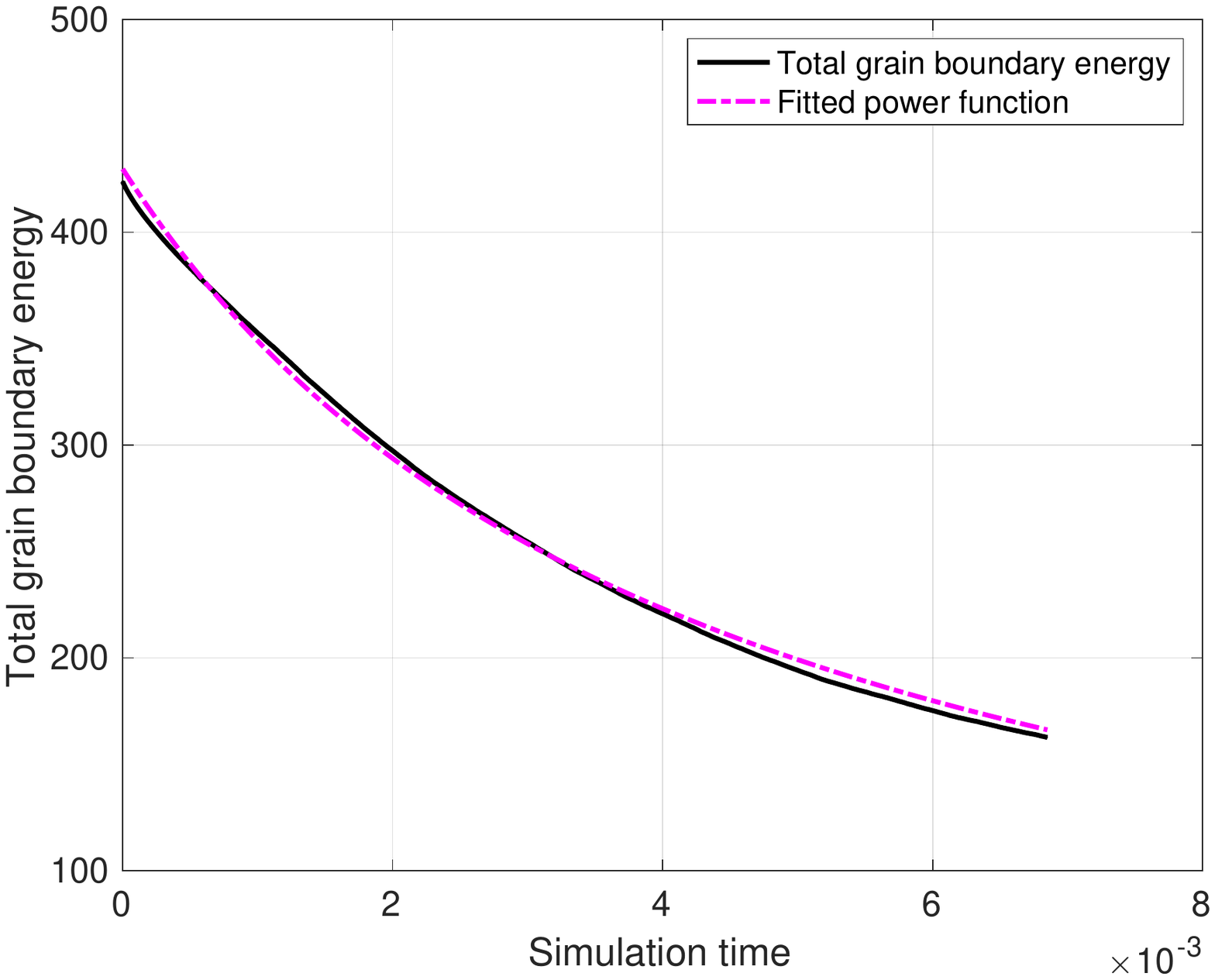}
\vspace{-2.cm}
%\end{tabular}
\caption{\footnotesize One run of $2$D trial with $10000$ initial
  grains: {\it (a) Left plot,} Total grain boundary energy plot
  (solid black) versus fitted  exponential decaying function
  $y(t)=415\exp(-151t)$ (dashed blue); {\it (b) Right
    plot}, Total grain boundary energy plot
  (solid black) versus fitted power law decaying function
  $y_1(t)=429.8286(1.0+231.5887t)^{-1}$ (dashed magenta).
 Mobility of the triple
junctions is 
$\eta=100$ and the misorientation parameter $\gamma=1$.}\label{fig10a}
\end{figure}

\begin{figure}[hbtp]
\centering
%\begin{tabular}{cc}
%(a) & (b)\\
\vspace{-2.cm}
\includegraphics[width=3.1in]{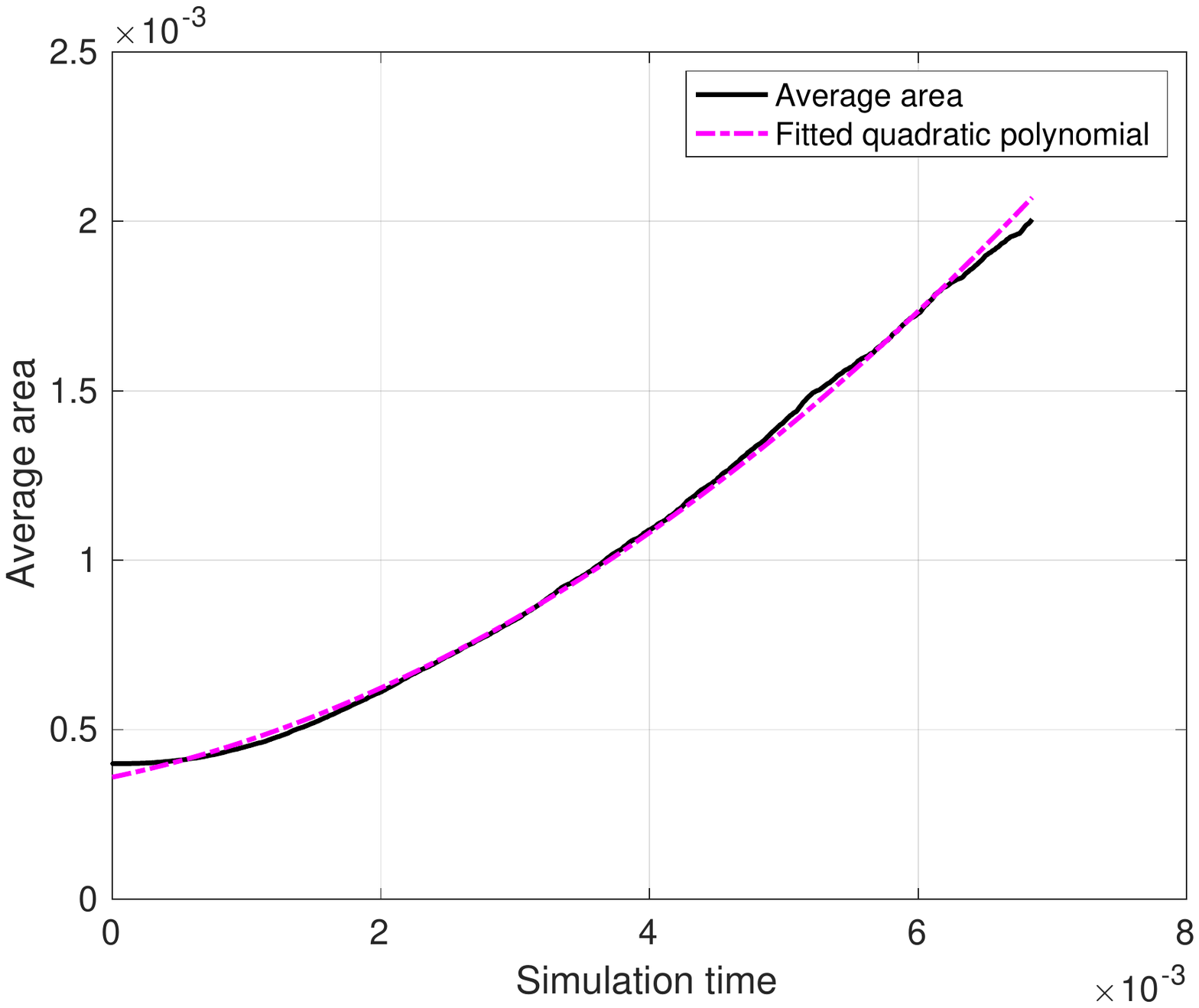}
\includegraphics[width=3.1in]{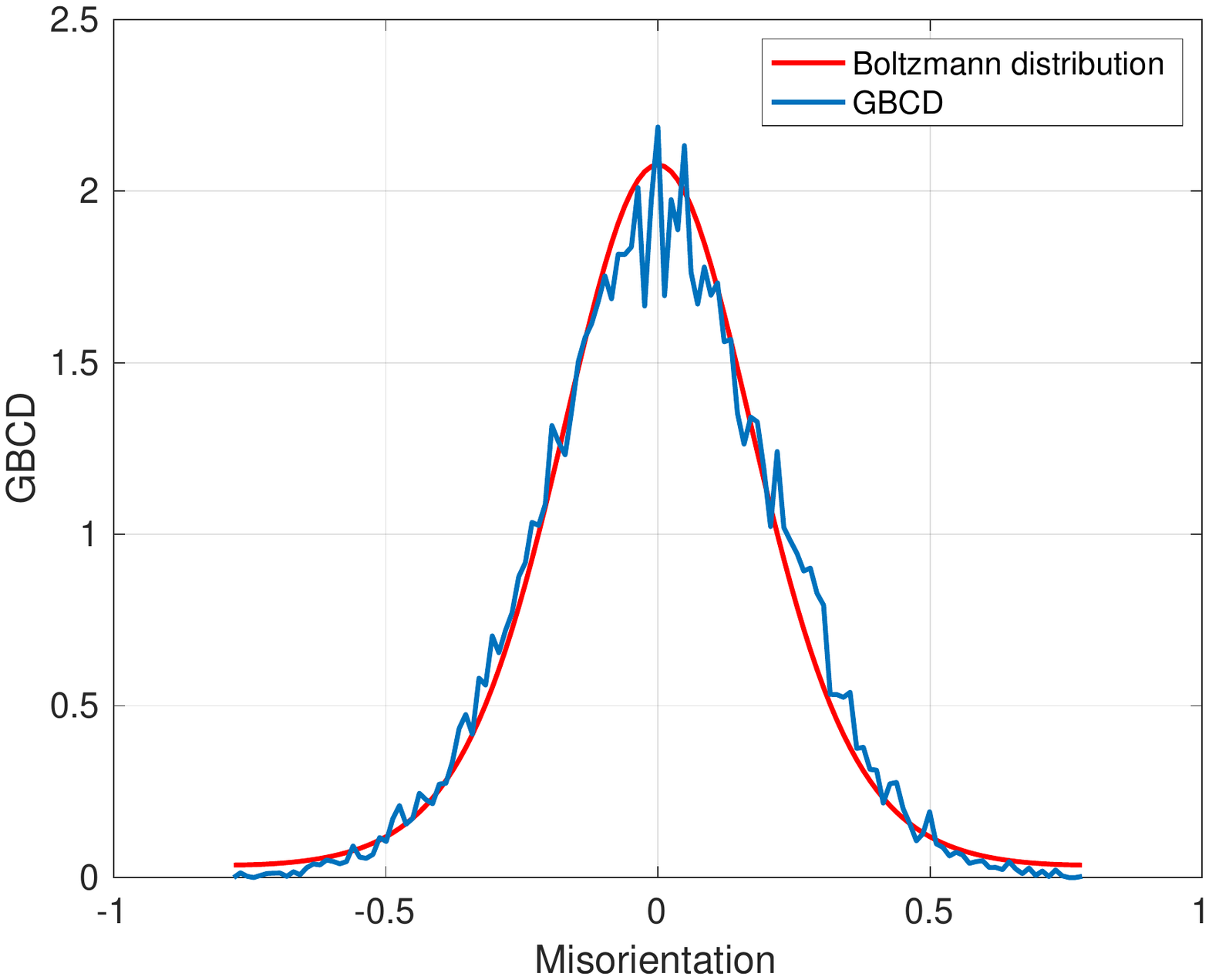}
\vspace{-2.cm}
%\end{tabular}
\caption{\footnotesize {\it (a) Left plot,} One run of $2$D trial with $10000$ initial
  grains: Growth of the average area of the
  grains (solid black) versus fitted  quadratic polynomial function
  $y(t)=24.34t^2+0.083t+0.00036$ (dashed magenta).  
%The growth of the average area is
%consistent with the energy decay, Fig.~\ref{fig8a};
 {\it (b) Right plot},
  steady-state GBCD (blue curve) averaged over 3 runs of $2$D trials with $10000$ initial
  grains versus Boltzmann distribution with ``temperature''-
$D\approx 0.0618$ (red curve). Mobility of triple junctions is
$\eta=100$ and the misorientation parameter $\gamma=1$.}\label{fig11a}
\end{figure}

\begin{figure}[hbtp]
\centering
%\begin{tabular}{cc}
%(a) & (b)\\
\vspace{-2.cm}
\includegraphics[width=3.1in]{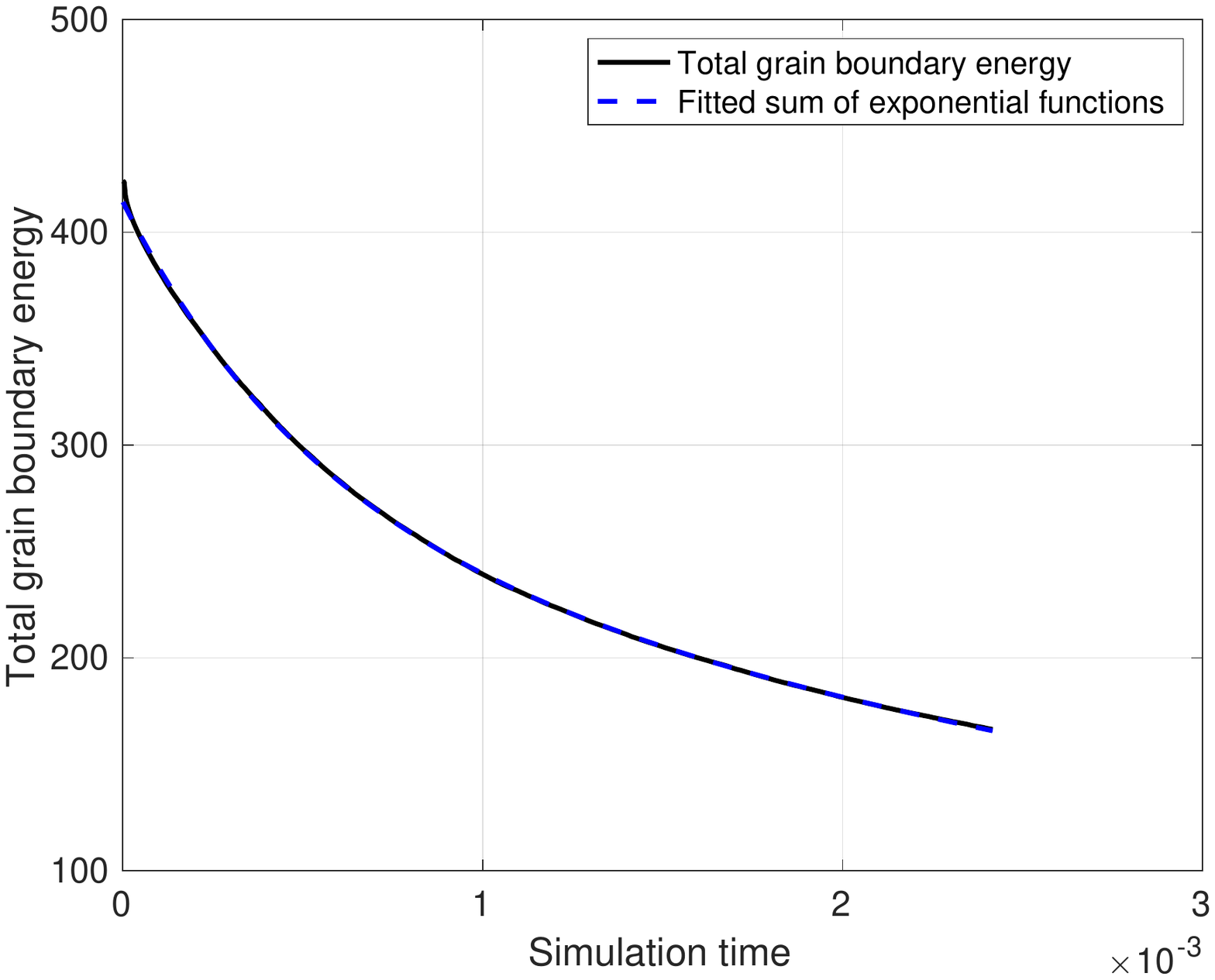}
\includegraphics[width=3.1in]{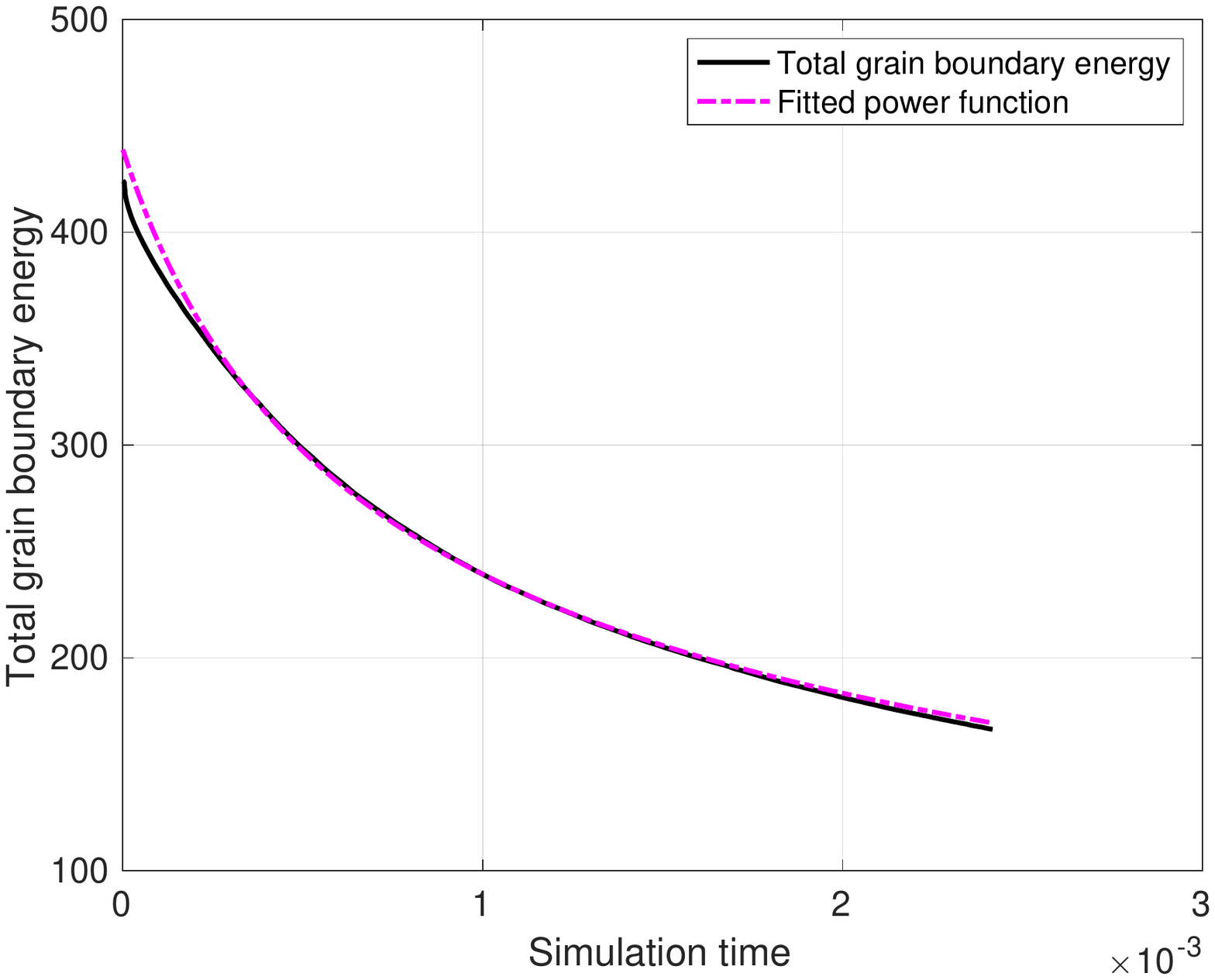}
\vspace{-2.cm}
%\end{tabular}
\caption{\footnotesize One run of $2$D trial with $10000$ initial
  grains: {\it (a) Left plot,} Total grain boundary energy plot
  (solid black) versus fitted  sum of exponential decaying functions
  $y(t)=262.6\exp(-194.6t)+151.9\exp(-1868t)$ (dashed blue); {\it (b) Right
    plot}, Total grain boundary energy plot
  (solid black) versus fitted power law decaying function
  $y_1(t)=439.3001(1.0+2369.2t)^{-0.5}$ (dashed magenta). Herring Condition is
imposed at the triple junctions $\eta\to \infty$
 and no ``dynamic'' misorientation.}\label{fig12a}
\end{figure}

\begin{figure}[hbtp]
\centering
%\begin{tabular}{cc}
%(a) & (b)\\
\vspace{-2.cm}
\includegraphics[width=3.1in]{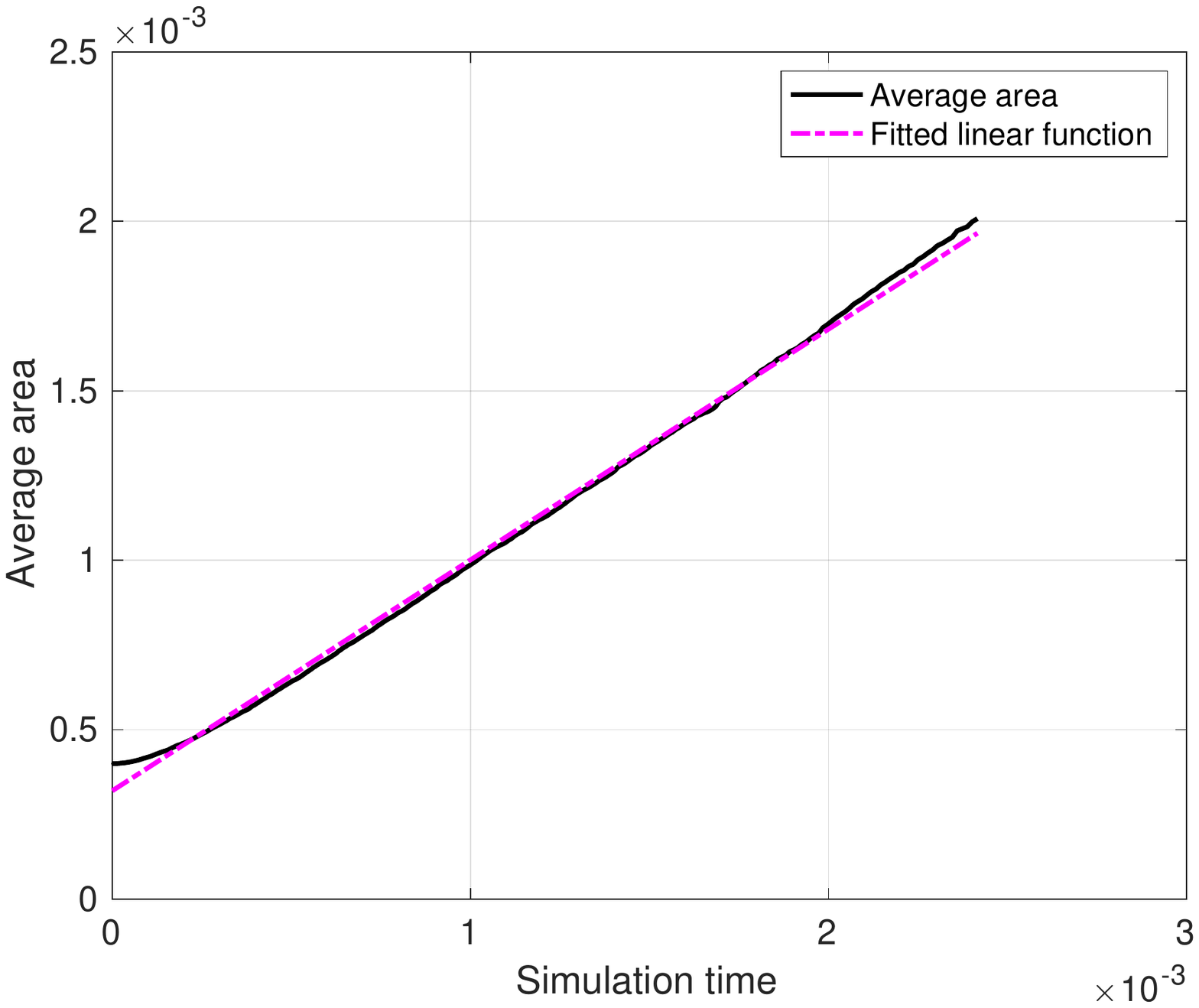}
\includegraphics[width=3.1in]{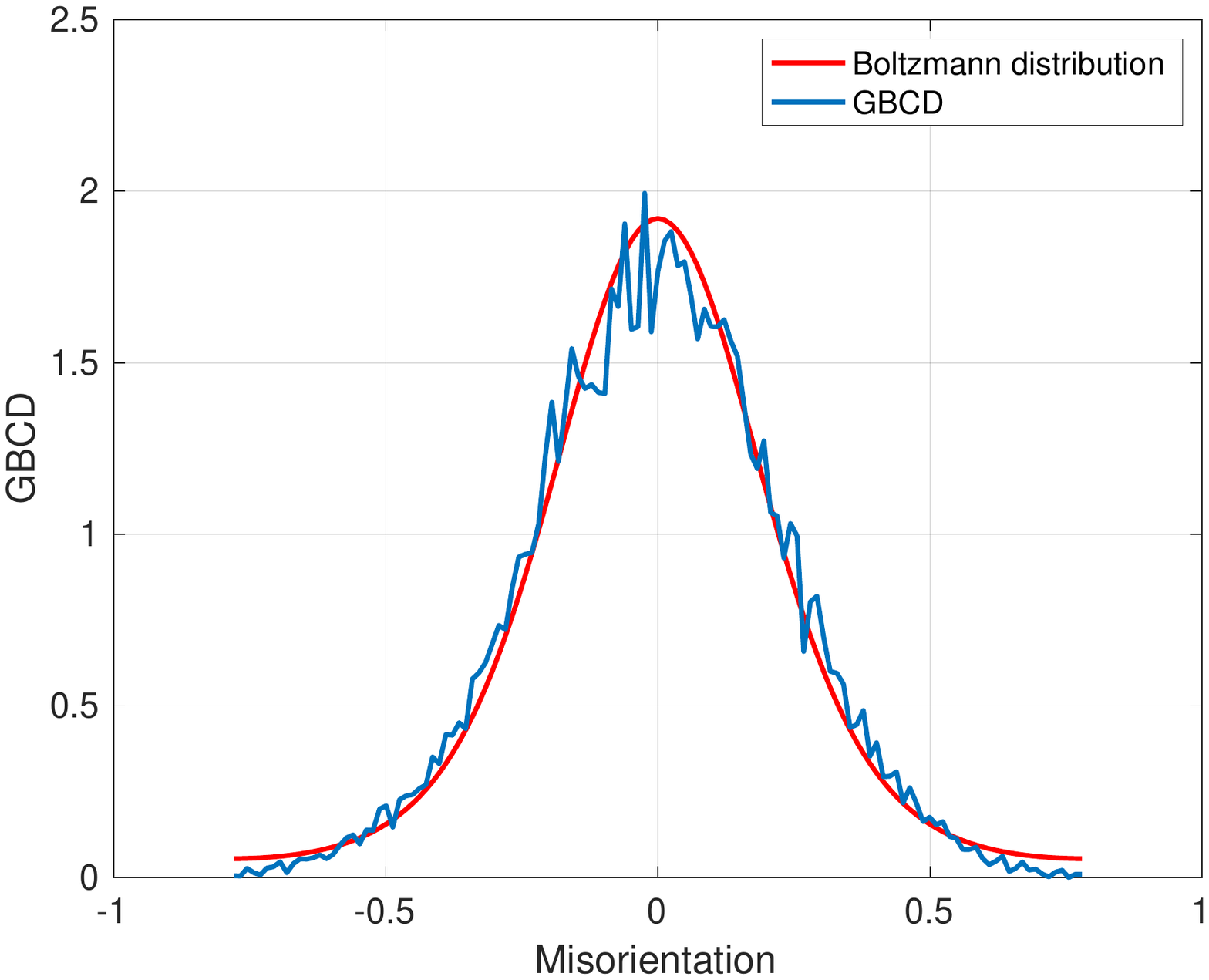}
\vspace{-2.cm}
%\end{tabular}
\caption{\footnotesize {\it (a) Left plot,} One run of $2$D trial with $10000$ initial
  grains: Growth of the average area of the
  grains (solid black) versus fitted  linear function
  $y(t)=0.6811t+0.0003198$ (dashed magenta).  
%The growth of the average area is
%consistent with the energy decay, Fig.~\ref{fig8a};
 {\it (b) Right plot},
  steady-state GBCD (blue curve) averaged over 3 runs of $2$D trials with $10000$ initial
  grains versus Boltzmann distribution with ``temperature''-
$D\approx 0.0704$ (red curve).  Herring Condition is
imposed at the triple junctions $\eta\to \infty$
 and no ``dynamic'' misorientation.}\label{fig13a}
\end{figure}

\begin{figure}[hbtp]
\centering
%\begin{tabular}{cc}
%(a) & (b)\\
\vspace{-2.cm}
\includegraphics[width=3.1in]{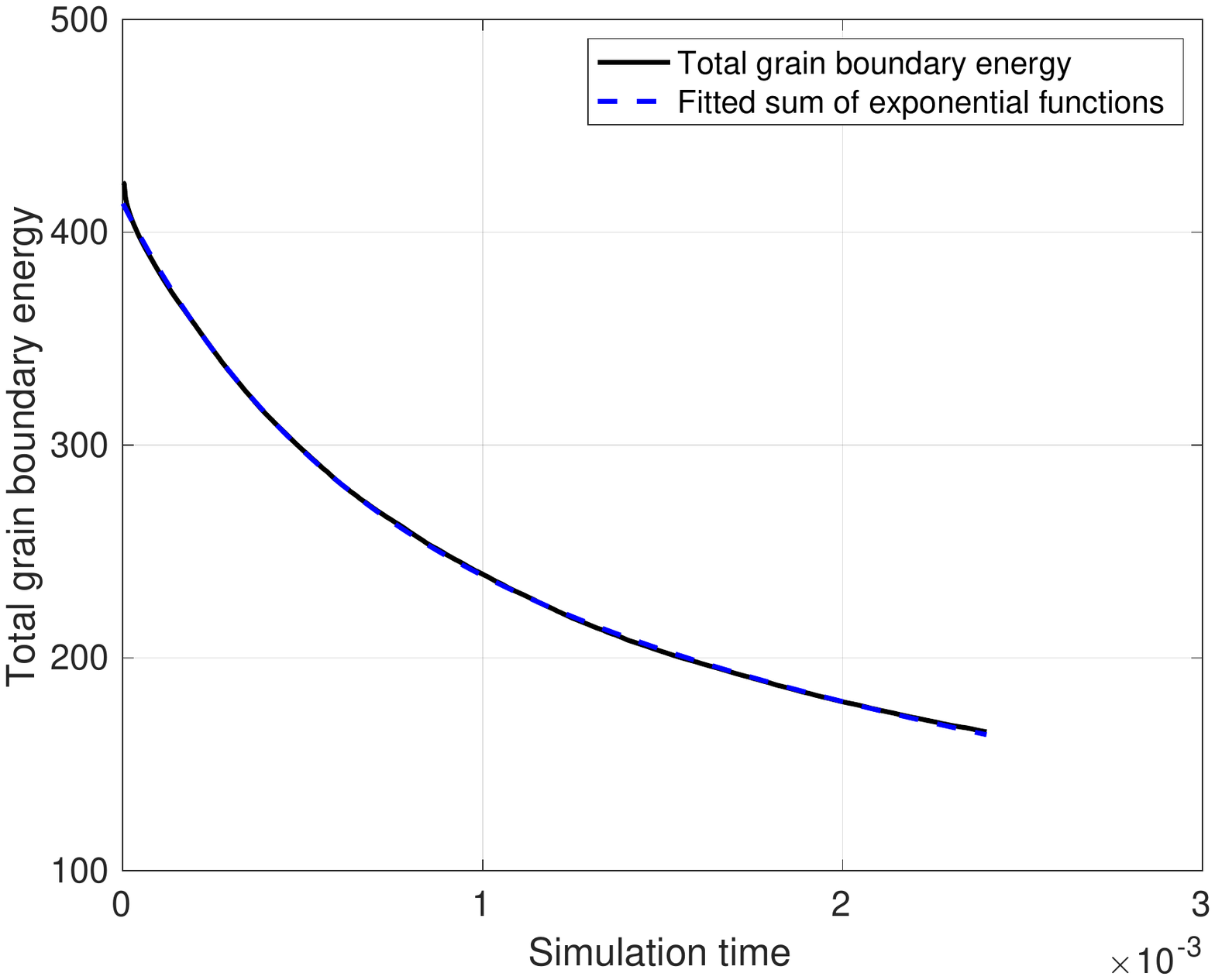}
\includegraphics[width=3.1in]{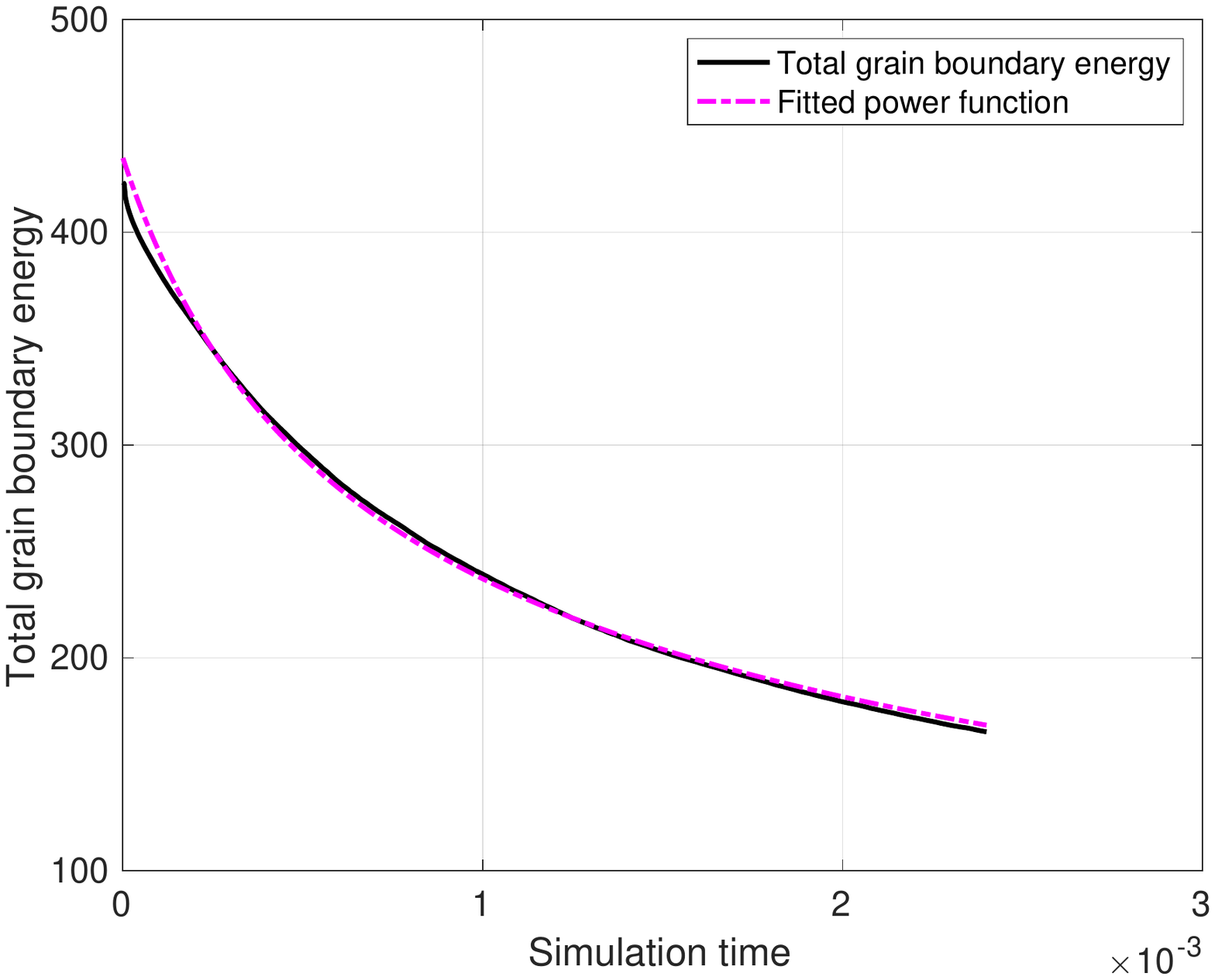}
\vspace{-2.cm}
%\end{tabular}
\caption{\footnotesize One run of $2$D trial with $10000$ initial
  grains: {\it (a) Left plot,} Total grain boundary energy plot
  (solid black) versus fitted  sum of exponential decaying functions
  $y(t)=263.5\exp(-202.3t)+150.3\exp(-1859t)$ (dashed blue); {\it (b) Right
    plot}, Total grain boundary energy plot
  (solid black) versus fitted power law decaying function
  $y_1(t)=435.3778(1.0+2369.2t)^{-0.5}$ (dashed magenta). Herring Condition is
imposed at the triple junctions $\eta\to \infty$
 and the misorientation parameter $\gamma=1$.}\label{fig14a}
\end{figure}

\begin{figure}[hbtp]
\centering
%\begin{tabular}{cc}
%(a) & (b)\\
\vspace{-2.cm}
\includegraphics[width=3.1in]{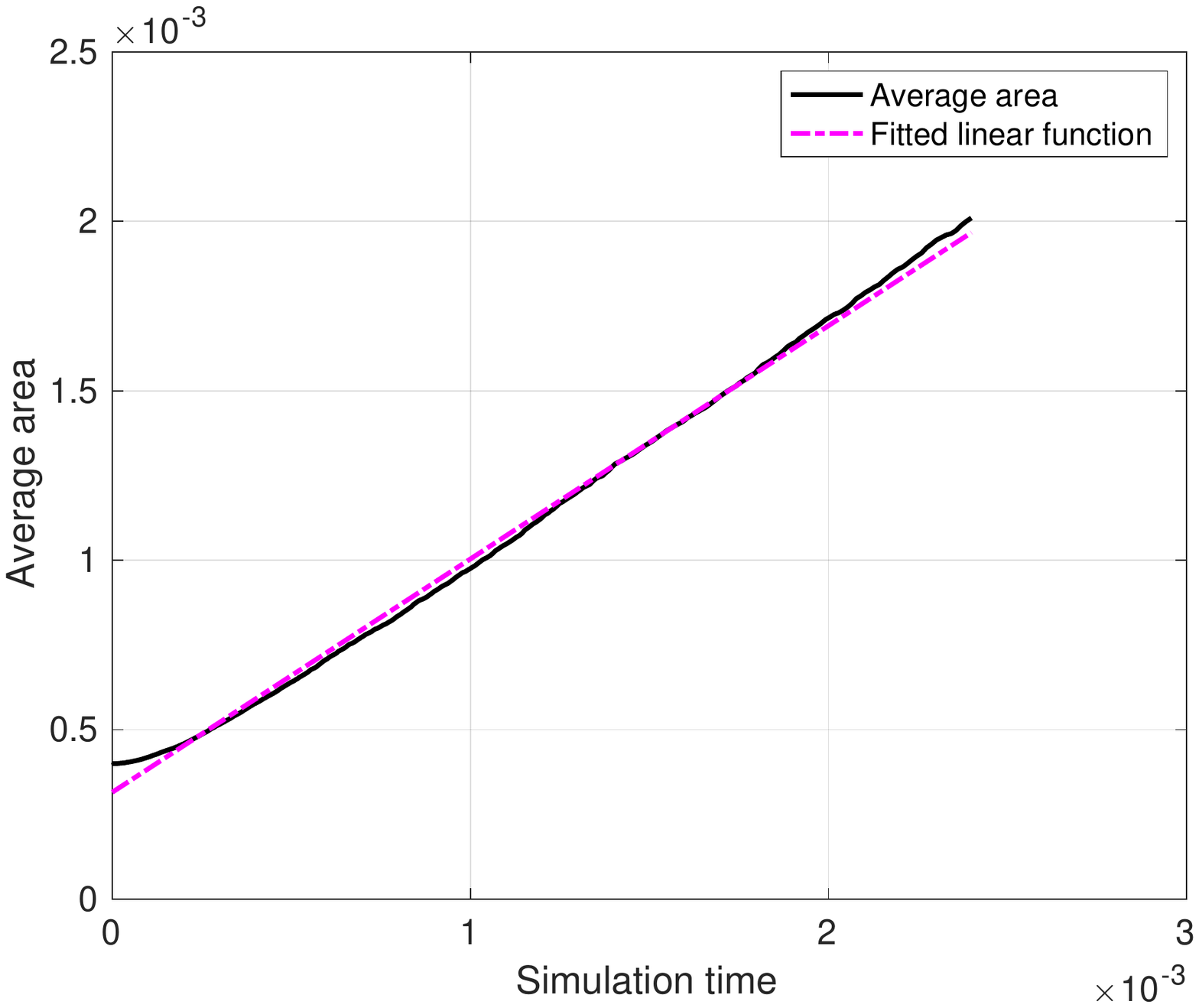}
\includegraphics[width=3.1in]{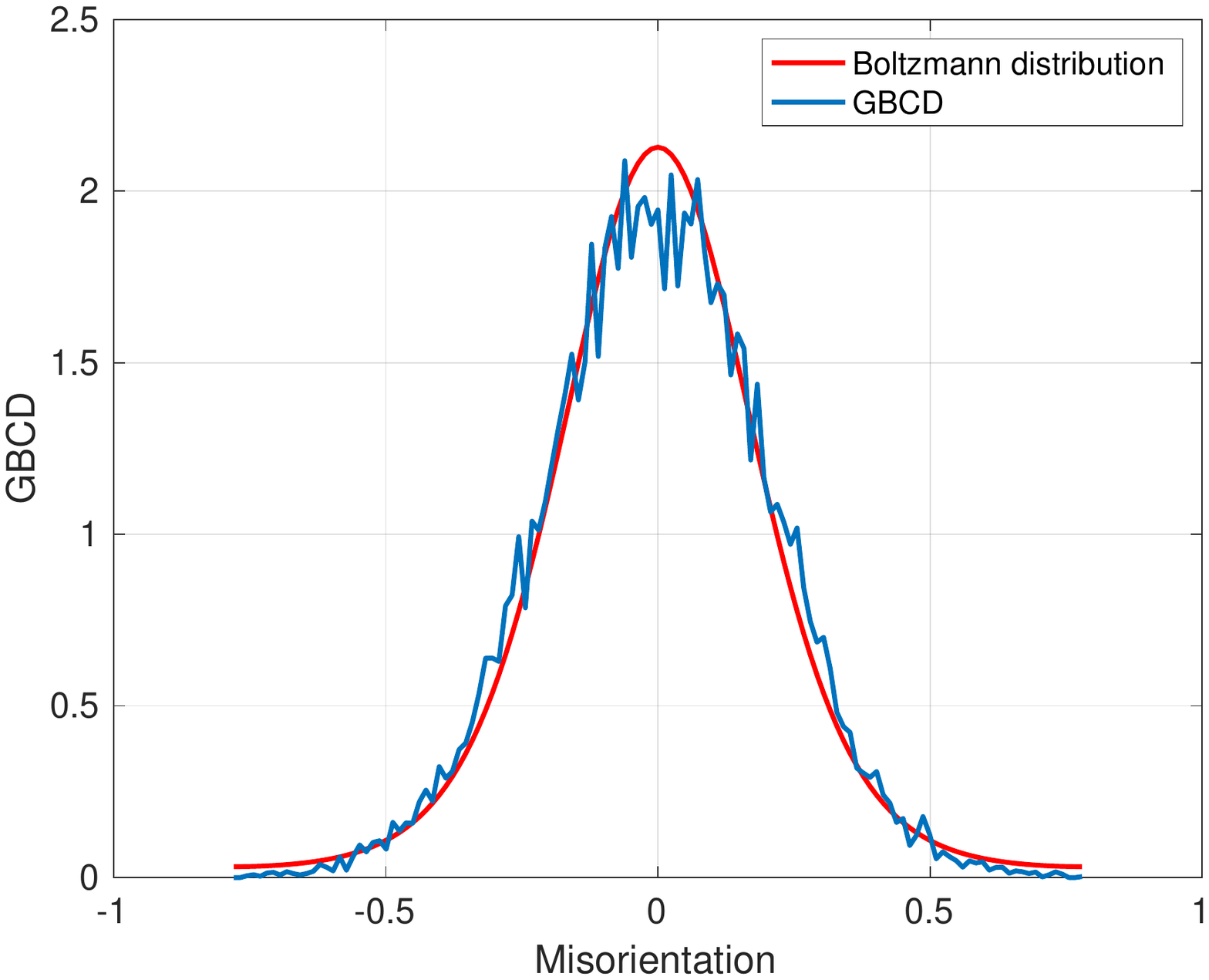}
\vspace{-2.cm}
%\end{tabular}
\caption{\footnotesize {\it (a) Left plot,} One run of $2$D trial with $10000$ initial
  grains: Growth of the average area of the
  grains (solid black) versus fitted  linear function
  $y(t)=0.6884t+0.0003154$ (dashed magenta).  
%The growth of the average area is
%consistent with the energy decay, Fig.~\ref{fig8a};
 {\it (b) Right plot},
  steady-state GBCD (blue curve) averaged over 3 runs of $2$D trials with $10000$ initial
  grains versus Boltzmann distribution with ``temperature''-
$D\approx 0.0594$ (red curve).  Herring Condition is
imposed at the triple junctions $\eta\to \infty$
 and the misorientation parameter $\gamma=1$.}\label{fig15a}
\end{figure}

\begin{figure}[hbtp]
\centering
%\begin{tabular}{cc}
%(a) & (b)\\
\vspace{-2.cm}
\includegraphics[width=3.1in]{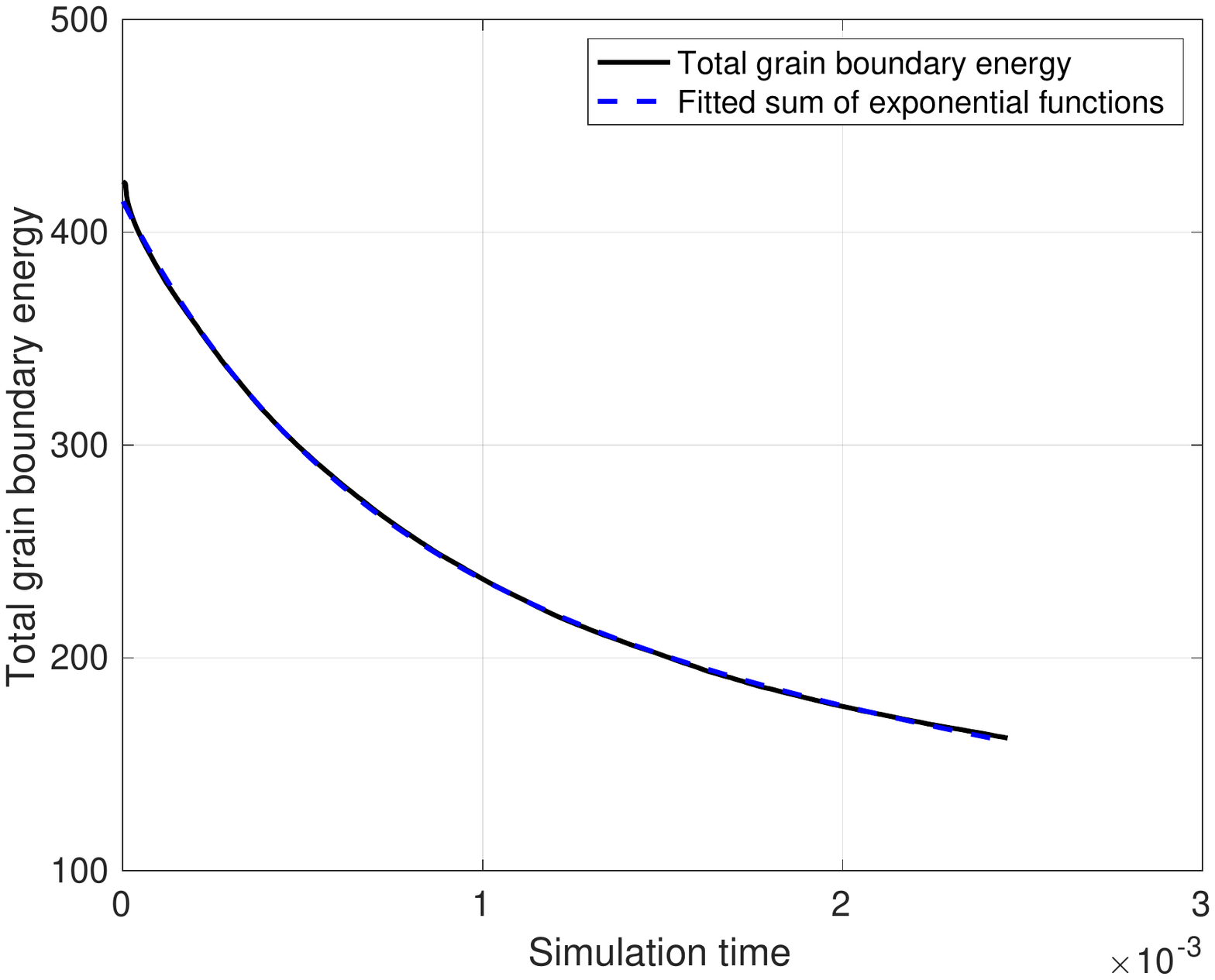}
\includegraphics[width=3.1in]{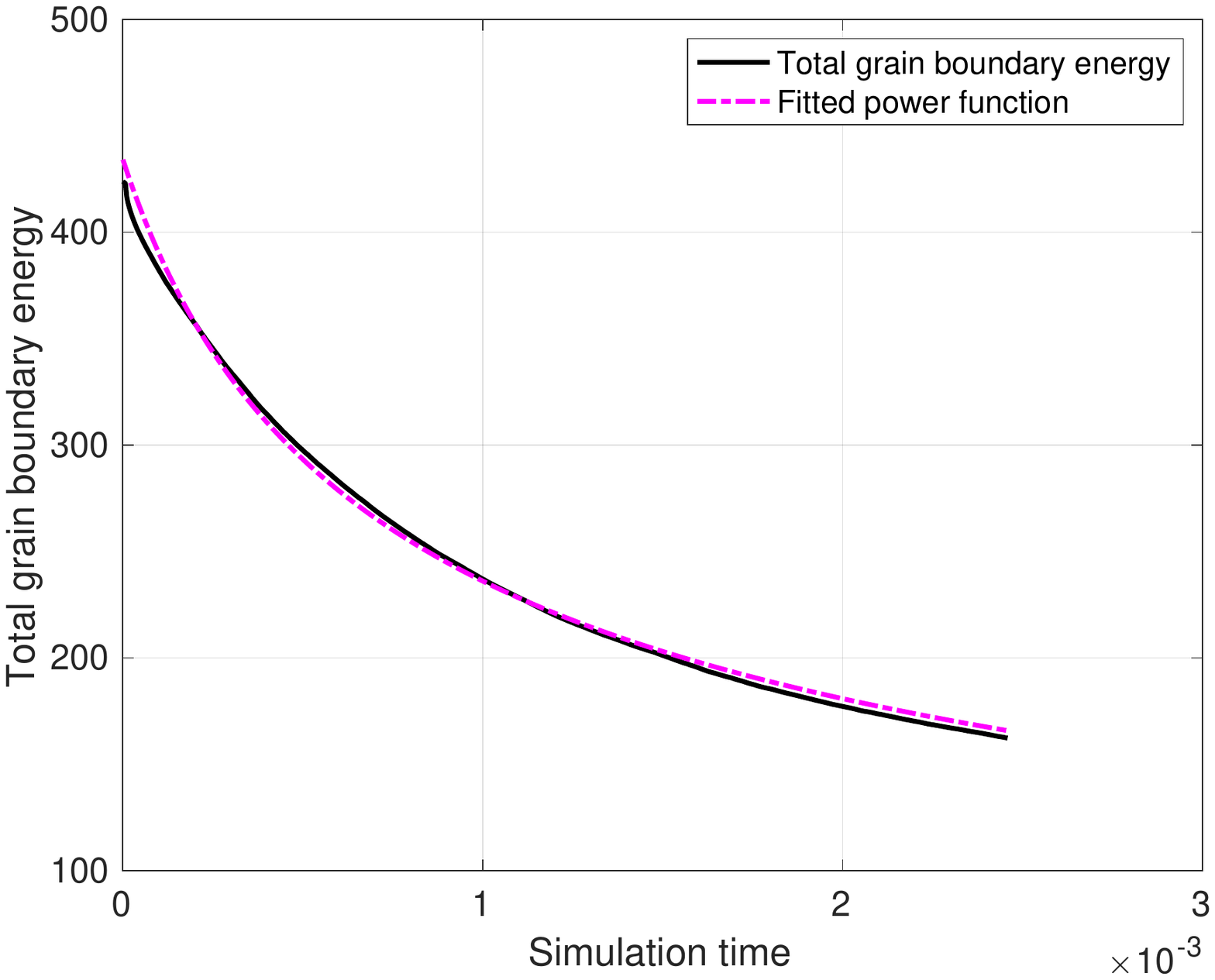}
\vspace{-2.cm}
%\end{tabular}
\caption{\footnotesize One run of $2$D trial with $10000$ initial
  grains: {\it (a) Left plot,} Total grain boundary energy plot
  (solid black) versus fitted  sum of exponential decaying functions
  $y(t)=248\exp(-182.8t)+166.8\exp(-1708t)$ (dashed blue); {\it (b) Right
    plot}, Total grain boundary energy plot
  (solid black) versus fitted power law decaying function
  $y_1(t)=434.6254(1.0+2388.1t)^{-0.5}$ (dashed magenta). Herring Condition is
imposed at the triple junctions $\eta\to \infty$
 and the misorientation parameter $\gamma=1000$.}\label{fig16a}
\end{figure}

\begin{figure}[hbtp]
\centering
%\begin{tabular}{cc}
%(a) & (b)\\
\vspace{-2.cm}
\includegraphics[width=3.1in]{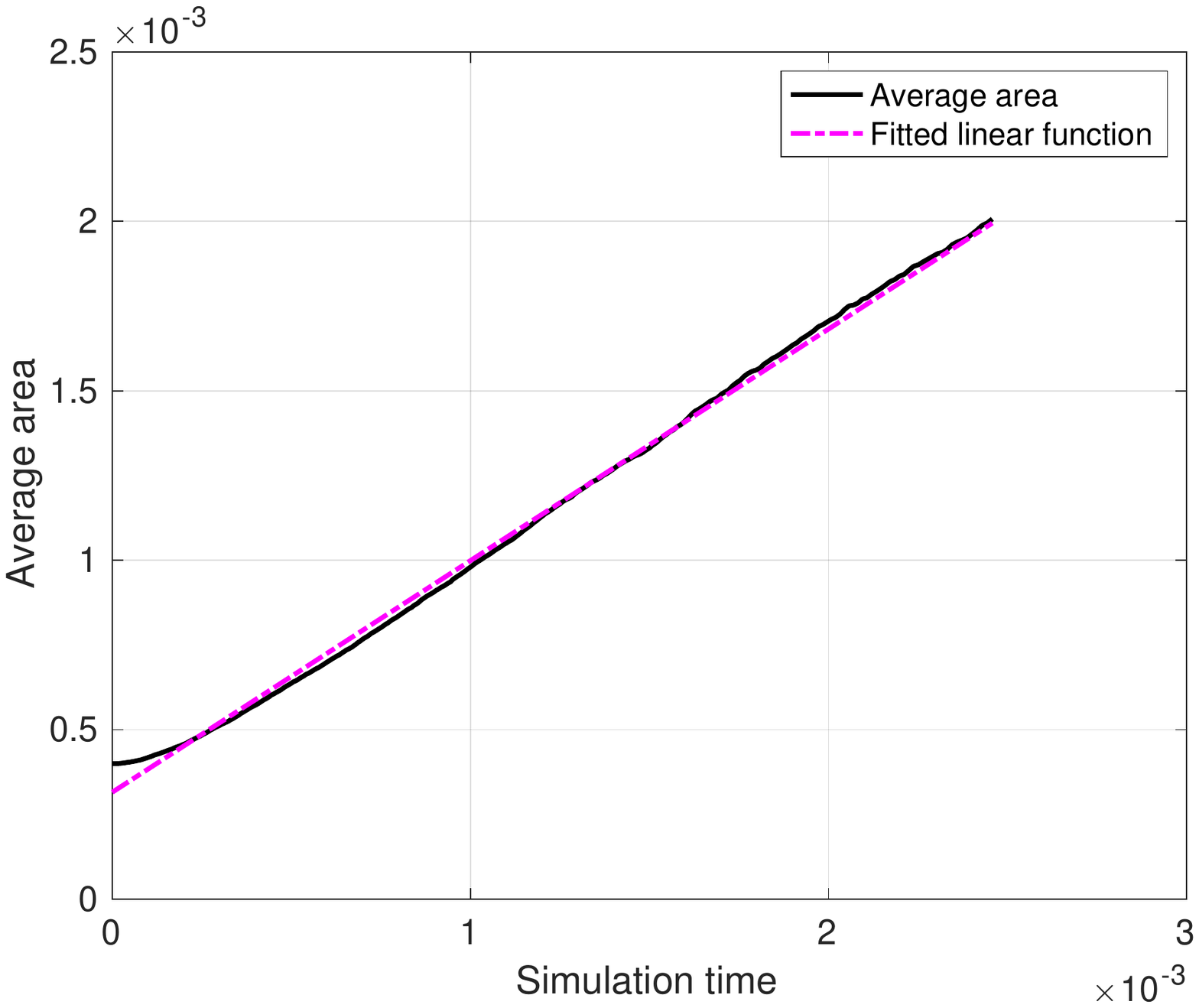}
\includegraphics[width=3.1in]{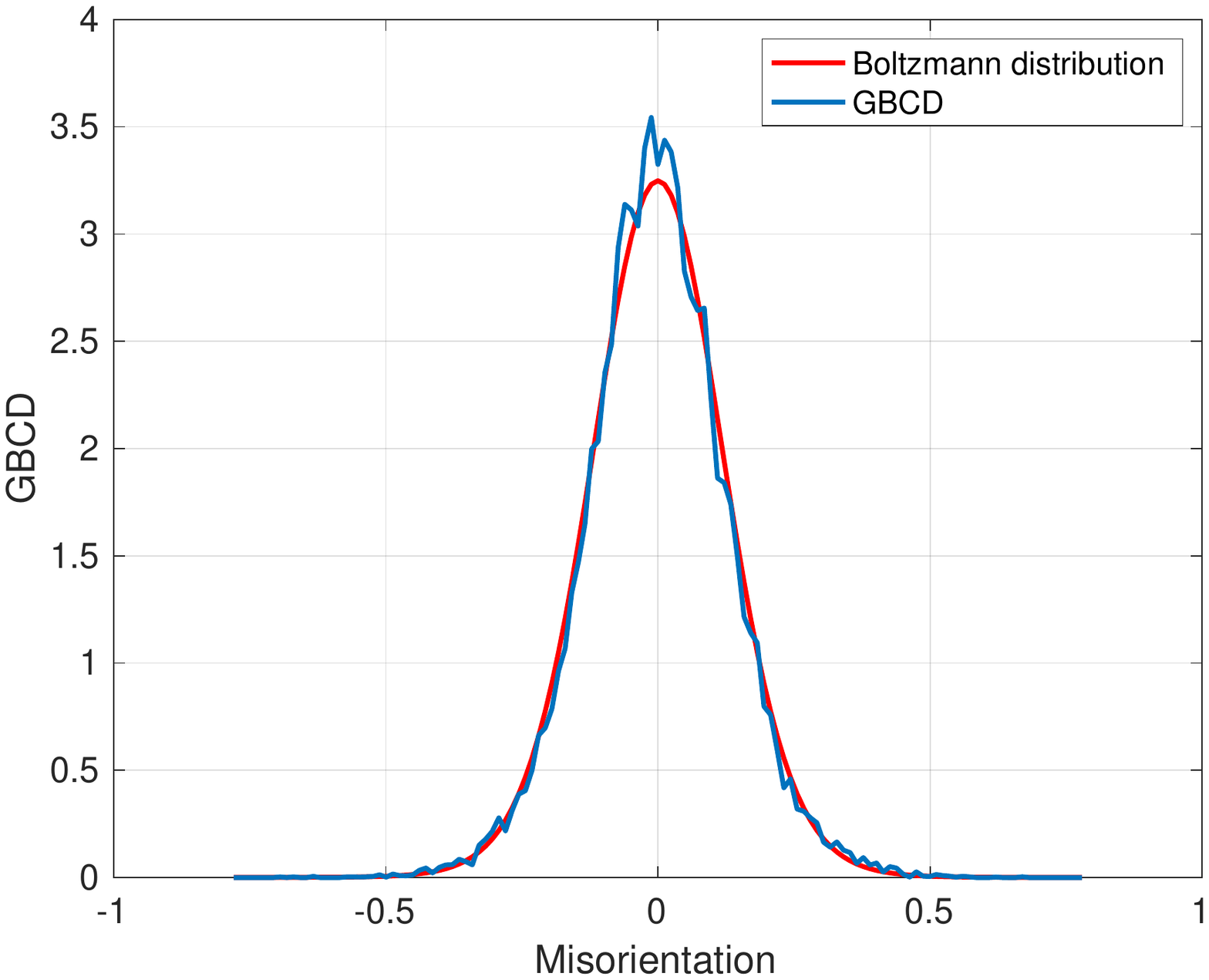}
\vspace{-2.cm}
%\end{tabular}
\caption{\footnotesize {\it (a) Left plot,} One run of $2$D trial with $10000$ initial
  grains: Growth of the average area of the
  grains (solid black) versus fitted  linear function
  $y(t)=0.6834t+0.0003155$ (dashed magenta).  
%The growth of the average area is
%consistent with the energy decay, Fig.~\ref{fig8a};
 {\it (b) Right plot},
  steady-state GBCD (blue curve) averaged over 3 runs of $2$D trials with $10000$ initial
  grains versus Boltzmann distribution with ``temperature''-
$D\approx 0.0283$ (red curve).  Herring Condition is
imposed at the triple junctions $\eta\to \infty$
 and the misorientation parameter $\gamma=1000$.}\label{fig17a}
\end{figure}
\par For the other series of tests, we impose Herring condition at the
triple junctions ($\eta \to \infty$), but we vary the misorientation
parameter $\gamma$, second equation of (\ref{eq:6.4}). We do not
observe as much effect on the energy decay or average area growth in this
case (probably due to the effect of the Herring condition at the triple junctions), but we observe the significant effect on the steady-state GBCD
and the diffusion coefficient/``temperature''-like parameter $D$,
see Figures \ref{fig12a}-\ref{fig17a}.  As concluded from our numerical
results, larger values of $\gamma$ give smaller diffusion
coefficient/''temperature''-like parameter $D$, and hence higher GBCD
peak near misorientation $0$. This is consistent with our theory that
basically, larger misorientation parameter $\gamma$ produces direct motion of misorientations
towards equilibrium state of zero misorientations, see Section
\ref{sec:1} and also \cite{Katya-Chun-Mzn}. Furthermore, from all numerical experiments with dynamic
misorientation and with different triple junction mobilities, we
observe that the steady-state GBCD is well-approximated by the Boltzmann
distribution for the grain boundary energy density see Figures
\ref{fig9a}, \ref{fig11a}, \ref{fig13a}, \ref{fig15a} and
\ref{fig17a} (right plots), which is similar to the
work in \cite{DK:BEEEKT,DK:gbphysrev,
  MR2772123, MR3729587,barmak_grain_2013}, but more detailed analysis
needs to be done for a system that undergoes critical events to
understand the relation between GBCD, ``temperature''-like/diffusion
parameter $D$, and different
relaxation time scales, as well as the effect of the time scales
on the dissipation mechanism and certain coarsening rates.
\begin{remark}
Note that, we performed $3$ runs for each numerical test
presented in this work. We report results of a single run for the energy decay
and the growth of the average area (the results from the other two
runs for each test were very similar to the presented ones),  and we
illustrate  averaged
over the $3$ runs the steady-state GBCD statistics. The curve-fitting
 for the energy and the average area plots was done using Matlab (\cite{Matlab}) toolbox cftool.
\end{remark}
 
\begin{remark}
\par Note, that the proposed model of dynamic orientations (\ref{eq:6.4}) (and, hence,
dynamic misorientations), or Langevin type
equation if critical events/disappearance events are taken into account) is reminiscent of the recently developed
theory for the grain boundary character distribution (GBCD)
\cite{DK:BEEEKT,DK:gbphysrev,MR2772123,MR3729587},  which suggests that the
evolution of the GBCD satisfies a Fokker-Planck equation. More
details will be presented in future studies.
\end{remark}

\section*{Acknowledgments}\label{sec:Ack}
The authors are grateful to David Kinderlehrer for the fruitful
discussions, inspiration and motivation of the work.  The authors are
also grateful to the anonymous referees for their valuable remarks and
questions, which led to significant improvement of the manuscript. Yekaterina
Epshteyn and Masashi Mizuno acknowledge partial support of Simons
Foundation Grant No. 415673, Yekaterina Epshteyn also acknowledges
partial support of NSF DMS-1905463, Masashi Mizuno
also acknowledges partial support of JSPS KAKENHI Grant No. 18K13446, Chun Liu acknowledges partial support of
NSF DMS-1759535 and NSF DMS-1759536.
Yekaterina Epshteyn would like to thank Nihon University for the hospitality
during her visit and Masashi Mizuno also would like to thank Penn State University, Illinois
Institute of Technology and the University of Utah for the hospitality
during his visits.

%
%
%\section{Final remark(TBD)}
%
%
\appendix

\section{Explicit form of the decay rate of the linearized problem
 \eqref{eq:5.7}}
\label{sec:A}

In this appendix, we give an explicit form of the constant
$\lambda_2$, which is a minimum eigenvalue of
\begin{equation*}
 L_{\vec{a}}
 =
 \sum_{j=1}^3
 \frac{1}{|\vec{b}_\infty^{(j)}|}
 \left(
  I
  -
  \frac{\vec{b}_\infty^{(j)}}{|\vec{b}_\infty^{(j)}|}
  \otimes
  \frac{\vec{b}_\infty^{(j)}}{|\vec{b}_\infty^{(j)}|}
 \right).
\end{equation*}
Since $L_{\vec{a}}$ is 2 dimensional matrix, it is enough to manipulate
the trace and the determinant of $L_{\vec{a}}$. The trace of
$L_{\vec{a}}$ is easily calculated as
\begin{equation*}
 \tr L_{\vec{a}}
  =
  \sum_{j=1}^3
  \frac{1}{|\vec{b}_\infty^{(j)}|}
  \left(
   2
   -
   \tr
   \frac{\vec{b}_\infty^{(j)}}{|\vec{b}_\infty^{(j)}|}
   \otimes
   \frac{\vec{b}_\infty^{(j)}}{|\vec{b}_\infty^{(j)}|}
  \right)
  =
  \sum_{j=1}^3
  \frac{1}{|\vec{b}_\infty^{(j)}|}.
\end{equation*}
Next we consider the determinant of $L_{\vec{a}}$. Denote
$\vec{b}_\infty^{(j)}=(b_{\infty,1}^{(j)},b_{\infty,2}^{(j)})$,
$b_k^{(j)}:=\frac{b_{\infty,k}^{(j)}}{b_\infty^{(j)}}$,
and $b_{\infty}^{(j)}=|\vec{b}_\infty^{(j)}|$. Then,
\[
 L_{\vec{a}}
 =
 \sum_{j=1}^3
 \begin{pmatrix}
  \frac{1}{b_\infty^{(j)}}
  \left(
  1-\left(b_1^{(j)}\right)^2
  \right)
  &
  -\frac{1}{b_\infty^{(j)}}
  b_1^{(j)}b_2^{(j)}  \\
  -\frac{1}{b_\infty^{(j)}}
  b_1^{(j)}b_2^{(j)}
  &
  \frac{1}{b_\infty^{(j)}}
  \left(
  1-\left(b_2^{(j)}\right)^2
  \right)
 \end{pmatrix}
\]
hence,
\begin{equation*}
 \begin{split}
  \det L_{\vec{a}}
  &=
  \left(
  \sum_{j=1}^3
  \frac{1}{b_\infty^{(j)}}
  \left(
  1-\left(b_{1}^{(j)}\right)^2
  \right)
  \right)
  \left(
  \sum_{j=1}^3
  \frac{1}{b_\infty^{(j)}}
  \left(
  1-\left(b_{2}^{(j)}\right)^2
  \right)
  \right)
  -
  \left(
  \sum_{j=1}^3
  \frac{1}{b_\infty^{(j)}}
  b_{1}^{(j)}
  b_{2}^{(j)}
  \right)^2 \\
  &=
  \sum_{j=1}^3
  \frac{1}{(b_\infty^{(j)})^2}
  \left(
  \left(
  1-\left(b_{1}^{(j)}\right)^2
  \right)
  \left(
  1-\left(b_{2}^{(j)}\right)^2
  \right)
  -
  \left(
  b_{1}^{(j)}
  b_{2}^{(j)}
  \right)^2
  \right) \\
  &\quad
  +
  \sum_{j\neq k}
  \frac{1}{b_\infty^{(j)}}
  \frac{1}{b_\infty^{(k)}}
  \left(
  \left(
  1-\left(b_{1}^{(j)}\right)^2
  \right)
  \left(
  1-\left(b_{2}^{(k)}\right)^2
  \right)
  -
  \left(
  b_{1}^{(j)}
  b_{2}^{(j)}
  \right)
  \left(
  b_{1}^{(k)}
  b_{2}^{(k)}
  \right)
  \right) \\
  &=
  \sum_{j\neq k}
  \frac{1}{b_\infty^{(j)}}
  \frac{1}{b_\infty^{(k)}}
  \left(
  \left(
  1-\left(b_{1}^{(j)}\right)^2
  \right)
  \left(
  1-\left(b_{2}^{(k)}\right)^2
  \right)
  -
  \left(
  b_{1}^{(j)}
  b_{2}^{(j)}
  \right)
  \left(
  b_{1}^{(k)}
  b_{2}^{(k)}
  \right)
  \right) \\
  &=
  \sum_{j<k}
  \frac{1}{b_\infty^{(j)}}
  \frac{1}{b_\infty^{(k)}}
  \biggl(
  1-\left(b_{1}^{(j)}\right)^2-\left(b_{2}^{(k)}\right)^2
  +\left(b_{1}^{(j)}\right)^2\left(b_{2}^{(k)}\right)^2 \\
  &\qquad
  +1-\left(b_{1}^{(k)}\right)^2-\left(b_{2}^{(j)}\right)^2
  +\left(b_{1}^{(k)}\right)^2\left(b_{2}^{(j)}\right)^2
  -2
  \left(
  b_{1}^{(j)}
  b_{2}^{(j)}
  b_{1}^{(k)}
  b_{2}^{(k)}
  \right)
  \biggr) \\
  &=
  \sum_{j<k}
  \frac{1}{b_\infty^{(j)}}
  \frac{1}{b_\infty^{(k)}}
  \left(
  b_1^{(j)}b_2^{(k)}-b_2^{(j)}b_1^{(k)}
  \right)^2
  =
  \sum_{j<k}
  \frac{1}{b_\infty^{(j)}}
  \frac{1}{b_\infty^{(k)}}
  \left(
  \frac{\vec{b}_\infty^{(j)}}{|\vec{b}_\infty^{(j)}|}
  \cdot
  R_{-\frac\pi2}\frac{\vec{b}_\infty^{(k)}}{|\vec{b}_\infty^{(k)}|}
  \right)^2,
 \end{split}
\end{equation*}
where $R_{-\frac\pi2}
=
\begin{pmatrix}
 0 & 1 \\
 -1 & 0
\end{pmatrix}$ is the $-\frac\pi2$ rotating matrix,  and note that
 $(b_1^{(j)})^2+(b_2^{(j)})^2=1$.  We also know that,
 $\left\{\frac{\vec{b}_\infty^{(k)}}{|\vec{b}_\infty^{(k)}|},\
 R_{-\frac\pi2}\frac{\vec{b}_\infty^{(k)}}{|\vec{b}_\infty^{(k)}|}\right\}$ is
 orthonormal basis on $\R^2$. Thus for $1\leq j,k\leq 3$, by Parseval's
 identity tells us
\begin{equation*}
  \left(
   \frac{\vec{b}_\infty^{(j)}}{|\vec{b}_\infty^{(j)}|} \cdot
   \frac{\vec{b}_\infty^{(k)}}{|\vec{b}_\infty^{(k)}|}
  \right)^2
  +
  \left(
   \frac{\vec{b}_\infty^{(j)}}{|\vec{b}_\infty^{(j)}|} \cdot
   R_{-\frac\pi2}\frac{\vec{b}_\infty^{(k)}}{|\vec{b}_\infty^{(k)}|}
  \right)^2=1.
\end{equation*}
Since,
\begin{equation*}
 \frac{\vec{b}_\infty^{(j)}}{|\vec{b}_\infty^{(j)}|}
  \cdot
  \frac{\vec{b}_\infty^{(k)}}{|\vec{b}_\infty^{(k)}|}
  =
  -\frac12+\frac32\delta_{jk},
\end{equation*}
we finally arrive, for $j\neq k$
\begin{equation*}
 \left(
  \frac{\vec{b}_\infty^{(j)}}{|\vec{b}_\infty^{(j)}|}
  \cdot
  R_{-\frac\pi2}\frac{\vec{b}_\infty^{(k)}}{|\vec{b}_\infty^{(k)}|}
 \right)^2
 =
 \frac{3}{4}.
\end{equation*}
Then $\lambda_2>0$ in Lemma \ref{lem:5.3} is explicitly given by,
\begin{equation}
\label{eq:A.1}
 \begin{split}
  \lambda_2
  &=
  \frac12
  \left(
  \tr L_{\vec{a}}
  -
  \sqrt{(\tr L_{\vec{a}})^2
  -4(\det L_{\vec{a}})}
  \right) \\
  &=
  \frac12
  \left(
  \sum_{j=1}^3\frac{1}{|\vec{b}_\infty^{(j)}|}
  -
  \sqrt{
  \left(\sum_{j=1}^3\frac{1}{|\vec{b}_\infty^{(j)}|}\right)^2
  -3\left(
  \sum_{j<k}
  \frac{1}{|\vec{b}_\infty^{(j)}|}
  \frac{1}{|\vec{b}_\infty^{(k)}|}
  \right)
  }
  \right) \\
  &=
  \frac12
  \left(
  \sum_{j=1}^3\frac{1}{|\vec{b}_\infty^{(j)}|}
  -
  \sqrt{
  \frac12
  \sum_{j<k}
  \left(
  \frac{1}{|\vec{b}_\infty^{(j)}|}
  -
  \frac{1}{|\vec{b}_\infty^{(k)}|}
  \right)^2
  }
  \right).
 \end{split}
\end{equation}
%
%\section*{Acknowledgments}


\begin{thebibliography}{99}
 \bibitem{MR3316603} H.~Abels, H.~Garcke,L.~M\"uller, \emph{Stability of
	 spherical caps under the volume-preserving mean curvature flow
	 with line tension}, Nonlinear Anal. \textbf{117}(2015), 8--37.

\bibitem{MR3729587}
Patrick Bardsley, Katayun Barmak, Eva Eggeling, Yekaterina Epshteyn, David
  Kinderlehrer, and Shlomo Ta'asan.
\newblock {\em Towards a gradient flow for microstructure.}
\newblock  Atti Accad. Naz. Lincei Rend. Lincei Mat. Appl.,
  28(4):777--805, 2017.

\bibitem{DK:BEEEKT}
K.~Barmak, E.~Eggeling, M.~Emelianenko, Y.~Epshteyn, D.~Kinderlehrer, and
  S.~Ta'asan.
\newblock  {\em Geometric growth and character development in large metastable
  networks.}
\newblock { Rend. Mat. Appl. (7)}, 29(1):65--81, 2009.

\bibitem{DK:gbphysrev}
K.~Barmak, E.~Eggeling, M.~Emelianenko, Y.~Epshteyn, D.~Kinderlehrer, R.~Sharp,
  and S.~Ta'asan.
\newblock  {\em Critical events, entropy, and the grain boundary character
  distribution.}
\newblock { Phys. Rev. B}, 83:134117, Apr 2011.

\bibitem{MR2772123}
	  K.~Barmak, E.~Eggeling,M.~Emelianenko,
	  Y.~Epshteyn, D.~Kinderlehrer,
	  R.~Sharp, S.~Ta'asan,
	  \emph{An entropy based theory of the grain boundary character
	  distribution},
	  Discrete Contin. Dyn. Syst. \textbf{30}(2011),
	  427--454.

\bibitem{barmak_grain_2013}
Katayun Barmak, Eva Eggeling, David Kinderlehrer, Richard Sharp, Shlomo
  Ta'asan, Anthony~D. Rollett, and Kevin~R. Coffey.
\newblock  {\em Grain {Growth} and the {Puzzle} of its {Stagnation} in {Thin}
  {Films}: {The} {Curious} {Tale} of a {Tail} and an {Ear}.}
\newblock  Prog. Mater. Sci., 58:987--1055, 2013.

 \bibitem{MR1677397}
	 V.~Boltyanski, H.~Martini, V.~Soltan,
	 \emph{Geometric methods and optimization problems},
	 Combinatorial Optimization, vol.~4, Kluwer Academic
	 Publishers, Dordrecht, 1999.

 \bibitem{MR0485012}
	 K.~ A.~Brakke,
	 \emph{The motion of a surface by its mean curvature},
	 Princeton University Press,
	 1978.

 \bibitem{MR1240580}
	 L.~Bronsard, F.~Reitich,
	 \emph{On three-phase boundary motion and the singular limit of
	 a vector-valued {G}inzburg-{L}andau equation},
	 Arch. Rational Mech. Anal. \textbf{124}(1993),
	 355--379.

 \bibitem{MR1100211}
	 Y.~G.~Chen, Y.~Giga, S.~Goto,
	 \emph{Uniqueness and existence of viscosity solutions of
	 generalized mean curvature flow equations},
	 J. Differential Geom. \textbf{33}(1991),
	 749--786.

 \bibitem{MR2024995}
	 K.~Ecker,
	 \emph{Regularity theory for mean curvature flow},
	 Progress in Nonlinear Differential Equations and
	 their Applications \textbf{57},
	 Birkh\"auser, 2004.

\bibitem{MR3787390}
Matt Elsey and Selim Esedo\={g}lu.
\newblock  {\em Threshold dynamics for anisotropic surface energies.}
\newblock { Math. Comp.}, 87(312):1721--1756, 2018.

\bibitem{MR2573343}
Matt Elsey, Selim Esedo{\=g}lu, and Peter Smereka.
\newblock  {\em Diffusion generated motion for grain growth in two and three
  dimensions.}
\newblock { J. Comput. Phys.}, 228(21):8015--8033, 2009.

\bibitem{MR2748098}
Matt Elsey, Selim Esedo{\=g}lu, and Peter Smereka.
\newblock  {\em Large-scale simulation of normal grain growth via diffusion-generated
  motion.}
\newblock { Proc. R. Soc. Lond. Ser. A Math. Phys. Eng. Sci.},
  467(2126):381--401, 2011.

 \bibitem{Katya-Chun-Mzn}
	 Y.~Epshteyn, C.~Liu, M.~Mizuno,
	 {\em  Motion of grain boundaries with dynamic lattice misorientations
	 and with triple junctions drag,} submitted, 2019, https://arxiv.org/abs/1903.11512.

 \bibitem{MR1100206}
	 L.~C.~Evans, J.~Spruck,
	 \emph{Motion of level sets by mean curvature. {I}},
	 J. Differential Geom. \textbf{33}(1991),
	 635--681.

 \bibitem{MR2561958}
	 H.~Garcke, Y.~Kohsaka, D.~\v{S}ev\v{c}ovi\v{c},
	 \emph{Nonlinear stability of stationary solutions for
	 curvature flow with triple function}
	 Hokkaido Math. J. \textbf{38}(2009), 721--769.


 \bibitem{doi:10.1007-978-3-642-59938-5_2}
	 C. Herring,
	 \emph{Surface tension as a motivation for sintering},
	 Fundamental Contributions to the Continuum Theory of Evolving
	 Phase Interfaces in Solids,
	 Springer, 1999,33--69.

 \bibitem{MR3612327}
	 L.~Kim, Y.~Tonegawa,
	 \emph{{On the mean curvature flow of grain boundaries}},
	 Ann. Inst. Fourier (Grenoble) \textbf{67}(2017),
	 43--142.

 \bibitem{MR1833000}
	 D.~Kinderlehrer, C.~Liu,
	 \emph{Evolution of grain boundaries},
	 Math. Models Methods Appl. Sci. \textbf{11}(2001),
	 713--729.

 \bibitem{MR2272185}
	 D.~Kinderlehrer, I.~Livshits, S.~Ta'asan,
	 \emph{A variational approach to modeling and simulation of
	 grain growth},
	 SIAM J. Sci. Comput. \textbf{28}(2006),
	 1694--1715.

\bibitem{BobKohn}
Robert~V. Kohn.
\newblock  {\em Irreversibility and the statistics of grain boundaries.}
\newblock { Physics, 4:33}, Apr 2011.

 \bibitem{MR3495423}
	 A.~Magni, C.~Mantegazza, M.~Novaga,
	 \emph{Motion by curvature of planar networks, {II}}
	 Ann. Sc. Norm. Super. Pisa Cl. Sci. (5) \textbf{15}(2016),
	 117--144.

 \bibitem{MR2815949}
	 C.~Mantegazza,
	 \emph{Lecture notes on mean curvature flow},
	 Progress in Mathematics, \textbf{290},
	 Birkh\"auser, 2011.

	 
 \bibitem{MR2076269}
	 C.~Mantegazza,
	 \emph{Evolution by curvature of networks of curves in the
	 plane}, in
	 Variational problems in {R}iemannian geometry,
	 Progr. Nonlinear Differential Equations Appl.
	 \textbf{59}(2014),
	 95--109.

 \bibitem{MR3565976}
	 C.~Mantegazza, M.~Novaga, A.~Pluda,
	 \emph{Motion by curvature of networks with two triple
	 junctions},
	 Geom. Flows \textbf{2}(2016),
	 18--48.

 \bibitem{arXiv:1611.08254}
	 C.~Mantegazza, M.~Novaga, A.~Pluda, F.~Schulze
	 \emph{Evolution of networks with multiple junctions},
	 preprint, arXiv:1611.08254.


 \bibitem{MR2075985}
	 C.~Mantegazza, M.~Novaga, V.~M.~Tortorelli,
	 \emph{Motion by curvature of planar networks},
	 Ann. Sc. Norm. Super. Pisa Cl. Sci. (5) \textbf{3}(2004),
	 235--324.

\bibitem{Matlab}
\emph{MATLAB. version 9.4.0 (R2018a).} The MathWorks Inc., Natick,
  Massachusetts, 2018.

 \bibitem{doi:10.1063-1.1722511}
	 W.~W.~Mullins,
	 \emph{Two-Dimensional Motion of Idealized Grain Boundaries},
	 J. Appl. Phy. \text{27}(1956),
	 900--904.


 \bibitem{doi:10.1063-1.1722742}
	 W.~W.~Mullins,
	 \emph{Theory of thermal grooving},
	 J. Appl. Phy. \text{28}(1957),
	 333--339.

\bibitem{Thomas2019PNAS}
	 S.~L.~Thomas, C.~Wei, J.~Han, Y.~Xiang, and
	 D.~J.~ Srolovitz,
	 \emph{Disconnection description of triple-junction motion},
	 Proc. Natl. Acad. Sci. USA
	 \textbf{116}(2019),
	 8756--8765.

 \bibitem{doi:10.1023-A:1008781611991}
	 M.~Upmanyu, D.~J.~Srolovitz, L.~S.~Shvindlerman, G.~Gottstein,
	 \emph{Triple junction mobility: A molecular dynamics study}
	 Interface Science \textbf{7}(1999),
	 307--319.

 \bibitem{doi:10.1016-S1359-6454(01)00446-3}
	 M.~Upmanyu, D.~J.~Srolovitz, L.~S.~Shvindlerman, G.~Gottstein,
	 \emph{Molecular dynamics simulation of triple junction
	 migration},
	 Acta materialia \textbf{50}(2002),
	 1405--1420.

\bibitem{Zhang2017PhysRevLett}
	 L.~Zhang, J.~Han, Y.~Xiang, and D.~J.~Srolovitz, 
	 \emph{Equation of Motion for a Grain Boundary},
	 Phys. Rev. Lett. \textbf{119}(2017), 
	 246101.

 \bibitem{Zhang2018JMPS}
	 L.~Zhang and Y.~Xiang,
	 \emph{Motion of grain boundaries incorporating dislocation
	 structure},
	 J. Mech. Phys. Solids
	 \textbf{117}(2018), 157--178.

\end{thebibliography}
\end{document}